\documentclass[a4paper,11pt,onecolumn,twoside]{article}
\usepackage{mathrsfs}
\usepackage{mathrsfs}
\usepackage{amsfonts}
\usepackage{fancyhdr}
\usepackage{amsmath,amsfonts,amssymb,graphicx,amsthm}
\allowdisplaybreaks
\usepackage{amsmath,latexsym,amsfonts,indentfirst,amsthm,amsxtra,amssymb,bm}

\numberwithin{equation}{section}

 \newtheorem{lemma}{Lemma}[section]
 
 \newtheorem{theorem}{Theorem}[section]
 \newtheorem{corollary}[lemma]{Corollary}
 \theoremstyle{remark}
 \newtheorem{remark}{Remark}[section]

\begin{document}

\title{\bf Global Spherical Symmetric Flows for a Viscous Radiative and Reactive Gas in an Exterior Domain with Large Initial Data}
\author{{\bf Yongkai Liao}\footnote{Email address: yongkai.liao@whu.edu.cn},\quad {\bf Tao Wang}\thanks{Email address: tao.wang@whu.edu.cn}\quad and \quad {\bf Huijiang Zhao}\thanks{Email address: hhjjzhao@whu.edu.cn. Telephone number: 0086-27-68754675}\\[2mm]
School of Mathematics and Statistics, Wuhan University, Wuhan 430072, China\\
and\\
Computational Science Hubei Key Laboratory, Wuhan University, Wuhan 430072, China}
\date{}

\maketitle

\begin{abstract}
In this paper, we study the global existence, uniqueness and large-time behavior of spherically symmetric solution of a viscous radiative and reactive gas in an unbounded domain exterior to the unit sphere in $\mathbb{R}^{n}$ for $n\geq 2$. The key point in the analysis is to deduce certain uniform estimates on the solutions, especially on the uniform positive lower and upper bounds on the specific volume and the temperature.
\bigbreak
\end{abstract}


\section{Introduction}
In this paper, we're concerned with the global existence, uniqueness and large-time behavior of spherically symmetric solution to a model for the combustion of compressible radiative and reactive gas in an unbounded domain $\Omega$ exterior to the unit sphere $B_1(O)=\{\xi\in\mathbb{R}^{n}: |\xi|<1\}\subset\mathbb{R}^{n}$, i.e. $\Omega=\mathbb{R}^n\backslash\overline{B}_1(O)$, for $n\geq 2$. The model consists of equations corresponding to the conservation laws of the mass, the momentum and the energy coupling with the reaction-diffusion equation which, in the Eulerian coordinates, can be written as (cf. \cite{Ducomet-Zlotnik-ARMA-2005}, \cite{Ducomet-Zlotnik-NonliAnal-2005}, \cite{Liao-Zhao-CMS-2017}, \cite{Umehara-Tani-JDE-2007})
\begin{eqnarray}\label{a1}
    \rho_{t}+\textmd{div}\left(\rho \mathbf{u}\right)&=&0,\nonumber\\
    \left(\rho\mathbf{u}\right)_{t}+\textmd{div}\left(\rho\mathbf{u}\otimes\mathbf{u}\right) +\nabla P&=&\textmd{div}\mathbf{S},\\
    \left(\rho e\right)_{t}+\textmd{div}\left(\rho e\mathbf{u}\right)+\textmd{div}\mathbf{q}&=&\mathbf{S}:\nabla\mathbf{u} -P\textmd{div}\mathbf{u}+\lambda\phi \rho z,\nonumber\\
    \left(\rho z\right)_{t}+\textmd{div}\left(\rho\mathbf{u}z\right)-\textmd{div} F&=&-\phi \rho z.\nonumber
\end{eqnarray}
Here the primary dependent variables are the density $\rho(t,\xi)$, the absolute temperature $\theta(t,\xi)$, the reactant mass fraction $z(t,\xi)$ and the velocity $\mathbf{u}(t,\xi)=\left(u_{1}(t,\xi),\ldots, u_{n}(t,\xi)\right)$, respectively, with
$(t,\xi)\in [0,+\infty)\times \Omega$ being the time variable and the Eulerian space variable, respectively. The internal energy $e=e\left(\rho, \theta\right)$, the pressure $P=P\left(\rho, \theta\right)$ and the heat flux $\mathbf{q}=\mathbf{q}\left(\rho, \theta\right)$ are functions of the density $\rho$ and the absolute temperature $\theta$. The viscous stress tensor $\mathbf{S}$ characterizes the measure of resistance of the fluid and $\rho z=\rho(t,\xi) z\left(t,\xi\right)$ represents the density of the reactant. The positive constant $\lambda$ stands for the difference in heat between the reactant and the product. The reaction function $\phi=\phi\left(\theta\right)$ is defined by the first-order Arrhenius law (cf. \cite{Ducomet-Zlotnik-NonliAnal-2005}, \cite{Liao-Zhao-CMS-2017}, \cite{Umehara-Tani-JDE-2007})
\begin{eqnarray}\label{a2}
 \phi\left(\theta\right)=K\theta^{\beta}\exp\left(-\frac{A}{\theta}\right),
\end{eqnarray}
where positive constants $K$ and $A$ are the coefficients of the rate of the reactant and the activation energy, respectively, and $\beta$ is a non-negative number. The species diffusion velocity $F$ is assumed to satisfy the Fickian law (cf. \cite{Donatelli-Trivisa-CMP-2006})
\begin{eqnarray}\label{a3}
F=D\nabla z,\nonumber
\end{eqnarray}
where $D=d\rho$ is the reactant flux diffusion coefficient and the positive constant $d$ stands for the species diffusion in the reaction. In accordance with the general principle on Newtonian fluids, the admissible form of $\mathbf{S}$ reads as
\begin{eqnarray}\label{a4}
\mathbf{S}=\mu \left(\nabla\mathbf{u}+\nabla\mathbf{u}^{T}\right)+\lambda_{1}\textmd{div}\mathbf{u}\ \mathbf{I},\nonumber
\end{eqnarray}
where $\mu$ and $\lambda_{1}$ are called the viscosity coefficients satisfying
\begin{eqnarray}\label{a5}
\mu>0,\quad n\lambda_{1}+2\mu>0.
\end{eqnarray}
Accordingly, the dissipative function $\mathbf{S}:\nabla\mathbf{u}$ stands for a real dissipation of the mechanical energy into heat and can be written as
\begin{eqnarray}\label{a6}
\mathbf{S}:\nabla\mathbf{u}=\frac{\mu}{2}\sum_{i, j=1}^n\left(\frac{\partial{u}_{i}}{\partial \xi_{j}}+\frac{\partial{u}_{j}}{\partial \xi_{i}}\right)^{2}+\lambda_{1}\left|\textmd{div}\mathbf{u}\right|^{2}.\nonumber
\end{eqnarray}
The thermo-radiative flux $\mathbf{q}$ satisfies the Fourier law
\begin{eqnarray}\label{a7}
\mathbf{q}=\mathbf{q}\left(\rho,\theta\right)=-\kappa\left(\rho,\theta\right)\nabla\theta,\quad \kappa\left(\rho,\theta\right)=\kappa_{1}+\kappa_{2}\frac{\theta^{b}}{\rho},
\end{eqnarray}
where $\kappa_{1}$, $\kappa_{2}$ and $b$ are the positive constants (cf. \cite{Umehara-Tani-JDE-2007}).

We treat the radiation as a continuous field and consider both the wave and photonic effect. Assume that the high-temperature radiation is at thermal equilibrium with the fluid. Then the pressure $P$ and the internal energy $e$ consist of a linear term in $\theta$ corresponding to the perfect polytropic contribution and a fourth-order radiative part due to the Stefan-Boltzmann radiative law (cf. \cite{Ducomet-MMAS-1999}, \cite{Mihalas-Mihalas-1984}, \cite{Umehara-Tani-JDE-2007})
\begin{eqnarray}\label{a8}
P=P\left(\rho,\theta\right)=R\rho\theta+\frac{a}{3}\theta^{4},\quad e=e\left(\rho,\theta\right)=c_{v}\theta+a\frac{\theta^{4}}{\rho},
\end{eqnarray}
where the positive constants $R$, $c_{v}$ and $a$ are the perfect gas constant, the specific heat and the Stefan-Boltzmann constant, respectively.

The system \eqref{a1} is supplemented with the initial and boundary conditions
\begin{eqnarray}\label{a9}
\rho\left(0,\xi\right)=\rho_{0}\left(\xi\right),\quad\mathbf{u}\left(0,\xi\right)=\mathbf{u}_{0}\left(\xi\right),
\quad\theta\left(0,\xi\right)=\theta_{0}\left(\xi\right),\quad z\left(0,\xi\right)=z_{0}\left(\xi\right),\quad \forall\xi\in\Omega
\end{eqnarray}
and
\begin{eqnarray}\label{a10}
\mathbf{u}\left(t,\xi\right)=0,\quad \frac{\partial\theta\left(t,\xi\right)}{\partial{\bf n}}=0,
\quad \frac{\partial z\left(t,\xi\right)}{\partial{\bf n}}=0,\quad \forall (t,\xi)\in [0,\infty)\times \partial B_1(O).
\end{eqnarray}
Here for each $\xi\in\partial\Omega$, ${\bf n}(\xi)=-\frac{\xi}{|\xi|}$ denotes the outer normal vector of $\Omega$ at $\xi\in\partial\Omega$.

If the initial data $\left(\rho_{0}\left(\xi\right), \mathbf{u}_{0}\left(\xi\right), \theta_{0}\left(\xi\right), z_{0}\left(\xi\right)\right)$ is assumed to be spherically symmetric, i.e.,
\begin{eqnarray}\label{a11}
\rho_{0}\left(\xi\right)=\hat{\rho}_{0}\left(r\right),\quad \mathbf{u}_{0}\left(\xi\right)=\frac{\xi}{r}\hat{{u}}_{0}\left(r\right),
\quad\theta_{0}\left(\xi\right)=\hat{\theta}_{0}\left(r\right),\quad z_{0}\left(\xi\right)=\hat{z}_{0}\left(r\right),\quad r=|\xi|\geq 1,
\end{eqnarray}
then we can deduce that the corresponding solution $(\rho(t,\xi), {\bf u}(t,\xi), \theta(t,\xi))$ to the initial-boundary value problem \eqref{a1}, \eqref{a9}, \eqref{a10} is also spherically symmetric, i.e.,
\begin{eqnarray*}
\rho\left(t,\xi\right)=\hat{\rho}\left(t,r\right),\quad \mathbf{u}\left(t,\xi\right)=\frac{\xi}{r}\hat{{u}}\left(t,r\right),
\quad\theta\left(t,\xi\right)=\hat{\theta}\left(t,r\right),\quad z\left(t,\xi\right)=\hat{z}\left(t,r\right),\quad r=|\xi|\geq 1.
\end{eqnarray*}
Thus one can obtain from the system \eqref{a1} by direct calculations and by ignoring the symbol ``$\hat{}$ " that
\begin{eqnarray}\label{a12}
    \rho_{t}+\left(\rho u\right)_{r}+\frac{n-1}{r}\rho u&=&0,\nonumber\\
    \left(\rho u\right)_{t}+\left(\rho u^{2}\right)_{r}+\frac{n-1}{r}\rho u^{2}+P_{r}&=&\alpha\left(u_{r}+\frac{n-1}{r}u\right)_{r},\\
    \left(\rho e\right)_{t}+\left(\rho eu\right)_{r}+\frac{n-1}{r}\rho eu-\kappa_{r}\theta_{r}-\kappa\left(\theta_{rr}+\frac{n-1}{r}\theta_{r}\right)&=&2\mu\left(u^{2}_{r}+\frac{n-1}{r^{2}}u^{2}\right)
    +\lambda_{1}\left(u_{r}+\frac{n-1}{r}u\right)^{2}\nonumber\\
    &&-P\left(u_{r}+\frac{n-1}{r}u\right)+\lambda\phi\rho z,\nonumber\\
    \left(\rho z\right)_{t}+\left(\rho uz\right)_{r}+\frac{n-1}{r}\rho uz&=&-\phi \rho z+Dz_{rr}+\frac{n-1}{r}Dz_{r}+D_{r}z_{r},\nonumber
\end{eqnarray}
where $\alpha=2\mu+\lambda_{1}>0$, $r\in \left(1,\infty\right)$, $t>0$, the initial and boundary conditions \eqref{a9}-\eqref{a10} become
\begin{eqnarray}\label{a13}
\rho\left(0,r\right)=\rho_{0}\left(r\right),\quad u\left(0,r\right)=u_{0}\left(r\right),
\quad\theta\left(0,r\right)=\theta_{0}\left(r\right),\quad z\left(0,r\right)=z_{0}\left(r\right),\quad r\geq 1
\end{eqnarray}
and
\begin{eqnarray}\label{a14}
u\left(t, 1\right)=0,\quad \frac{\partial\theta\left(t, 1\right)}{\partial r}=0,\quad \frac{\partial z\left(t, 1\right)}{\partial r}=0,\quad t\geq 0.
\end{eqnarray}

For the convenience of our analysis, we convert the system \eqref{a12} from the Eulerian coordinates $\left(t, r\right)$ into that in Lagrangian coordinates $\left(t, x\right)$. For this purpose, if we define
\begin{eqnarray}\label{a15}
r\left(t,x\right)=r_{0}\left(x\right)+\int_{0}^{t}u\left(s,r\left(s,x\right)\right)ds
\end{eqnarray}
with
\begin{eqnarray}\label{a16}
\int_{1}^{r_{0}(x)}y^{n-1}\rho_{0}(y)dy=x,
\end{eqnarray}
then by using \eqref{a15}, \eqref{a16}, $\eqref{a12}_{1}$ and the boundary condition $u\left(t, 1\right)=0$, we have for $t\geq 0$ that
\begin{eqnarray}\label{a17}
\int_{1}^{r(t,x)}y^{n-1}\rho\left(t,y\right)dy=\int_{1}^{r_{0}(x)}y^{n-1}\rho_{0}(y)dy=x.
\end{eqnarray}

From \eqref{a17}, it is esy to see that $r=1$ if $x=0$ and $r\rightarrow\infty$ if $x\rightarrow\infty$, as long as $\rho>0$ for all $\left(t,y\right)\in [0,\infty)\times[0,\infty)$. Moreover, one can deduce from \eqref{a15} and \eqref{a17} that
\begin{eqnarray}\label{a18}
\frac{\partial r\left(t,x\right)}{\partial t} &=&u\left(t,r\left(t,x\right)\right),\nonumber\\  \frac{\partial r\left(t,x\right)}{\partial x} &=&\frac{1}{r^{n-1}\left(t,x\right)\rho\left(t,r\left(t,x\right)\right)}.
\end{eqnarray}

If we introduce
\begin{eqnarray}\label{a19}
\tilde{v}\left(t,x\right)&=:&\frac{1}{\rho\left(t,r\left(t,x\right)\right)},\nonumber\\ \tilde{u}\left(t,x\right)&=:&u\left(t,r\left(t,x\right)\right),\\
\; \tilde{\theta}\left(t,x\right)&=:&\theta\left(t,r\left(t,x\right)\right),\nonumber\\ \tilde{z}\left(t,x\right)&=:&z\left(t,r\left(t,x\right)\right),\nonumber
\end{eqnarray}
and express \eqref{a12} in terms of $\left(\tilde{v},\tilde{u},\tilde{\theta},\tilde{z}\right)$ (denote still by $\left(v, u, \theta, z\right)$ below) in variables $\left(t, x\right)$
\begin{eqnarray}\label{a20}
    v_t&=&\left(r^{n-1}u\right)_{x},\nonumber\\
    u_t&=&r^{n-1}\left(\frac{\alpha\left(r^{n-1}u\right)_{x}}{v}-P\right)_x,\\
    e_t&=&\left(\frac{r^{2n-2}\kappa\theta_{x}}{v}\right)_x+\left(\frac{\alpha\left(r^{n-1}u\right)_{x}}{v}-P\right)\left(r^{n-1}u\right)_{x}
    -2\mu\left(n-1\right)\left(r^{n-2}u^{2}\right)_{x}+\lambda\phi z,\nonumber\\
    z_{t}&=&\left(\frac{dr^{2n-2}z_x}{v^{2}}\right)_{x}-\phi z,\nonumber
\end{eqnarray}
where $(t,x)\in[0,\infty)\times[0,\infty)$.

The corresponding initial data, the boundary conditions and the far field behavior are
\begin{eqnarray}\label{a21}
 \left(v\left(0,x\right),u\left(0,x\right),\theta\left(0,x\right), z\left(0,x\right)\right)=\left(v_0\left(x\right),u_0\left(x\right),\theta_{0}\left(x\right), z_{0}\left(x\right)\right), \quad x\in[0,\infty),
\end{eqnarray}
and
\begin{eqnarray}\label{a22}
&&u\left(t,0\right)=0,\quad \frac{\partial \theta\left(t,0\right)}{\partial x}=0,\quad  \frac{\partial z\left(t,0\right)}{\partial x}=0,\nonumber\\
&&\lim_{x\rightarrow+\infty}\left(v\left(t, x\right),u\left(t, x\right),\theta\left(t, x\right), z\left(t, x\right)\right)=(1,0,1,0),
\end{eqnarray}
respectively.

In view of \eqref{a19}, we can deduce from \eqref{a15} and \eqref{a18} that
\begin{eqnarray}\label{a23}
r\left(t,x\right)&=&r_{0}\left(x\right)+\int_{0}^{t}u\left(s,x\right)ds,\nonumber\\ \frac{\partial r(t,x)}{\partial t}&=&u(t,x),\\
\frac{\partial r(t,x)}{\partial x}&=&r^{1-n}(t,x)v(t,x).\nonumber
\end{eqnarray}
Integrating the last equality in \eqref{a23} yields
\begin{eqnarray}\label{a24}
r^{n}\left(t,x\right)=1+n\int_{0}^{x}v\left(t,y\right)dy.
 \end{eqnarray}
Furthermore, it follows from \cite{Jiang-CMP-1996} that
\begin{eqnarray}\label{a25}
r\left(t,x\right)\geq r\left(t, 0\right)=1,\quad \left(t, x\right)\in [0,\infty)\times[0,\infty).
 \end{eqnarray}


Before stating our main results, let us review some related results in the literature. The mathematical study of radiation hydrodynamics has attracted a lot of interest recently (a complete literature in this direction is beyond the scope of this paper; however, we want to mention \cite{Donatelli-Trivisa-CMP-2006, Ducomet-MMAS-1999, Ducomet-ARMA-2004, Ducomet-Feireisl-CMP-2006, Ducomet-Zlotnik-ARMA-2005, Ducomet-Zlotnik-NonliAnal-2005, Jiang-ZHeng-JMP-2012, Jiang-ZHeng-ZAMP-2014, Qin-Hu-Wang-Huang-Ma-JMAA-2013, Umehara-Tani-JDE-2007, Umehara-Tani-PJA-2008} and references cited therein). To go directly to the main points of the present paper, in what follows we only review some former results which are closely related to our main results:

For the one-dimensional case, Documet \cite{Ducomet-MMAS-1999} established the global existence and exponential decay in $H^{1}([0,1])$ of solutions to the initial-boundary value problem of the one-dimensional model in the bounded interval $(0,1)$ for $b\geq 4$ with the following initial-boundary conditions
\begin{eqnarray*}
(v(0,x),u(0,x),\theta(0,x),z(0,x))&=&(v_0(x), u_0(x),\theta_0(x),z_0(x)),\quad x\in(0,1),\\
u(t,0)&=&u(t,1)=0,\quad \forall t>0,\\
\frac{\partial\theta(t,0)}{\partial x}&=&\frac{\partial \theta(t,1)}{\partial x}=0,\quad \forall t>0,\\
\frac{\partial z(t,0)}{\partial x}&=&\frac{\partial z(t,1)}{\partial x}=0,\quad \forall t>0.
\end{eqnarray*}
Later on, Jiang and Zheng \cite{Jiang-ZHeng-ZAMP-2014} improved this result to the case of $b\geq 2$ and $0\leq\beta< b+9$.

For the corresponding initial-boundary value problem in the bounded interval $(0,1)$ with the free boundary condition $\sigma(t,0)=\sigma(t,1)=-p_e<0$ with $\sigma=-p(v,\theta)+\frac{\mu u_x}{v}$ and homogeneous Neumann conditions $\left(\theta_x(t,0),z_x(t,0)\right)=\left(\theta_x(t,1),z_x(t,1)\right)=\left(0,0\right)$ on both $\theta(t,x)$ and $z(t,x)$, Umehara and Tani \cite{Umehara-Tani-JDE-2007} established the global existence, uniqueness of a classical solutions under the assumptions $4\leq b\leq 16$ and $0\leq\beta\leq \frac{13}{2}$. Later on, they improved their results in \cite{Umehara-Tani-PJA-2008} to the case of $b\geq 3$ and $0\leq\beta< b+9$. Moreover, Qin \cite{Qin-Hu-Wang-Huang-Ma-JMAA-2013} strengthened the results to the case $\left(b, \beta\right)\in E$, where $E=E_{1}\bigcup E_{2}$ with
\begin{eqnarray*}
E_{1}&&=\left\{\left(b,\beta\right)\in\mathbb{R}^{2}:\quad \frac{9}{4}< b< 3,\ 0\leq \beta< 2b+6\right\},\\
E_{2}&&=\left\{\left(b,\beta\right)\in\mathbb{R}^{2}:\quad 3\leq b,\ 0\leq \beta<b+9\right\}.
 \end{eqnarray*}
Jiang and Zheng \cite{Jiang-ZHeng-JMP-2012} further studied global solvability and asymptotic behavior for the problem for the case $b\geq 2$ and $0 \leq \beta < b+9$. It is worth pointing out that all the results mentioned above are concerned with the case when the space variable $x$ belongs to a bounded interval ($x\in \left[0, 1\right]$). Recently, Liao and Zhao \cite{Liao-Zhao-arXive-2017} obtained global existence and large-time behavior of the solutions to the Cauchy problem of the one-dimensional viscous radiative and reactive gas under the assumption $b>\frac{11}{3}$ and $0\leq\beta< b+9$. We refer also the readers to \cite{Chen-SIMA-1992, Ducomet-M3AS-1996, Ducomet-MMNA-1997, Ducomet-BanachCenterPubl-2000, Ducomet-Zlotnik-CRASP-2004, He-Liao-Wang-Zhao-2017, Liao-Zhao-CMS-2017, Qin-Hu-JMP-2011, Shandarin-Zeldovichi-RMP-1989, Wylen-Sonntag-1985, Zhang-Xie-JDE-2008} for more references and some recent discussions.

For the multidimensional case, there are also some results concerning the spherically symmetric flow of compressible viscous and polytropic ideal fluid (cf. \cite{Jiang-CMP-1996, Liang-arXive-2017, Nakamura-Nishibata-2008, Qin-Zhang-Su-Cao-JMFM-2016, Wan-Wang-JDE-2017}). Among them, Jiang \cite{Jiang-CMP-1996} proved the global existence of spherically symmetric smooth solutions for viscous polytropic ideal gas in an exterior domain with large initial data (in dimension $n=2$ or $n=3$). Later on, Nakamura and Nishibata \cite{Nakamura-Nishibata-2008} established the asymptotic behavior of a spherically symmetric solutions to the above problem ($n\geq 3$). Furthermore, Liang \cite{Liang-arXive-2017} used the techniques developed in Li and Liang \cite{Li-Liang-ARMA-2016} to improve the result obtained in \cite{Nakamura-Nishibata-2008} to include the case $n=2$. Recently, Qin \cite{Qin-Zhang-Su-Cao-JMFM-2016} established the global existence and exponential stability of spherically symmetric solutions in $H^{i}\times H^{i}\times H^{i}\times H^{i}$ ($i=1, 2, 4$) for the system \eqref{a20}. We note, however, that the discussion in Qin \cite{Qin-Zhang-Su-Cao-JMFM-2016} is concerned with the case when the space variable $x$ is in a bounded domain, {\it thus it is a natural and interesting problem to established the global existence and large-time behavior result for the spherically solutions of the system \eqref{a20}, \eqref{a21}, \eqref{a22} and \eqref{a8} with large initial data.}

The aim of the present work is devoted to such a problem and the main result can be stated as follows:
\begin{theorem}\label{Th1.1}
Suppose that
\begin{itemize}
\item  The parameters $b$ and $\beta$ are assumed to satisfy:
\begin{eqnarray}\label{1.9}
 b>\frac{19}{4}, \quad 0\leq\beta< b+9;
 \end{eqnarray}
\item The initial data $ \left(v_{0}(x), u_{0}(x), \theta_{0}(x), z_{0}(x)\right)$ satisfy
\begin{eqnarray}\label{aa9}
  \left(v_{0}(x)-1, u_{0}(x ), \theta_{0}(x)-1, z_{0}(x)\right)\in L^{2}\left([0,\infty)\right),\nonumber\\
   \left(r^{n-1}\partial_{x}v_{0}(x), r^{n-1}\partial_{x}u_{0}(x), r^{n-1}\partial_{x}\theta_{0}(x), r^{n-1}\partial_{x}z_{0}(x)\right)\in L^{2}\left([0,\infty)\right),\\
   r^{n-1}\partial_{xx}v_{0}(x)\in L^{2}\left([0,\infty)\right),\quad \partial_{xx}u_{0}(x)\in L^{2}\left([0,\infty)\right),\quad z_{0}(x)\in L^{1}\left([0,\infty)\right),\nonumber\\
     \inf\limits_{x\in[0,\infty) }v_{0}\left(x\right)>0, \quad\inf\limits_{x\in[0,\infty)} \theta_{0}\left(x\right)>0, \quad 0\leq z_{0}\left(x\right)\leq 1,  \quad \forall x\in[0,\infty),\nonumber
  \end{eqnarray}
\end{itemize}
and are compatible with the boundary conditions \eqref{a22}. Then the system \eqref{a20}, \eqref{a21}, \eqref{a22} and \eqref{a8} admits a unique global solution $\left(v(t,x), u(t,x), \theta(t,x), z(t,x)\right)$ which satisfies
\begin{eqnarray}\label{aa11}
\underline{V}\leq v(t,x)&\leq&\overline{V},\nonumber\\
\underline{\Theta}\leq\theta(t,x)&\leq&\overline{\Theta},\nonumber\\
0\leq z(t,x)&\leq& 1
\end{eqnarray}
for all $\left(t,x\right)\in [0,\infty)\times[0,\infty)$ and
\begin{eqnarray}\label{aa12}
&&\sup\limits_{0\leq t<\infty}\left\{\left\|\left(v-1, u, \theta-1, z\right)(t)\right\|^2_{L^2([0,\infty))}+\left\|r^{n-1}\left(v_{x}, u_{x}, \theta_{x}, z_{x}\right)(t)\right\|^2_{L^2([0,\infty))}\right.\nonumber\\
&&\quad \left.+\left\|u_{xx}(t)\right\|^2_{L^2([0,\infty))}+\left\|r^{n-1}v_{xx}(t)\right\|^2_{L^2([0,\infty))}
+\left\|z(t)\right\|_{L^{1}([0,\infty))}\right\}\\
&&+\int_{0}^{\infty}\left\|\left(r^{n-1}v_{x}, r^{n-1}u_{x}, r^{n-1}\theta_{x}, r^{n-1}z_{x}, r^{n-1}v_{xx}, r^{2n-2}u_{xx}, r^{n-1}\theta_{xx}, r^{n-1}z_{xx}\right)(s)\right\|^{2}_{L^2([0,\infty))}ds\leq C.\nonumber
\end{eqnarray}
Here $\underline{V},$ $\overline{V},$ $\underline{\Theta},$ $\overline{\Theta}$ and $C$ are some positive constants which depend only on the fixed constants $\mu$, $\lambda_{1}$, $\lambda$, $K$, $A$, $d$, $R$, $c_{v}$, $a$, $\kappa_{1}$, $\kappa_{2}$, $n$ and the initial data $(v_{0}(x), u_{0}(x), \theta_{0}(x), z_{0}(x))$.

Moreover, the large time behavior of the global solution $\left(v(t,x), u(t,x), \theta(t,x), z(t,x)\right)$ constructed above can be described by the non-vacuum equilibrium state $(1,0,1,0)$ in the sense that
\begin{eqnarray}\label{aa13}
\lim_{t\rightarrow+\infty}\sup\limits_{x\in[0,\infty)}\left|\left(v\left(t,x\right)-1, u, \theta\left(t,x\right)-1, z\left(t,x\right)\right)\right|=0.
\end{eqnarray}

\end{theorem}

\begin{remark} \eqref{aa11} tells us that if the initial data $(v_{0}(x), u_{0}(x), \theta_{0}(x), z_{0}(x))$ is assumed to be without vacuum, mass concentrations, or vanishing temperatures, then the same holds for the unique global solution $\left(v(t,x), u(t,x), \theta(t,x), z(t,x)\right)$ constructed in Theorem \ref{Th1.1}. It follows from Sobolev's imbedding theorem that the unique solution obtained in Theorem \ref{Th1.1} is indeed a globally smooth non-vacuum solution with large initial data. Moreover, this result in Lagrangian coordinates can be easily be converted to an equivalent statement for the corresponding problem in Eulerian coordinates.

\end{remark}

Now we outline the main difficulties of the problem and our strategy to deduce our main result obtained in Theorem \ref{Th1.1}. As pointed out in Liao and Zhao \cite{Liao-Zhao-arXive-2017}, the crucial step to construct the global solutions of the initial-boundary value problem \eqref{a20}, \eqref{a21}, \eqref{a22} and \eqref{a8} with large initial data is to obtain the positive lower and upper bounds of the specific volume $v\left(t,x\right)$ and the temperature $\theta\left(t,x\right)$. To this end, Jiang and Zheng \cite{Jiang-ZHeng-ZAMP-2014} used the techniques developed by Kazhikov and Shelukhin \cite{Kazhikhov-Shelukhin-JAMM-1977} to derive a representation formula of $v\left(t,x\right)$ to deduce the desired lower and upper bounds of the specific volume $v(t,x)$. For our problem, such a method loses its power since $x$ is in an unbounded domain ($x\in[0,\infty)$). Another difficulty lies on the way to derive estimates on some nonlinear terms in order to deduce the upper bound of $\theta\left(t,x\right)$. Among them, the most difficult term we need to control is the high order (about $r\left(t, x\right)$) nonlinear term $\int_{0}^{t}\int_{\Omega}r^{4n-4}u^{4}_{x}dxds$. In Qin \cite{Qin-Zhang-Su-Cao-JMFM-2016}, since $r(t,x)$ is assumed to satisfy $0<r_{1}\leq r\left(t, x\right)\leq r_{2}$ with $r_{1}$ and $r_{2}$ being two positive constants, the region $\Omega$ under their consideration there is bounded, one can thus control the term $\int_{0}^{t}\int_{\Omega}r^{4n-4}u^{4}_{x}dxds$ by the term $\int_{0}^{t}\int_{\Omega}u^{4}_{x}dxds$ (cf. (3.57) in \cite{Qin-Zhang-Su-Cao-JMFM-2016}) since
\begin{eqnarray}\label{aa14}
\int_{0}^{t}\int_{\Omega}r^{4n-4}u^{4}_{x}dxds\leq C\int_{0}^{t}\int_{\Omega}u^{4}_{x}dxds.
\end{eqnarray}
But for the problem considered in this paper, since $r\left(t, x\right)\in [1,\infty)$ lies in an unbounded domain, the estimate \eqref{aa14} does not hold true any longer! To overcome such difficulties, we will first deduce bounds on $\left\|v_{x}(t)\right\|_{L^2([0,\infty))}$, $\int_{0}^{t}\left\|r^{n-1}(s)u_{xx}(s)\right\|^{2}_{L^2([0,\infty))}ds$, $\int_{0}^{t}\left\|r^{n-1}(s)u_{x}(s)\right\|^{2}_{L^{\infty}([0,\infty))}ds$, $\left\|r^{n-1}(t)v_{x}\right\|_{L^2([0,\infty))}$ and $\left\|r^{n-1}(t)u_{x}(t)\right\|_{L^2([0,\infty))}$ in terms of $\left\|\theta\right\|_{L^\infty([0,T]\times[0,\infty))}$. Then the nonlinear term $\int_{0}^{t}\int_{0}^{\infty}r^{4n-4}u^{4}_{x}dxds$ can be estimated as follows (see also \eqref{bb147} and the definition of $l_{1}$, $l_{2}$ and $l_{3}$ can be seen in \eqref{bd58}, \eqref{cc1} and \eqref{cb24}, respectively)
\begin{eqnarray*}
\int_{0}^{t}\int_{0}^{\infty}r^{4n-4}u^{4}_{x}dxds&\leq& C\int_{0}^{t}\left(\left\|r^{n-1}(s)u_{x}(s)\right\|^{2}_{L^{\infty}([0,\infty))}
\int_{0}^{\infty}r^{2(n-1)}(s,x)u^{2}_{x}(s,x)dx\right)ds\\
&\leq& C\left\|\theta\right\|_{L^\infty([0,T]\times[0,\infty))}^{2l_{2}+l_{3}}.
\end{eqnarray*}
The above estimates will play a central role in our analysis.

In summary, the key points in our discussion can be outlined as in the following:
\begin{itemize}
\item [(i).] We first construct a normalized entropy $\widetilde{E}$ (see (\eqref{b3})) to the system \eqref{a20} to derive the basic energy estimates for our problem, which will play a fundamental role in deducing the desired uniform positive lower and upper bounds of the specific volume $v(t,x)$ and temperature $\theta(t,x)$. We emphasize that the method to deduce the basic energy estimates here is different from that used in  \cite{Qin-Zhang-Su-Cao-JMFM-2016} due to the unboundedness of the domain under our consideration;

\item [(ii).] Motivated by the works of Jiang \cite{Jiang-CMP-1996, Jiang-AMPA-1998,  Jiang-CMP-1999, Jiang-PRSE-2002} on the one-dimensional, compressible Navier-Stokes system for a viscous and heat conducting ideal polytropic gas, we use a special cut-off function $\varphi\left(x\right)$ (see \eqref{b23}) as in Liao and Zhao \cite{Liao-Zhao-arXive-2017} to derive a new representation of $v\left(t,x\right)$, that is, \eqref{b24}. Based on such a formula, we can derive the desired uniform upper bound of $v\left(t,x\right)$ for $0\leq t\leq T$ and the desired uniform positive lower bound of $v(t,x)$ for the time range $t_0\leq t\leq T$ with $t_0$ being a suitably chosen sufficiently large positive constant. Then we adopt the method developed by Kazhikhov and Shelukhin in \cite{Kazhikhov-Shelukhin-JAMM-1977} (cf. also \cite{Antontsev-Kazhikov-Monakhov-1990}) for the one-dimensional, compressible Navier-Stokes system for a viscous and heat conducting ideal polytropic gas to yield the desired positive lower bound of $v\left(t,x\right)$ when $0< t<t_{0}$. Noticing that the term $Q\left(t,x\right)$ defined in \eqref{b25} depends on both time variable $t$ and space variable $x$ due to the additional term
    $-\frac{1}{\alpha}\int_{0}^{t}\int_{x}^{\infty}\left(n-1\right)\varphi u^{2}r^{-n}dyds$, which is different from the corresponding term in Liao and Zhao \cite{Liao-Zhao-arXive-2017} (see (2.16) in \cite{Liao-Zhao-arXive-2017}). The important point to note here is that \eqref{bb54} will play a key role in our discussion and all the bounds obtained above are independent of time variable $t$, which is crucial in studying the large-time behavior of our problem;

\item [(iii).] Having obtained the desired uniform positvie lower and upper bound of $v\left(t,x\right)$, we turn to estimate the term $\left\|v_{x}(t)\right\|_{L^2([0,\infty))}$, $\int_{0}^{t}\left\|r^{n-1}(s)u_{xx}(s)\right\|^{2}_{L^2([0,\infty))}ds$, $\int_{0}^{t}\left\|r^{n-1}(s)u_{x}(s)\right\|^{2}_{L^{\infty}([0,\infty))}ds$, $\left\|r^{n-1}(t)v_{x}(t)\right\|_{L^2([0,\infty))}$ and $\left\|r^{n-1}(t)u_{x}(t)\right\|_{L^2([0,\infty))}$ in terms of $\left\|\theta\right\|_{L^\infty([0,T]\times[0,\infty))}$ in Lemma 3.7- Lemma 3.11, which will be useful in deriving the upper bound of $\theta\left(t,x\right)$.

    It should be pointed out that \eqref{c16} holds true only for the case of $n\geq 3$, this is also the main reason why the results obtained in \cite{Nakamura-Nishibata-2008} hold true only for the case of $n\geq 3$. Fortunately, we can deduce \eqref{cc16} for the case $n\geq 2$. With \eqref{cc16} in hand, we can obtain Lemma 3.7- Lemma 3.11 by using Gronwall's inequality to continue our discussion;

\item [(iv).] Motivated by Kawohl \cite{Kawohl-JDE-1985} and Liao and Zhao \cite{Liao-Zhao-arXive-2017}, we introduce some auxiliary functions $X(t)$, $Y(t)$ and $Z(t)$ (see \eqref{b56}) to deduce the upper bound of $\theta\left(t,x\right)$. We will see that the definition of $Y(t)$ is different from that in defined in Liao and Zhao \cite{Liao-Zhao-arXive-2017} (see (2.51) in Liao and Zhao \cite{Liao-Zhao-arXive-2017}) which is due to \eqref{b70} and the fact that $r\left(t,x\right)\geq 1$.

    It is worth to point out that the method used by Liang in \cite{Liang-arXive-2017} to deduce the uniform upper bound of $\theta\left(t,x\right)$ (see also Corollary 4.5 in Liang \cite{Liang-arXive-2017}) relies on the following Sobolev inequality
\begin{eqnarray*}
\left\|\theta(t)-1\right\|^{2}_{L^{\infty}([0,\infty))}\leq C\left\|\theta(t)-1\right\|_{L^2([0,\infty))}\left\|\theta_{x}(t)\right\|_{L^2([0,\infty))}
\leq C\left(1+\left\|\theta\right\|_{L^\infty([0,T]\times[0,\infty))}\right).
\end{eqnarray*}
For our problem $\left\|\theta(t)-1\right\|_{L^2([0,\infty))}$ is bounded due to \eqref{b3}, \eqref{bb13} and \eqref{b43}, but the method employed in \cite{Liang-arXive-2017} to deduce the estimate on $\left\|\theta_{x}(t)\right\|_{L^2([0,\infty))}$ loses its power in our case which is due to the fourth-order radiative part in both $P(v,\theta)$ and $e(v,\theta)$, cf. \eqref{a8};

\item [(v).] Finally, we derive a local in time estimate on the lower bound of $\theta\left(t,x\right)$ (see \eqref{b180}) in Section 5. Although such a bound depends on time variable $t$, it is sufficient to extend the local solution to a global one by combining the above estimates with the continuation argument designed in Liao and Zhao \cite{Liao-Zhao-arXive-2017}, which is motivated by \cite{Wan-Wang-Zhao-JDE-2016, Wan-Wang-Zou-Nonlinearity-2016, Wang-Zhao-M3AS-2016}.
\end{itemize}

Before concluding this section, it is worth pointing out that since the energy producing process inside the medium is taken into account in the equations \eqref{a1}, that is, the gas consists of a reacting mixture and the combustion process is current at the high temperature stage, and the experimental results for gases at high temperatures in \cite{Zeldovich-Raizer-1967} show that the viscosity coefficient $\mu$ may depend on the specific volume $v(t,x)$ and/or temperature $\theta(t,x)$. Thus it would be interesting and necessary to consider the corresponding global wellposedness theory for the case when the viscosity coefficient $\mu$ is a function of $v$ and $\theta$.

For such a problem, if $n=1$ and the viscosity coefficient $\mu$ is a smooth function of the specific volume $v$ for $v>0$ which can be degenerate, some global solvability results are established in \cite{Liao-Zhao-CMS-2017} for the above mentioned two types of initial-boundary value problems of the system \eqref{a1}, \eqref{a2}, \eqref{a5}, \eqref{a7}, \eqref{a8} (cf. also \cite{Chen-Zhao-Zou-PRSE-2017, Jenssen-Karper-SIMA-2010, Kanel-DU-1968, Pan-Zhang-CMS-2015, Tan-Yang-Zhao-Zou-SIMA-2013} for the corresponding results for one-dimensional compressible Navier-Stokes equations for a viscous and heat conducting ideal polytropic gas).

As for the case when the viscosity coefficient $\mu$ depends also on the temperature, note that even for one-dimensional compressible Navier-Stokes equations for a viscous and heat conducting ideal polytropic gas, as pointed out in \cite{Jenssen-Karper-SIMA-2010}, temperature dependence of the viscosity $\mu$ has turned out to be especially problematic. Even so, there are some recent progress in this problem for viscous heat-conducting ideal polytropic gas, cf. \cite{Huang-Liao-M3AS-2017, Huang-Wang-Xiao-KRM-2016, Liu-Yang-Zhao-Zou-SIMA-2014, Wan-Wang-JDE-2017, Wang-Zhao-M3AS-2016} and the references cited therein, and a result similar to that of \cite{Wang-Zhao-M3AS-2016} has been obtained in \cite{He-Liao-Wang-Zhao-2017} for the Cauchy problem of the system modeling one-dimensional viscous radiative and reactive gas when the viscosity coefficient $\mu$ is a smooth function of $v$ and $\theta$. We're convinced that the arguments used in \cite{He-Liao-Wang-Zhao-2017} and this paper can be also adapted to construct spherically symmetric solutions to the system \eqref{a1}, \eqref{a2}, \eqref{a5}, \eqref{a7}, \eqref{a8} in an unbounded domain exterior to the unit sphere $B_1(O)\subset\mathbb{R}^n$ for $n\geq 2$ when the viscosity coefficient $\mu$ is a smooth function of $v$ and $\theta$ and such a problem is under our current research.

The rest of the paper is organized as follows: we first derive some useful energy type estimates in Section 2. Section 3 is devoted to yielding the desired uniform positive lower and upper bound of $v\left(t,x\right)$. Then we will deduce the uniform upper bound of $\theta\left(t,x\right)$ in Section 4. Finally, the local in time lower bound of $\theta\left(t,x\right)$ will be obtained and thus completes the proof of Theorem 1.1 in Section 5.

\bigbreak
\noindent{\bf Notations:}\quad Throughout this paper, $C\geq 1$ is used to denote a generic positive constant which depends only on the fixed constants $\mu$, $\lambda_{1}$, $\lambda$, $K$, $A$, $d$, $R$, $c_{v}$, $a$, $\kappa_{1}$, $\kappa_{2}$, $n$ and the initial data $\left(v_{0}(x), u_{0}(x), \theta_{0}(x), z_{0}(x)\right)$. Note that such a onstant may vary from line to line. $C\left(\cdot,\cdot\right)$ stands for some generic positive constant depending only on the quantities listed in the parenthesis.  $\epsilon$ stands for some small positive constant.

For function spaces, ~$L^q\left([0,\infty)\right)\left(1\leq q\leq \infty\right)$~denotes the usual Lebesgue space on~$[0,\infty)$ with norm $\|\cdot\|_{L^q\left([0,\infty)\right)},$ while $H^q\left([0,\infty)\right)$ represents for the usual Sobolev space in the $L^2$ sense with norm $\|{\cdot}\|_{H^q\left([0,\infty)\right)}$. For simplicity, we use $\|\cdot\|_{\infty}$ to denote the norm in $L^{\infty}\left([0,T]\times[0,\infty)\right)$ for some $T>0$ and use $\|\cdot\|$ to denote the norm ~$\|\cdot\|_{L^2\left([0,\infty)\right)}$.

\section{Basic energy estimates}
The main purpose of this section is to deduce certain energy type estimates on the solutions of the initial-boundary value problem \eqref{a20}, \eqref{a21}, \eqref{a22} and \eqref{a8} in terms of the initial data $(v_0(x), u_0(x), \theta_0(x), z_0(x))$. To this end, for some constants $0<T\leq +\infty$, $0<M_1<M_2,$ $0<N_1<N_2$, we first define the set of functions $X(0,T;M_1,M_2;N_1,N_2)$ for which we seek the solution of the initial-boundary value problem \eqref{a20}, \eqref{a21}, \eqref{a22} and \eqref{a8} as follows:
\begin{eqnarray*}
   &&X(0, T;M_1,M_2;N_1,N_2)\\
&:=&\left\{\left(v, u,\theta,z\right)\left(t,x\right) \left|
   \begin{array}{c}
   0\leq z(t,x)\in  C\left(0,T;L^{2}\left(\Omega\right)\cap L^1([0,\infty))\right),\\
   \left(v\left(t,x\right)-1,u\left(t,x\right),\theta\left(t,x\right)-1\right)\in C\left(0,T;L^{2}\left([0,\infty)\right)\right),\\
   \left( r^{n-1}v_{x},  r^{n-1}u_{x},  r^{n-1}\theta_{x},  r^{n-1}z_{x}\right)\left(t,x\right)\in C\left(0,T;L^{2}\left([0,\infty)\right)\right),\\
  \left( r^{n-1}v_{x},  r^{n-1}u_{x},  r^{n-1}\theta_{x},  r^{n-1}z_{x}\right)\left(t,x\right)\in L^{2}\left(0,T;L^{2}\left([0,\infty)\right)\right),\\
  \left(r^{n-1}v_{xx}, u_{xx}\right)\in C\left(0,T;L^{2}\left([0,\infty)\right)\right),\\
   \left(r^{n-1}v_{xx}, r^{n-1}u_{xx},  r^{n-1}\theta_{xx},  r^{n-1}z_{xx}\right)\left(t,x\right)\in L^{2}\left(0,T;L^{2}\left([0,\infty)\right)\right),\\
   M_1\leq v(t,x)\leq M_2,\ \forall (t,x)\in[0,T]\times[0,\infty),\\
   N_1\leq \theta(t,x)\leq N_2,\ \forall (t,x)\in[0,T]\times[0,\infty)
   \end{array}
   \right.
   \right\}.
  \end{eqnarray*}

The standard local wellposedness result on the initial-boundary value problem of the hyperbolic-parabolic coupled system tells us that there exists a sufficiently small positive constant $t_1>0$, which depends only on $m_0=\inf\limits_{x\in[0,\infty)} v_0(x), n_0=\inf\limits_{x\in[0,\infty)} \theta_0(x)$ and
\begin{eqnarray*}
   \ell_0&=&\left\|\left(v_{0}(x)-1,u_{0}(x),\theta_{0}(x)-1,z_{0}(x)\right)\right\|^{2}+\left\|r^{n-1}\left(\partial_{x}v_{0}(x),\partial_{x}u_{0}(x),
   \partial_{x}\theta_{0}(x),\partial_{x}z_{0}(x)\right)\right\|^{2}\\
   &&+\left\|z_{0}\right\|_{L^{1}(\Omega)}+\left\|r^{n-1}\partial_{xx}v_{0}(x)\right\|^{2}+\left\|\partial_{xx}u_{0}(x)\right\|^{2}
  \end{eqnarray*}
such that the initial-boundary value problem \eqref{a20}, \eqref{a21}, \eqref{a22} and \eqref{a8} admits a unique solution $(v(t,x), u(t,x),\theta(t,x)$, $z(t,x))\in X(0,t_1;m_0/2,2+2\ell_0;n_0/2,2+2\ell_0)$. Now suppose that such a solution has been extended to the time step $t=T\geq t_1$ and $(v(t,x), u(t,x),\theta(t,x),z(t,x))\in X(0,T;M_1,M_2;N_1,N_2)$ for some positive constants $M_2>M_1>0, N_2>N_1>0$, we now try to deduce certain a priori energy type estimates on $(v(t,x), u(t,x),\theta(t,x),z(t,x))$ in terms of the initial data $(v_0(x), u_0(x), \theta_0(x), z_0(x))$.

Our first result is concerned with the basic energy estimates, which will play a fundamental role in deducing the desired positive lower and upper bounds of $v\left(t,x\right)$. To do so, if we use $E(v,\theta)$ to denote the entropy, then the second law of thermodynamics asserts that
\begin{eqnarray}\label{b1}
   \frac{\partial E\left(v,\theta\right)}{\partial v}&=&\frac{\partial P\left(v,\theta\right)}{\partial\theta},\nonumber\\
    \frac{\partial E\left(v,\theta\right)}{\partial\theta}&=&\frac{1}{\theta}\frac{\partial e\left(v,\theta\right)}{\partial\theta},\\
    \frac{\partial e\left(v,\theta\right)}{\partial v}&=&\theta \frac{\partial P(v,\theta)}{\partial\theta}-P\left(v,\theta\right).\nonumber
\end{eqnarray}

From which and the constitutive relations \eqref{a8}, one easily deduce that
\begin{equation*}\label{b2}
E(v,\theta)=c_{v}\ln\theta+\frac{4}{3}av\theta^{3}+R\ln v
 \end{equation*}
and the normalized entropy $\widetilde{E}(v,\theta)$ around $(v,\theta)=(1,1)$ is given by
\begin{eqnarray}\label{b3}
\widetilde{E}(v,\theta)&&=c_{v}\theta+av\theta^{4}-\left(c_{v}+a\right)+\left(R+\frac{a}{3}\right)\left(v-1\right)-\left(E-\frac{4}{3}a\right)\nonumber\\
&&=c_{v}\left(\theta-\ln\theta-1\right)+R\left(v-\ln v-1\right)+\frac{1}{3}av\left(\theta-1\right)^{2}\left(3\theta^{2}+2\theta+1\right).
  \end{eqnarray}
Moreover, one can deduce from \eqref{a20}, \eqref{a8} and \eqref{b3} that
\begin{eqnarray}\label{b4}
&&\left(\widetilde{E}+\frac{u^{2}}{2}\right)_{t}+\frac{\alpha\left|\left(r^{n-1}u\right)_{x}\right|^{2}}{v\theta}
+\frac{\kappa\left(r^{n-1}\theta_{x}\right)^{2}}{v\theta^{2}}
+\frac{\lambda\phi z}{\theta}+2\mu\left(n-1\right)\left(1-\frac{1}{\theta}\right)\left(r^{n-2}u^{2}\right)_{x}\nonumber\\
&=&\left[r^{n-1}u\left(\frac{\alpha\left(r^{n-1}u\right)_{x}}{v}-\frac{R\theta}{v}\right)+\left(R+\frac{a}{3}-\frac{a}{3}\theta^{4}\right)r^{n-1}u
+\left(1-\frac{1}{\theta}\right)\frac{\kappa r^{2n-2}\theta_{x}}{v}\right]_{x}
+\lambda\phi z.
\end{eqnarray}

Now integrating $\eqref{a20}_{4}$ with respect to $t$ and $x$ over $[0,t)\times[0,\infty)$ and by using the boundary conditions \eqref{a22}, we have
\begin{eqnarray}\label{b5}
   \int_{0}^{\infty} z(t,x)dx+\int_{0}^{t}\int_{0}^{\infty} \phi(t,x) z(t,x)dxds=\int_{0}^{\infty} z_{0}(x)dx.
\end{eqnarray}
Then by integrating the identity $\eqref{b4}$ with respect to $t$ and $x$ over $[0,t)\times[0,\infty)$ and by using the identity \eqref{b5} and the boundary conditions \eqref{a22}, we can get that
\begin{eqnarray}\label{b6}
   &&\int_{0}^{\infty}\left(\widetilde{E}(t,x)+\frac 12 u^2(t,x)\right)dx
+\int_{0}^{t}\int_{0}^{\infty}\left(\frac{\alpha\left|\left(r^{n-1}u\right)_{x}\right|^{2}}{v\theta}+\frac{\kappa\left(r^{n-1}\theta_{x}\right)^{2}}{v\theta^{2}}
+\frac{\lambda\phi z}{\theta}\right)(s,x)dxds\nonumber\\
&=&
\int_{0}^{\infty}\left(\widetilde{E}_0(x)+\frac 12 u_0(x)^2\right)dx+\int_{0}^{t}\int_{0}^{\infty}\left(\lambda\phi z+2\mu\left(n-1\right)\frac{\left(r^{n-2}u^{2}\right)_{x}}{\theta}\right)dxds\\
&\leq& C+2\mu\left(n-1\right)\int_{0}^{t} \int_{0}^{\infty}\frac{\left(r^{n-2}u^{2}\right)_{x}}{\theta}dxds.\nonumber
\end{eqnarray}

In view of \eqref{a23} and by simple calculation, we can deduce that
\begin{eqnarray}\label{b7}
&&\frac{\alpha\left|\left(r^{n-1}u\right)_{x}\right|^{2}}{v\theta}-2\mu\left(n-1\right)\frac{\left(r^{n-2}u^{2}\right)_{x}}{\theta}\nonumber\\
&=&\left(\lambda_{1}+\frac{2\mu}{n}\right)\frac{\left|\left(r^{n-1}u\right)_{x}\right|^{2}}{v\theta}\\
&&+\frac{2\mu\left(n-1\right)}{v\theta}\left[\frac{\left(r^{n-1}u\right)_{x}}{\sqrt{n}} -\frac{\sqrt{n}vu}{r}\right]^{2}\geq 0,\nonumber
\end{eqnarray}
and consequently we can compute from \eqref{b6} and \eqref{b7} that
\begin{eqnarray}\label{b8}
&&\int_{0}^{\infty}\left(\widetilde{E}(t,x)+\frac 12 u^2(t,x)\right)dx+\int_{0}^{t}\int_{0}^{\infty}\bigg\{\frac{\kappa\left(r^{n-1}\theta_{x}\right)^{2}}{v\theta^{2}}+
\left(\lambda_{1}+\frac{2\mu}{n}\right)\frac{\left|\left(r^{n-1}u\right)_{x}\right|^{2}}{v\theta}\nonumber\\
&&+\frac{2\mu\left(n-1\right)}{v\theta}\left[\frac{\left(r^{n-1}u\right)_{x}}{\sqrt{n}}-\frac{\sqrt{n}vu}{r}\right]^{2}
+\frac{\lambda\phi z}{\theta}\bigg\}(s,x)dxds\leq C.
\end{eqnarray}

A direct consequence of the inequality \eqref{b8} is
\begin{eqnarray}\label{b9}
&&\int_{0}^{t}\int_{0}^{\infty}\left(\frac{\left|\left(r^{n-1}u\right)_{x}\right|^{2}}{v\theta}+\frac{vu^{2}}{r^{2}\theta}\right)dxds\nonumber\\
&\leq& C+C\int_{0}^{t}\int_{0}^{\infty}\frac{\left|u\left(r^{n-1}u\right)_{x}\right|}{r\theta}\\
&\leq & C+\frac{1}{2}\int_{0}^{t}\int_{0}^{\infty}\frac{vu^{2}}{r^{2}\theta}dxds+C\int_{0}^{t}\int_{\Omega}\frac{\left|\left(r^{n-1}u\right)_{x}\right|^{2}}{v\theta}dxds\nonumber\\
&\leq& C+\frac{1}{2}\int_{0}^{t}\int_{0}^{\infty}\frac{vu^{2}}{r^{2}\theta}dxds,\nonumber
\end{eqnarray}
from which one can infer that
\begin{eqnarray}\label{b10}
\int_{0}^{t}\int_{0}^{\infty}\frac{vu^{2}}{r^{2}\theta}dxds\leq C.
\end{eqnarray}
Moreover, by making use of \eqref{a23} again, we find that
\begin{eqnarray}\label{b11}
&&\int_{0}^{t}\int_{0}^{\infty}\frac{r^{2(n-1)}u^{2}_{x}}{v\theta}dxds\nonumber\\ &=&\int_{0}^{t}\int_{0}^{\infty}\frac{\left[\left(r^{n-1}u\right)_{x}-\left(r^{n-1}\right)_{x}u\right]^{2}}{v\theta}dxds\nonumber\\
&=&\int_{0}^{t}\int_{0}^{\infty}\frac{\left[\left(r^{n-1}u\right)_{x}-\left(n-1\right)uvr^{-1}\right]^{2}}{v\theta}dxds\\
&\leq& C\int_{0}^{t}\int_{0}^{\infty}\left(\frac{\left|\left(r^{n-1}u\right)_{x}\right|^{2}}{v\theta}+\frac{vu^{2}}{r^{2}\theta}\right)dxds\nonumber\\
&\leq& C.\nonumber
\end{eqnarray}

Putting \eqref{b6}-\eqref{b11} together, we arrive at
\begin{eqnarray}\label{b12}
&&\int_{0}^{\infty}\left(\widetilde{E}(t,x)+\frac 12 u^2(t,x)\right)dx\nonumber\\
&&+\int_{0}^{t}\int_{0}^{\infty}\left\{\frac{\kappa\left(r^{n-1}\theta_{x}\right)^{2}}{v\theta^{2}}+
\frac{\left|\left(r^{n-1}u\right)_{x}\right|^{2}}{v\theta}
+\frac{vu^{2}}{r^{2}\theta}+\frac{r^{2(n-1)}u^{2}_{x}}{v\theta}
+\frac{\phi z}{\theta}\right\}(s,x)dxds\\
&\leq& C.\nonumber
\end{eqnarray}

Finally, multiplying $\eqref{a20}_{4}$ by $z(t,x)$ and integrating the result identity with respect to $t$ and $x$ over $[0,t)\times[0,\infty)$ and by using the boundary conditions \eqref{a22}, we can get that
\begin{eqnarray}\label{b13}
 \int_{0}^{\infty} z^{2}(t,x)dx+\int_{0}^{t}\int_{0}^{\infty}\left(\frac{dr^{2n-2}z_{x}^{2}}{v^{2}}+\phi z^{2}\right)(s,x)dxds
  =\int_{0}^{\infty} z_{0}^{2}(x)dx.
\end{eqnarray}

Combining \eqref{b1}-\eqref{b13}, we can obtain the following lemma
\begin{lemma} [Basic energy estimates] Suppose that $(v(t,x), u(t,x),\theta(t,x),z(t,x))\in X(0,T;M_1,M_2;$ $N_1,N_2)$ for some positive constants $T>0, M_2>M_1>0, N_2>N_1>0$, then for all $0\leq t\leq T$, we have
\begin{eqnarray}\label{b14}
   \int_{0}^{\infty} z(t,x)dx+\int_{0}^{t}\int_{0}^{\infty} \phi(s,x) z(s,x)dxds=\int_{0}^{\infty} z_{0}(x)dx,
\end{eqnarray}
\begin{eqnarray}\label{b15}
 \int_{0}^{\infty} z^{2}(t,x)dx +\int_{0}^{t}\int_{0}^{\infty}\left(\frac{dr^{2n-2}z_{x}^{2}}{v^{2}}+\phi z^{2}\right)(s,x)dxds
  =\int_{0}^{\infty} z_{0}^{2}(x)dx,
\end{eqnarray}
and
\begin{eqnarray}\label{bb13}
&&\int_{0}^{\infty}\left(\widetilde{E}(t,x)+\frac 12 u^2(t,x)\right)dx\nonumber\\
&&+\int_{0}^{t}\int_{0}^{\infty}\left\{\frac{\kappa\left(r^{n-1}\theta_{x}\right)^{2}}{v\theta^{2}}+
\frac{\left|\left(r^{n-1}u\right)_{x}\right|^{2}}{v\theta}
+\frac{vu^{2}}{r^{2}\theta}+\frac{r^{2(n-1)}u^{2}_{x}}{v\theta}
+\frac{\phi z}{\theta}\right\}(s,x)dxds\\
&\leq& C.\nonumber
\end{eqnarray}

\end{lemma}
The next lemma is concerned with the estimate of $z(t,x)$. To this end, we can deduce by repeating the method used in \cite{Chen-SIMA-1992} that
\begin{lemma} Under the assumptions stated in Lemma 2.1, we have for any $\left(t,x\right)\in [0,T]\times[0,\infty)$ that
  \begin{eqnarray}\label{bb14}
  0\leq z\left(t,x\right)\leq 1.
  \end{eqnarray}
\end{lemma}

\section{Uniform bounds of $v\left(t,x\right)$}
This section is devoted to deducing the unform positive lower and upper bounds of the specific volume $v\left(t, x\right)$ which are independent of the time variable $t$. For this purpose, we first give the following two lemmas, which will be frequently used later on. The first one is concerned with the bounds of $v(t,x)$ and $\theta(t,x)$ at some specially chosen points whose proof can be found in \cite{Jiang-CMP-1999}.
\begin{lemma} Let $a_{1}$, $a_{2}$ be two (positive) roots of the equation $y-\log y-1=\frac{C}{\min \left\{R, c_{v}\right\}}$ with $C$ being given by \eqref{bb13}. Then for each $k\in\mathbb{N}$ and every $0<t\leq T$, there exist $a_{k}\left(t\right), b_{k}\left(t\right)\in \left[k,k+1\right]$ such that
 \begin{eqnarray}\label{b17}
 a_{1}&\leq&\int_{k}^{k+1}v\left(t,x\right)dx\leq a_2,\nonumber\\
 a_1&\leq &\int_{k}^{k+1}\theta\left(t,x\right)dx\leq a_{2}
 \end{eqnarray}
and
 \begin{eqnarray}\label{b18}
 a_{1}\leq v\left(t,a_{k}\left(t\right)\right)\leq a_2,\quad a_1\leq \theta\left(t,b_{k}\left(t\right)\right)\leq a_{2}
  \end{eqnarray}
hold for $0<t\leq T$.
\end{lemma}
The next lemma is concerned with a rough estimate on $\theta(t,x)$ in terms of the entropy dissipation rate functional
$$
V(t)=\int_{0}^{\infty}\left\{\frac{\kappa\left(r^{n-1}\theta_{x}\right)^{2}}{v\theta^{2}}+
\frac{\left|\left(r^{n-1}u\right)_{x}\right|^{2}}{v\theta}+\frac{vu^{2}}{r^{2}\theta}+\frac{r^{2(n-1)}u^{2}_{x}}{v\theta}
+\frac{\phi z}{\theta}\right\}\left(t,x\right)dx.
$$
To this end, to simplify the presentation, we set $\Omega_{k}:=\left(k,k+1\right), k\in \mathbb{N}$ and we can get that
\begin{lemma}For $0\leq m\leq\frac{b+1}{2}$ and each $x\in[0,\infty)$ (without loss of generality, we can assume that $x\in\Omega_k$ for some $k\in\mathbb{N}$), we can deduce that
\begin{eqnarray}\label{b19}
\left |\theta^{m}\left(t,x\right)-\theta^{m}\left(t, b_{k}\left(t\right)\right)\right|\leq CV^{\frac{1}{2}}\left(t\right)
\end{eqnarray}
holds for $0\leq t\leq T$ and consequently
\begin{eqnarray}\label{b20}
 \left|\theta\left(t,x\right)\right|^{2m}\leq C+CV\left(t\right), \quad x\in\overline{\Omega}_{k},\ \  0\leq t\leq T.
\end{eqnarray}
\end{lemma}
\begin{proof} For $x\in\Omega_k$, it is easy to see from \eqref{a8} that
\begin{eqnarray}\label{b21}
&&\left|\theta^{m}\left(t,x\right)-\theta^{m}\left(t, b_{k}\left(t\right)\right)\right|\nonumber\\
&\leq& C\int_{\Omega_{k}}\left|\theta^{m-1}\theta_{x}\right|dx\\
&\leq& C\left(\int_{\Omega_{k}}\frac{v\theta^{2m}}{\left(1+v\theta^{b}\right)r^{2(n-1)}}dx\right)^{\frac{1}{2}}
 \left(\int_{\Omega_{k}}\frac{\kappa\left(r^{n-1}\theta_{x}\right)^{2}}{v\theta^{2}}dx\right)^{\frac{1}{2}}\nonumber\\
&\leq& C\left(\int_{\Omega_{k}}\frac{v\theta^{2m}}{1+v\theta^{b}}dx\right)^{\frac{1}{2}}V\left(t\right)^{\frac{1}{2}}.\nonumber
\end{eqnarray}
Moreover, since
\begin{eqnarray*}
\theta^{2m}\leq C\left(1+\theta^{b+1}\right),
  \end{eqnarray*}
holds for $0\leq m\leq \frac{b+1}{2}$, one thus gets from \eqref{b17} that
\begin{eqnarray}\label{b22}
\int_{\Omega_{k}}\frac{v\theta^{2m}}{1+v\theta^{b}}dx\leq C\int_{\Omega_{k}}\left(v+\theta\right)dx\leq C.
\end{eqnarray}

Combining (\ref{b21}) and (\ref{b22}), we can deduce the estimates \eqref{b19} and \eqref{b20} immediately from \eqref{b18}. This completes the proof of Lemma 3.2.
\end{proof}

We now turn to obtain the lower bound and upper bound of the specific volume $v\left(t,x\right)$ which are independent of the time variable $t$. To this end, motivated by the work of Jiang \cite{Jiang-CMP-1999} for the one-dimensional viscous and heat-conducting ideal polytropic gas motion, we first give a local representation of $v\left(t,x\right)$ by using the following cut-off function $\varphi\in W^{1,\infty}\left([0,\infty)\right)$:
\begin{eqnarray}\label{b23}
 \varphi\left(x\right)=
 \begin{cases}
 1, & \quad 0\leq x\leq k+1,\\
 k+2-x, & \quad k+1\leq x\leq k+2,\\
 0, & \quad x\geq k+2,
 \end{cases}
\end{eqnarray}
from which one can deduce that
\begin{lemma} Under the assumptions stated in Lemma 2.1, we have for each $0\leq t\leq T$ that
\begin{eqnarray}\label{b24}
v\left(t,x\right)=B\left(t,x\right)Q\left(t, x\right)+\frac{1}{\alpha}\int_{0}^{t}\frac{B\left(t,x\right)Q\left(t,x\right)v\left(s,x\right)P\left(s,x\right)}{B\left(s,x\right)Q\left(s,x\right)}ds,
\quad x\in\overline{\Omega}_k.
\end{eqnarray}
Here
\begin{eqnarray}\label{b25}
B\left(t,x\right)&:=&v_{0}\left(x\right)\exp\left\{\frac{1}{\alpha}\int_{x}^{\infty}\left(u_{0}r_{0}^{1-n}-ur^{1-n}\right)\varphi\left(y\right)dy\right\},\nonumber\\
Q\left(t,x\right)&:=&\exp\left\{\frac{1}{\alpha}\int_{0}^{t}\int_{k+1}^{k+2}\sigma\left(s,y\right)dyds-\frac{1}{\alpha}\int_{0}^{t}\int_{x}^{\infty}\left(n-1\right)
\varphi u^{2}r^{-n}dyds\right\},\\
\sigma &:=&-P(v,\theta)+\frac{\alpha\left(r^{n-1}u\right)_{x}}{v}.\nonumber
\end{eqnarray}
\end{lemma}

\begin{proof} For each $x\in\overline{\Omega}_{k}$, we can get by multiplying $(\ref{a20})_{2}$ by $\varphi(x)$ and integrating the resulting identity with respect to $x$ over $\left(x,\infty\right)$ that
\begin{eqnarray}\label{b26}
&&-\int_{x}^{\infty}\left[u\left(t,y\right) \varphi\left(y\right)r^{1-n}\left(t,y\right)\right]_{t}dy\nonumber\\
&=&\alpha\left[\log v(t,x)\right]_{t}-P\left(t,x\right)-\int_{k+1}^{k+2}\sigma\left(t,y\right)dy\\
&&+\left(n-1\right)\int_{x}^{\infty}\varphi\left(y\right) u^{2}\left(t,y\right)r^{-n}\left(t,y\right)dy.\nonumber
\end{eqnarray}

Noticing the definition of $B\left(t,x\right)$ and $Q\left(t,x\right)$, we integrate (\ref{b26}) over $\left(0,t\right)$ with respect to $t$ and take the exponential on both sides of the resulting equation to deduce that
\begin{eqnarray}\label{b27}
\frac{1}{B\left(t,x\right)Q\left(t, x\right)}=\frac{1}{v\left(t,x\right)}\exp\left\{\frac{1}{\alpha}\int_{0}^{t}P\left(s,x\right)ds\right\},\quad x\in\overline{\Omega}_{k},\  t\geq 0.
\end{eqnarray}

Multiplying (\ref{b27}) by $\frac{P\left(t,x\right)}{\alpha}$ and integrating over $(0,t)$, we can infer that
\begin{eqnarray}\label{b28}
\exp\left\{\frac{1}{\alpha}\int_{0}^{t}P\left(s,x\right)ds\right\}=1+\frac{1}{\alpha}\int_{0}^{t}\frac{v\left(s,x\right)P\left(s,x\right)}{B\left(s,x\right)Q\left(s,x\right)}ds.
\end{eqnarray}

Combining (\ref{b27}) with (\ref{b28}), we obtain (\ref{b24}). This completes the proof of Lemma 3.3.
\end{proof}

To deduce the desired positive lower and upper bounds of $v\left(t,x\right)$ which are independent of the time variable $t$, one can first prove by repeating the argument used by Jiang in \cite{Jiang-CMP-1999} that the following estimates
\begin{eqnarray}\label{b29}
C\left(k\right)^{-1}\leq B\left(t,x\right)\leq C\left(k\right), \quad\forall x\in\overline{\Omega}_{k}
\end{eqnarray}
and
\begin{eqnarray}\label{b30}
-\int_{s}^{t}\inf\limits_{x\in[k+1,k+2]}\theta\left(\tau,\cdot\right)d\tau\leq
 \begin{cases}
 0, &  0\leq t-s\leq 1 ,\\
  -C(t-s), &  t-s\geq 1
 \end{cases}
\end{eqnarray}
hold true for $0\leq s\leq t\leq T$. Consequently, one can get from H\"{o}lder's inequality and Jenssen's inequality $\left(\int_{k+1}^{k+2}vdx\right)^{-1}\leq \int^{k+2}_{k+2}v^{-1}dx$ that
\begin{eqnarray}\label{b31}
&&\int_{s}^{t}\int_{k+1}^{k+2}\left(\frac{\alpha\left(r^{n-1}u\right)_{x}}{v}-\frac{R\theta}{v}\right)dxd\tau\nonumber\\
&\leq& C\int_{s}^{t}\int_{k+1}^{k+2}\frac{\left|\left(r^{n-1}u\right)_{x}\right|^{2}}{v\theta}dxd\tau
-\frac{R}{2}\int_{s}^{t}\int_{k+1}^{k+2}\frac{\theta}{v}dxd\tau\nonumber\\
&\leq& C-\frac{R}{2}\int_{s}^{t}\inf\limits_{x\in[k+1,k+2]}\theta\left(\tau,\cdot\right)\int_{k+1}^{k+2}\frac{1}{v}dxd\tau\\
&\leq& C-\frac{R}{2}\int_{s}^{t}\inf\limits_{x\in[k+1,k+2]}\theta\left(\tau,\cdot\right)\left(\int_{k+1}^{k+2}vdx\right)^{-1}d\tau\nonumber\\
&\leq& C-\frac{R}{2a_{2}}\int_{s}^{t}\inf\limits_{x\in[k+1,k+2]}\theta\left(\tau,\cdot\right)d\tau\nonumber\\
&\leq& C-C(t-s),\quad 0\leq s\leq t\leq T.\nonumber
\end{eqnarray}

From the definition of $Q\left(t,x\right)$ and (\ref{b31}), we have for $0\leq s\leq t\leq T$ that
\begin{eqnarray}\label{b32}
\frac{Q\left(t,x\right)}{Q\left(s,x\right)}&&=\frac{\exp\left\{\frac{1}{\alpha}{\displaystyle\int_{s}^{t}\int_{k+1}^{k+2}} \left(\frac{\alpha \left(r^{n-1}u\right)_{x}}{v}-\frac{R\theta}{v}\right)dyd\tau\right\}}
{\exp\left\{\frac{1}{\alpha}{\displaystyle\int_{s}^{t}\int_{k+1}^{k+2}}\frac{a}{3}\theta^{4}dyd\tau\right\}\exp\left\{\frac{1}{\alpha}{\displaystyle\int_{s}^{t}\int_{x}^{\infty}}(n-1)\phi u^{2}r^{-n}dyd\tau\right\}}\nonumber\\
&&\leq\exp\left\{\frac{1}{\alpha}\int_{s}^{t}\int_{k+1}^{k+2}\left(\frac{\alpha\left(r^{n-1}u\right)_{x}}{v}- \frac{R\theta}{v}\right)dyd\tau\right\}\\
&&\leq C\exp\left\{-C(t-s)\right\}.\nonumber
 \end{eqnarray}

If we set $s=0$ in (\ref{b32}), we can deduce that
\begin{eqnarray}\label{b33}
0\leq Q\left(t,x\right)\leq C\exp\left\{-Ct\right\}.
\end{eqnarray}

For $x\in\overline{\Omega}_{k}, t\geq 0$, we now turn to deduce a upper bound of $v(t,x)$. To this end, since
\begin{eqnarray}\label{b39}
v\left(t,x\right)&&\leq CQ\left(t,x\right)+C\int_{0}^{t}\frac{Q\left(t,x\right)}{Q\left(s,x\right)}v\left(s, x\right)P\left(s, x\right)ds\nonumber\\
&&\leq C+C\int_{0}^{t}\exp\left\{-C(t-s)\right\}\left(\theta+v\theta^{4}\right)\left(s, x\right)ds,
 \end{eqnarray}
and by noticing that the following two estimates
\begin{eqnarray}\label{b40}
\int_{0}^{t}\exp\left\{-C(t-s)\right\}\theta\left(s, x\right)ds\leq C\int_{0}^{t}\exp\left\{-C(t-s)\right\}(1+V\left(s\right))ds\leq C
\end{eqnarray}
and
\begin{eqnarray}\label{b41}
\int_{0}^{t}\exp\left\{-C(t-s)\right\}v(s,x)\theta^{4}\left(s, x\right)ds\leq C\int_{0}^{t}v\left(s,x\right)\exp\left\{-C(t-s)\right\}(1+V\left(s\right))ds
\end{eqnarray}
hold for $x\in\overline{\Omega}_{k}, t\geq 0$, we can get from (\ref{b39})-(\ref{b41}) and by using Gronwall's inequality that the following estimate
\begin{eqnarray}\label{b42}
v\left(t,x\right)\leq C\left(k\right)
\end{eqnarray}
holds for some positive constant $C(k)$. Here we emphasize that we have used Lemma 3.2 with $m=\frac{1}{2}$ and $m=2$ in deducing \eqref{b41}.

Now we turn to deduce a positive lower bound of $v\left(t,x\right)$ for the case when the time variable $t$ is sufficiently large. For this purpose, by Jensen's inequality, \eqref{a23}, \eqref{a25} and \eqref{b42}, we can deduce that
\begin{eqnarray}\label{bb42}
&&\int_{k}^{k+1}\theta(t,r)dr-\ln\int_{k}^{k+1}\theta(t,r) dr-1\nonumber\\
&\leq&\int_{k}^{k+1}\left(\theta(t,r)-\ln\theta(t,r)-1\right)dr\nonumber\\
&\leq& C\int_{k}^{k+1}\left(\theta(t,r)-\ln\theta(t,r)-1\right)\frac{r^{n-1}}{v}dr\nonumber\\
&\leq& C\int_{\Omega}\left(\theta(t,x)-\ln\theta(t,x)-1\right)dx\nonumber\\
&\leq& C,\nonumber
\end{eqnarray}
thus
\begin{eqnarray}\label{bb43}
C^{-1}\leq\int_{k}^{k+1}\theta(t,r) dr\leq C.
\end{eqnarray}

Noticing \eqref{a23} and
\begin{eqnarray}\label{bb44}
\left|\left(r^{-\frac{n}{2}}u\right)\left(t,\cdot\right)\right|\leq\int_{0}^{\infty}\left(\left|r^{-\frac{n}{2}}u_{x}\right|
+\left|-\frac{n}{2}r^{-\frac{3n}{2}}vu\right|\right)dx,\nonumber
\end{eqnarray}
we can get by H\"{o}lder's inequality that
\begin{eqnarray}\label{bb45}
\left\|\frac{u}{r^{\frac{n}{2}}}\right\|_{L^{\infty}\left([0,\infty)\right)}^{2}&\leq& C\left(\int_{0}^{\infty}\frac{r^{2(n-1)}u^{2}_{x}}{v\theta}dx\right)
\left(\int_{0}^{\infty}\frac{v\theta}{r^{3n-2}}dx\right)\nonumber\\
&&+C\left(\int_{0}^{\infty}\frac{vu^{2}}{r^{2}\theta}dx\right)\left(\int_{0}^{\infty}\frac{v\theta}{r^{3n-2}}dx\right).
\end{eqnarray}

\eqref{a23} and \eqref{bb43} tell us that
\begin{eqnarray}\label{bb46}
\int_{0}^{\infty}\frac{v\theta}{r^{3n-2}}dx=\int_{0}^{\infty}\frac{\theta}{r^{2n-1}}\frac{v}{r^{n-1}}dx=\int_{1}^{\infty}\frac{\theta}{r^{2n-1}}dr
\leq\sum_{k=1}^{\infty}\frac{1}{k^{2n-1}}\int_{k}^{k+1}\theta dr\leq C
\end{eqnarray}
holds for $n\geq 2$, therefore, thanks to \eqref{bb13}, \eqref{bb45} and \eqref{bb46}, we can deduce that
\begin{eqnarray}\label{bbb47}
\int_{0}^{t}\left\|\frac{u}{r^{\frac{n}{2}}}\right\|_{L^{\infty}\left([0,\infty)\right)}^{2}ds\leq C
\end{eqnarray}
holds for $n\geq 2$ and  consequently
\begin{eqnarray}\label{bb47}
\left|\int_{0}^{t}\int_{x}^{\infty}\varphi r^{-n}u^{2}dyds\right|\leq C\int_{0}^{t}\left\|\frac{u}{r^{\frac{n}{2}}}\right\|_{L^{\infty}\left([0,\infty)\right)}^{2}ds\leq C
\end{eqnarray}
holds for $n\geq 2$.

On the other hand, one can get from $\eqref{a20}_{1}$ that
\begin{eqnarray}\label{bb48}
\int_{0}^{t}\int_{k+1}^{k+2}\frac{\left(r^{n-1}u\right)_{x}}{v}dxds=\int_{0}^{t}\int_{k+1}^{k+2}\frac{v_{t}}{v}dxds=\int_{k+1}^{k+2}\ln\frac{v}{v_{0}}dx,
\end{eqnarray}
which together with \eqref{b17} imply
\begin{eqnarray}\label{bb49}
-C_{2}\leq-C_{1}-\int_{k+1}^{k+2}\left(v-\ln v-1\right)dx\leq\int_{k+1}^{k+2}\ln\frac{v}{v_{0}}dx\leq C_{1}+\ln\int_{k+1}^{k+2}vdx\leq C_{2}.
\end{eqnarray}
Thus we can infer from \eqref{bb48} and \eqref{bb49} that
\begin{eqnarray}\label{bb50}
\left|\int_{0}^{t}\int_{k+1}^{k+2}\frac{\left(r^{n-1}u\right)_{x}}{v}dxds\right|\leq C,
\end{eqnarray}
and consequently we can deduce from \eqref{bb47} and \eqref{bb50} that
\begin{eqnarray}\label{bb51}
C^{-1}&\leq&\exp\left[\int_{s}^{t}\int_{k+1}^{k+2}\frac{\left(r^{n-1}u\right)_{x}}{v}dxd\tau\right]\leq C,\nonumber\\ C^{-1}&\leq&\exp\left[\frac{1}{\alpha}\int_{s}^{t}\int_{x}^{\infty}\left(n-1\right)\varphi r^{-n}u^{2}dyd\tau\right]\leq C.
\end{eqnarray}

If we define
\begin{eqnarray}\label{bb52}
\widetilde{Q}\left(t\right)=\exp\left[\frac{1}{\alpha}\int_{0}^{t}\int_{k+1}^{k+2}\left(-\frac{R\theta}{v}-\frac{a}{3}\theta^{4}\right)dxds\right],
\end{eqnarray}
then we can infer from \eqref{b25} that
\begin{eqnarray}\label{bb53}
\frac{Q\left(t,x\right)}{Q\left(s,x\right)}=\frac{\widetilde{Q}\left(t\right)}{\widetilde{Q}\left(s\right)}
\cdot\frac{\exp\left[\displaystyle\int_{s}^{t}\int_{k+1}^{k+2}\frac{\left(r^{n-1}u\right)_{x}}{v}dxd\tau\right]}
{\exp\left[\frac{1}{\alpha}\displaystyle\int_{s}^{t}\int_{x}^{\infty}\left(n-1\right)\varphi r^{-n}u^{2}dyd\tau\right]},
\end{eqnarray}
and we can get from \eqref{bb51} and \eqref{bb53} that
\begin{eqnarray}\label{bb54}
C^{-1}\frac{\widetilde{Q}\left(t\right)}{\widetilde{Q}\left(s\right)}\leq\frac{Q\left(t,x\right)}{Q\left(s,x\right)}\leq C\frac{\widetilde{Q}\left(t\right)}{\widetilde{Q}\left(s\right)}.
\end{eqnarray}

On the other hand, integrating (\ref{b24}) with respect to $x$ over $\Omega_k=(k,k+1)$ and by using \eqref{b3}, \eqref{b17}, (\ref{b29}) and (\ref{b33}), we can conclude that
\begin{eqnarray}\label{b34}
a_{1}&&\leq C\int_{k}^{k+1}Q\left(t,x\right)dx+C\int_{0}^{t}\int_{k}^{k+1}\frac{Q\left(t,x\right)}{Q\left(s,x\right)}P\left(s,x\right)v\left(s,x\right)dxds\nonumber\\
&&\leq C\exp\left\{-Ct\right\}+C\int_{0}^{t}\frac{\widetilde{Q}\left(t\right)}{\widetilde{Q}\left(s\right)}\int_{k}^{k+1} \left(\theta+v\theta^{4}\right)(s,x)dxds\\
&&\leq C\exp\left\{-Ct\right\}+C\int_{0}^{t}\frac{\widetilde{Q}\left(t\right)} {\widetilde{Q}\left(s\right)}ds.\nonumber
\end{eqnarray}

Furthermore, setting $m=\frac{1}{2}$ in (\ref{b19}) and by using (\ref{b18}), we can derive that
\begin{eqnarray}
a^{\frac{1}{2}}_{1}-CV^{\frac{1}{2}}\left(t\right)\leq\theta^{\frac{1}{2}}\left(t,x\right),\nonumber
 \end{eqnarray}
which implies
\begin{eqnarray}\label{b35}
\theta\left(t,x\right)\geq\frac{a_{1}}{2}-CV\left(t\right).
 \end{eqnarray}
Thus we can conclude from \eqref{b24}, (\ref{b32}), \eqref{bb54}, (\ref{b34}) and (\ref{b35}) that for $x\in\Omega_k, 0\leq t\leq T$
\begin{eqnarray}\label{b36}
v\left(t,x\right)&&\geq C\int_{0}^{t}\frac{Q\left(t,x\right)}{Q\left(s,x\right)}\theta\left(s,x\right)ds\nonumber\\
&&\geq C\int_{0}^{t}\frac{Q\left(t,x\right)}{Q\left(s,x\right)}\left(\frac{a_{1}}{2}-CV\left(s\right)\right)ds\\
&&\geq C\int_{0}^{t}\frac{\widetilde{Q}\left(t\right)}{\widetilde{Q}\left(s\right)}ds
-C\int_{0}^{t}\frac{Q\left(t,x\right)}{Q\left(s,x\right)}V\left(s\right)ds,\nonumber\\
&&\geq C-\exp\left(-Ct\right)-C\int_{0}^{t}\frac{Q\left(t,x\right)}{Q\left(s,x\right)}V\left(s\right)ds.\nonumber
\end{eqnarray}
While
\begin{eqnarray}\label{b37}
&&\int_{0}^{t}\frac{Q\left(t,x\right)}{Q\left(s,x\right)}V\left(s\right)ds\nonumber\\
&\leq& C\int_{0}^{t}\exp\left\{-C(t-s)\right\}V\left(s\right)ds\\
&=&C\left(\int_{0}^{\frac{t}{2}}\exp\left\{-C(t-s)\right\}V\left(s\right)ds
+\int_{\frac{t}{2}}^{t}\exp\left\{-C(t-s)\right\}V\left(s\right)ds\right)\nonumber\\
&\leq& C\left(\exp\left\{-\frac{Ct}{2}\right\}\int_{0}^{\frac{t}{2}}V\left(s\right)ds +\int_{\frac{t}{2}}^{t}V\left(s\right)ds\right)\rightarrow 0\quad as\; t\rightarrow +\infty,\nonumber
\end{eqnarray}
then putting (\ref{b36}) and (\ref{b37}) together, we can deduce that there exist positive constants $t_{0}$ and $C$ such that if $t\geq t_{0}$,
\begin{eqnarray}\label{bz38}
v\left(t,x\right)\geq C,\quad \forall t\geq t_0,\ \ x\in\overline{\Omega}_k.
\end{eqnarray}

Having obtained \eqref{bz38}, to deduce the desired uniform positive lower bound on $v(t,x)$, it suffices to deduce the lower bound of $v\left(t,x\right)$ for $0<t\leq t_{0}$. For this purpose, noticing that $(\ref{a1})_{2}$ can be rewritten as
\begin{eqnarray}\label{bzz38}
-\alpha\left[\log v\right]_{xt}+P_{x}=-r^{1-n}u_{t},
 \end{eqnarray}
as in \cite{Kazhikhov-Shelukhin-JAMM-1977}, we can get by integrating (\ref{bzz38}) over $\left(0,t\right)\times\left(a_{k}\left(t\right), x\right)$ with respect to $t$ and $x$ that
\begin{eqnarray}\label{bz39}
&&-\alpha\log v\left(t, x\right)+\int_{0}^{t}P\left(s, x\right)ds\nonumber\\
&=&-\int_{0}^{t}\int_{a_{k}\left(t\right)}^{x}r^{1-n}u_{t}dyds
+\int_{0}^{t}P\left(s, a_{k}\left(t\right)\right)ds\\
&&+\alpha\log\frac{v_{0}\left(a_{k}\left(t\right)\right)}{v_{0} \left(x\right)v\left(t,a_{k}\left(t\right)\right)}.\nonumber
\end{eqnarray}
Taking the exponential on both sides of the resulting equation, we obtain
\begin{eqnarray}\label{bz40}
&&\frac{1}{v\left(t,x\right)}\exp\left\{\frac{1}{\alpha}\int_{0}^{t}P\left(s, x\right)ds\right\}\nonumber\\
&=&\frac{v_{0}\left(a_{k}\left(t\right)\right)\exp\left\{-\frac{1}{\alpha}{\displaystyle\int_{0}^{t}\int_{a_{k}(t)}^{x}}r^{1-n}u_{t}dyds\right\}}
{v\left(t,a_{k}\left(t\right)\right)v_{0}\left(x\right)}
\exp\left\{\frac{1}{\alpha}{\displaystyle\int_{0}^{t}}P\left(s,a_{k}\left(t\right)\right)ds\right\}\\
&:=&B_{k}\left(t,x\right)Q_{k}\left(t\right).\nonumber
\end{eqnarray}
Obviously, we have
\begin{eqnarray}\label{bz41}
C^{-1}\left(k\right)\leq B_{k}\left(t,x\right)\leq C\left(k\right).
\end{eqnarray}
Multiplying (\ref{bz40}) by $\frac{P\left(t,x\right)}{\alpha}$ and integrating over $(0,t)$, gives
\begin{eqnarray}\label{bz42}
\exp\left\{\frac{1}{\alpha}\int_{0}^{t}P\left(s,x\right)ds\right\}=1+\frac{1}{\alpha}\int_{0}^{t}v\left(s,x\right)P\left(s,x\right)B_{k}\left(s,x\right)Q_{k}\left(s\right)ds.
\end{eqnarray}
Combining (\ref{bz40}) with (\ref{bz42}), we arrive at
\begin{eqnarray}\label{bz43}
v\left(t,x\right)Q_{k}\left(t\right)=B^{-1}_{k}\left(t,x\right)\left(1 +\frac{1}{\alpha}\int_{0}^{t}v\left(s,x\right)P\left(s,x\right)B_{k}\left(s,x\right)Q_{k}\left(s\right)ds\right).
\end{eqnarray}
Integrating (\ref{bz43}) with respect to $x$ over $\left(k,k+1\right)$ and by using \eqref{b3}, \eqref{b17} and (\ref{bz41}), one has
\begin{equation}\label{bz44}
Q_{k}\left(t\right)\leq C\left(1+\int_{0}^{t}\left(\int_{k}^{k+1}v\left(s,x\right)P\left(s,x\right)dx\right)Q_{k}\left(s\right)ds\right).\nonumber
\end{equation}
With the help of Gronwall's inequality, we have for $0\leq t\leq t_0$ that
\begin{eqnarray}\label{bz45}
Q_{k}\left(t\right)&&\leq C\exp\left\{\int_{0}^{t}\int_{k}^{k+1}v\left(s,x\right)P\left(s,x\right)dxds\right\}\nonumber\\
&&\leq C\exp\left\{\int_{0}^{t}\int_{k}^{k+1}\left(R\theta+\frac{av\theta^{4}}{3}\right)\left(s, x\right)dxds\right\}\\
&&\leq C\exp\left\{Ct_{0}\right\}.\nonumber
\end{eqnarray}

Combining (\ref{bz40}), (\ref{bz41}) and (\ref{bz45}), we can deduce that there exists a positive constant $C(k,t_0)$ depending only on $k$ and $t_0$ such that the following estimate
\begin{eqnarray}\label{b38}
v\left(t,x\right)\geq C(k,t_0)
\end{eqnarray}
holds for $0\leq t\leq t_0$ and $x\in\overline{\Omega}_k$ and the estimates \eqref{bz38} and \eqref{b38} imply that
\begin{equation}\label{Lower-bound-on-v}
v(t,x)\geq C(k)>0
\end{equation}
holds for $0\leq t\leq T, x\in\overline{\Omega}_k$ and some positive constant $C(k)$ depending only on $k$.

So far, for each $x\in\overline{\Omega}_k$ and $0\leq t\leq T$, we have deduced a uniform positive lower bound and the upper bound of $v\left(t,x\right)$ which depend only on $k$ and the initial data but are independent of the time variable $t$. Such an estimate together with the fact $v\left(t,x\right)-1\in C\left(0,T;H^{1}\left([0,\infty)\right)\right)$ tell us that
\begin{lemma} Under the assumptions listed in Lemma 2.1, there exist positive constants $\underline{V}$, $\overline{V}$, which depend only on the fixed constants $\mu$, $\lambda_{1}$, $\lambda$, $K$, $A$, $d$, $R$, $c_{v}$, $a$, $\kappa_{1}$, $\kappa_{2}$, $n$ and the initial data $(v_0(x), u_0(x), \theta_0(x), z_0(x))$, such that
\begin{eqnarray}\label{b43}
 \underline{V}\leq v\left(t,x\right)\leq \overline{V}
\end{eqnarray}
holds for all $(t,x)\in[0,T]\times[0,\infty)$.
\end{lemma}

Inequalities \eqref{a24} and \eqref{b43} tell us that
\begin{corollary} It holds for all $(t,x)\in[0,T]\times[0,\infty)$ that
\begin{eqnarray}\label{b44}
 1+\frac{x}{C}\leq r^{n}\left(t,x\right)\leq 1+Cx.
\end{eqnarray}
\end{corollary}

The next lemma will give a nice bound on $\int_{0}^{t}\left\|\frac{u}{r}\right\|^{2}_{L^{\infty}\left([0,\infty)\right)}ds$, which is essential for our analysis.
\begin{lemma} Under the assumptions listed in Lemma 2.1, we have
\begin{eqnarray}\label{cc16}
 \int_{0}^{t}\left\|\frac{u}{r}\right\|^{2}_{L^{\infty}\left([0,\infty)\right)}ds\leq C
\end{eqnarray}
holds for all $t\in[0,T]$ and $n\geq 2$.
\end{lemma}
\begin{proof} For the case when $n=2$, the estimate \eqref{cc16} follows immediately from
\eqref{bbb47}. As for the case of $n\geq 3$, since
\begin{eqnarray}\label{c18}
u^{2}=\left(\int_{x}^{\infty}u_{x}dx\right)^{2}\leq C\int_{0}^{\infty}\frac{r^{2n-2}u_{x}^{2}}{v\theta}dx\int_{0}^{\infty}\frac{v\theta}{r^{2n-2}}dx,\nonumber
\end{eqnarray}
and noticing that \eqref{bb43} tells us that
\begin{eqnarray}\label{c19}
\int_{0}^{\infty}\frac{v\theta}{r^{2n-2}}dx=\sum_{k=1}^{\infty}\int_{k}^{k+1}\frac{\theta}{r^{n-1}}dr
\leq\sum_{k=1}^{\infty}\frac{1}{k^{n-1}}\int_{k}^{k+1}\theta dr\leq C\sum_{k=1}^{\infty}\frac{1}{k^{n-1}}\leq C\nonumber
\end{eqnarray}
holds for $n\geq 3$, we can get by combining the above inequalities with \eqref{bb13} that
\begin{eqnarray}\label{c16}
 \int_{0}^{t}\left\|u\right\|^{2}_{L^{\infty}\left([0,\infty)\right)}ds\leq C
\end{eqnarray}
holds for $n\geq 3$. Consequently, for $n\geq 3$, we can deduce from \eqref{a25} and \eqref{c16} that
\begin{eqnarray*}
 \int_{0}^{t}\left\|\frac{u}{r}\right\|^{2}_{L^{\infty}\left([0,\infty)\right)}ds\leq\int_{0}^{t}\left\|u\right\|^{2}_{L^{\infty}\left([0,\infty)\right)}ds\leq C.
\end{eqnarray*}
This completes the proof of our lemma.
\end{proof}

To deduce the upper bound on $\theta(t,x)$, we first deduce a bound on $\|v_{x}(t)\|$ in terms of $\|\theta\|_{\infty}$, which will be used later on. In fact, we have the following lemma:
\begin{lemma} Under the assumptions listed in Lemma 2.1 , we have for any $0\leq t\leq T$ that
\begin{eqnarray}\label{bz57}
\left\|v_{x}(t)\right\|^2+\int_{0}^{t}\left\|\sqrt{\theta(s)} v_{x}(s)\right\|^2 ds\leq C+C\|\theta\|^{l_{1}}_{\infty},
\end{eqnarray}
where
\begin{eqnarray}\label{bd58}
l_{1}=\max\{1,(7-b)_{+}\}.
\end{eqnarray}
\end{lemma}
\begin{proof}
Noticing that $\eqref{a20}_{2}$ can be rewritten as
\begin{eqnarray}\label{bd44}
r^{1-n}u_{t}+P_{x}=\alpha\left(\frac{v_{t}}{v}\right)_{x}=\alpha\left(\frac{v_{x}}{v}\right)_{t},
\end{eqnarray}
and
\begin{eqnarray}\label{bbb44}
 \frac{\partial P(v,\theta)}{\partial x}= \left(\frac{R\theta}{v}+\frac{a\theta^{4}}{3}\right)_{x}=\frac{R\theta_{x}}{v}-\frac{R\theta v_{x}}{v^{2}}
 +\frac{4}{3}a\theta^{3}\theta_{x},
\end{eqnarray}
we can get by multiplying \eqref{bd44} by $\frac{v_{x}}{v}$ and integrating the resulting identity with respect to $t$ and $x$ over $\left[0,t\right)\times[0,\infty)$ that
\begin{eqnarray}\label{b45}
&&\frac{\alpha}{2}\int_{0}^{\infty}\frac{v_{x}^{2}}{v^{2}}dx +\int_{0}^{t}\int_{0}^{\infty}\frac{R\theta v_{x}^{2}}{v^{3}}dxds\nonumber\\
&=&\frac{\alpha}{2}\int_{0}^{\infty}\frac{v_{0x}^2}{v_0^2}dx +\underbrace{\int^t_0\int_{0}^{\infty}\frac{R\theta_{x}v_x}{v^{2}}dxds}_{I_1}
+\underbrace{\frac{4a}{3}\int^t_0\int_{0}^{\infty}\frac{\theta^3v_x\theta_x}{v}dxds}_{I_2}\\
&&+R\underbrace{\int^t_0\int_{0}^{\infty}\frac{r^{1-n}v_{x}u_{t}}{v}dxds}_{I_3}.\nonumber
\end{eqnarray}

Now we deal with $I_k (1\leq k\leq 3)$ term by term. To this end, we first get by employing Cauchy's inequality, \eqref{a25}, \eqref{a8} and Lemma 2.1 that
\begin{eqnarray}\label{b46}
I_{1}&&\leq\epsilon\int_{0}^{t}\int_{0}^{\infty}\frac{R\theta v_{x}^{2}}{v^{3}}dxds
+C\left(\epsilon\right)\int_{0}^{t}\int_{0}^{\infty}\frac{\kappa r^{2(n-1)}\theta^{2}_{x}}{v\theta^{2}}\cdot\frac{\theta}{\kappa r^{2(n-1)}} dx\nonumber\\
&&\leq\epsilon\int_{0}^{t}\int_{0}^{\infty}\frac{R\theta v_{x}^{2}}{v^{3}}dxds
+C\left(\epsilon\right)
\end{eqnarray}
and
\begin{eqnarray}\label{b49}
I_{2}&&\leq\epsilon\int_{0}^{t}\int_{0}^{\infty}\frac{R\theta v_{x}^{2}}{v^{3}}dxds
+C\left(\epsilon\right)\int_{0}^{t}\int_{0}^{\infty}\frac{\kappa(v,\theta)r^{2(n-1)}\theta^{2}_{x}} {v\theta^{2}}\cdot\frac{\theta^{7}}{\kappa(v,\theta)r^{2(n-1)}}dxds\nonumber\\
&&\leq\epsilon\int_{0}^{t}\int_{0}^{\infty}\frac{R\theta v_{x}^{2}}{v^{3}}dxds
+C\left(\epsilon\right)\int_{0}^{t}\int_{0}^{\infty}\frac{\kappa(v,\theta)\theta^{2}_{x}}{v\theta^{2}}\cdot\frac{\theta^{7}}{1+\theta^{b}}dxds\\
&&\leq\epsilon\int_{0}^{t}\int_{0}^{\infty}\frac{R\theta v_{x}^{2}}{v^{3}}dxds
+C\left(\epsilon\right)\|\theta\|^{(7-b)_{+}}_{\infty},\nonumber
\end{eqnarray}
where $(7-b)_{+}:=\max\{0, 7-b\}$.

As to the term $I_3$, we get by integration by parts and from the boundary condition \eqref{a22} that
\begin{eqnarray}\label{b51}
I_{3}&=&\int_{0}^{\infty}\frac{r^{1-n}uv_{x}}{v}dx -\int_{0}^{\infty}\frac{r^{1-n}uv_{x}}{v}\left(0,x\right)dx\nonumber\\
&&-\int_{0}^{t}\int_{0}^{\infty}r^{1-n}u\left(\frac{v_{x}}{v}\right)_{t}dxds\\
&&+\left(n-1\right)\int_{0}^{t}\int_{0}^{\infty}r^{-n}u^{2}\frac{v_{x}}{v}dxds.\nonumber
\end{eqnarray}
For the first term in the right hand side of \eqref{b51}, we can deduce from \eqref{a25} that
\begin{eqnarray}\label{b52}
\int_{0}^{\infty}\frac{r^{1-n}uv_{x}}{v}dx &&\leq\epsilon\int_{0}^{\infty}\frac{v^{2}_{x}}{v^{2}}dx +C\left(\epsilon\right)\int_{0}^{\infty} r^{2(1-n)}u^{2}dx\nonumber\\
&&\leq\epsilon\int_{0}^{\infty}\frac{v^{2}_{x}}{v^{2}}dx+C\left(\epsilon\right).
\end{eqnarray}

Moreover, for the third term in the right hand side of \eqref{b51}, we can deduce that
\begin{eqnarray}\label{b53}
&&-\int_{0}^{t}\int_{0}^{\infty}r^{1-n}u\left(\frac{v_{x}}{v}\right)_{t}dxds\nonumber\\
&=&-\int_{0}^{t}\int_{0}^{\infty}r^{1-n}u\left(\frac{v_{t}}{v}\right)_{x}dxds\nonumber\\
&=&\int_{0}^{t}\int_{0}^{\infty}\frac{\left(r^{n-1}u\right)_{x}\left(r^{1-n}u\right)_{x}}{v}dxds\nonumber\\
&=&\int_{0}^{t}\int_{0}^{\infty}\frac{r^{2(1-n)}\left|\left(r^{n-1}u\right)_{x}\right|^{2}}{v}dxds
+2\left(1-n\right)\int_{0}^{t}\int_{0}^{\infty}r^{1-2n}u\left(r^{n-1}u\right)_{x}dxds\\
&\leq& C\int_{0}^{t}\int_{0}^{\infty}\frac{r^{2(1-n)}\left|\left(r^{n-1}u\right)_{x}\right|^{2}}{v\theta}\cdot\theta dxds
+C\int_{0}^{t}\int_{0}^{\infty}\frac{vu^{2}}{r^{2}\theta}\cdot\theta dxds\nonumber\\
&&+\int_{0}^{t}\int_{0}^{\infty}\frac{r^{2(1-n)}\left|\left(r^{n-1}u\right)_{x}\right|^{2}}{v\theta}\cdot\theta dxds\nonumber\\
&\leq& C\|\theta\|_{\infty},\nonumber
\end{eqnarray}
where we have used $\eqref{a20}_{1}$ and the fact that
\begin{eqnarray}\label{bbb54}
\left(r^{1-n}u\right)_{x}=r^{2(1-n)}\left(r^{n-1}u\right)_{x}+2\left(1-n\right)r^{1-2n}vu.
\end{eqnarray}

Furthermore, for the fourth term in the right hand side of \eqref{b51}, we have from \eqref{a25} and \eqref{bb13} that
\begin{eqnarray}\label{bc55}
\int_{0}^{t}\int_{0}^{\infty}r^{-n}u^{2}\frac{v_{x}}{v}dxds &&\leq C\int_{0}^{t}\int_{0}^{\infty}\frac{u^{2}}{r^{2}}dxds +C\int_{0}^{t}\int_{0}^{\infty}\frac{u^{2}}{r^{2n-2}}\cdot\frac{v^{2}_{x}}{v^{2}}dxds\nonumber\\
&&\leq C\left\|\theta\right\|_{\infty}+C \int_{0}^{t}\left\|\frac{u}{r}\right\|^{2}_{L^{\infty}\left([0,\infty)\right)}\int_{0}^{\infty}\frac{v^{2}_{x}}{v^{2}}dxds.
\end{eqnarray}
Thus we can conclude from \eqref{b51}-\eqref{bc55} that
\begin{eqnarray}\label{bcb58}
I_{3}\leq\epsilon\int_{0}^{\infty}\frac{v^{2}_{x}}{v^{2}}dx +C\left(\epsilon\right)+C\left\|\theta\right\|_{\infty}
+C\int_{0}^{t}\left\|\frac{u}{r}\right\|^{2}_{L^{\infty}\left([0,\infty)\right)}\int_{0}^{\infty}\frac{v^{2}_{x}}{v^{2}}dxds.
\end{eqnarray}

Inserting the above estimates on $I_k (k=1,2,3)$ into \eqref{b45} and then using \eqref{cc16} and Gronwall's inequality, we can deduce \eqref{bz57}. This completes the proof of Lemma 3.7.
\end{proof}

The next lemma is concerned with the bound on the term $\int_{0}^{t}\int_{0}^{\infty}\frac{r^{2(n-1)}u^{2}_{xx}}{v}dxds$.
\begin{lemma} Under the assumptions listed in Lemma 2.1, we have for all $0\leq t\leq T$ that
\begin{eqnarray}\label{c1}
\int_{0}^{\infty}u_{x}^{2}dx+\int_{0}^{t}\int_{0}^{\infty}\frac{r^{2(n-1)}u^{2}_{xx}}{v}dxds\leq C+C\left\|\theta\right\|^{l_{2}}_{\infty},
\end{eqnarray}
where
\begin{eqnarray}\label{cc1}
l_{2}=\max\left\{2l_{1}+1, \left(8-b\right)_{+}\right\}.
\end{eqnarray}
\end{lemma}
\begin{proof} Firstly, we have by multiplying $\eqref{a20}_{2}$ by $-u_{xx}$ that
\begin{eqnarray}\label{c2}
&&\frac{1}{2}\partial_{t}\left(u_{x}^{2}\right)+\frac{\alpha r^{2(n-1)}u^{2}_{xx}}{v}\nonumber\\
&=&\left(u_{x}u_{t}\right)_{x}+\alpha u_{xx}\left(\frac{r^{2(n-1)}u_{x}v_{x}}{v^{2}}+(n-1)\frac{uv}{r^{2}}-2(n-1)r^{n-2}u_{x}\right)\\
&&+Ru_{xx}r^{n-1}\left(\frac{\theta_{x}}{v}-\frac{\theta v_{x}}{v^{2}}\right)+u_{xx}r^{n-1}\cdot\frac{4}{3}a\theta^{3}\theta_{x}.\nonumber
\end{eqnarray}

Integrating \eqref{c2} with respect to $t$ and $x$ over $\left[0,t\right)\times[0,\infty)$, we can get by using the boundary conditions \eqref{a22} and Cauchy's inequality that
\begin{eqnarray}\label{c3}
&&\frac{1}{2}\int_{0}^{\infty}u_{x}^{2}dx+\int_{0}^{t}\int_{0}^{\infty}\frac{\alpha r^{2(n-1)}u^{2}_{xx}}{v}dxds\nonumber\\
&\leq& C+\frac{1}{6}\int_{0}^{t}\int_{0}^{\infty}\frac{\alpha r^{2(n-1)}u^{2}_{xx}}{v}dxds\nonumber\\
&&+C\int_{0}^{t}\int_{0}^{\infty}\left(r^{2(n-1)}u_{x}^{2}v_{x}^{2}+\frac{u^{2}}{r^{2(n+1)}}+\frac{u^{2}_{x}}{r^{2}}+\theta^{2}_{x} +\theta^{2}v_{x}^{2}+\theta^{6}\theta^{2}_{x}\right)dxds\\
&\leq& C+\frac{1}{6}\int_{0}^{t}\int_{0}^{\infty}\frac{\alpha r^{2(n-1)}u^{2}_{xx}}{v}dxds+C\left\|\theta\right\|_{\infty}+C\left\|\theta\right\|^{1+l_{1}}_{\infty}
\nonumber\\
&&+C\left\|\theta\right\|^{(8-b)_{+}}_{\infty} +C\int_{0}^{t}\int_{0}^{\infty}r^{2(n-1)}u_{x}^{2}v_{x}^{2}dxds.\nonumber
\end{eqnarray}
On the other hand, one can infer from \eqref{bz57} that
\begin{eqnarray}\label{c4}
&&\int_{0}^{t}\int_{0}^{\infty}r^{2(n-1)}u_{x}^{2}v_{x}^{2}dxds\nonumber\\
&\leq&\int_{0}^{t}\left\|r^{n-1}u_{x}\right\|^{2}_{L^{\infty}\left([0,\infty)\right)}\left\|v_{x}\right\|^{2}ds\\
&\leq & C\int_{0}^{t}\left\|r^{n-1}u_{x}\right\|^{2}_{L^{\infty}\left([0,\infty)\right)}ds +C\left\|\theta\right\|^{l_{1}}_{\infty}\int_{0}^{t} \left\|r^{n-1}u_{x}\right\|^{2}_{L^{\infty}\left([0,\infty)\right)}ds\nonumber\\
&\leq& \underbrace{C\int_{0}^{t}\left\|r^{n-1}u_{x}\right\| \left\|\left(r^{n-1}u_{x}\right)_{x}\right\|ds}_{K_{1}}
+\underbrace{C\left\|\theta\right\|^{l_{1}}_{\infty} \int_{0}^{t}\left\|r^{n-1}u_{x}\right\|\left\|\left(r^{n-1}u_{x}\right)_{x}\right\|ds}_{K_{2}}.\nonumber
\end{eqnarray}
To control $K_i (i=1,2)$, noticing
\begin{eqnarray}\label{c5}
\left(r^{n-1}u_{x}\right)_{x}=\left(n-1\right)r^{-1}vu_{x}+r^{n-1}u_{xx},
\end{eqnarray}
we can deduce that
\begin{eqnarray}\label{c6}
\int_{0}^{t}\int_{0}^{\infty}\left(r^{n-1}u_{x}\right)_{x}^{2}dxds&&\leq C\int_{0}^{t}\int_{0}^{\infty}\left(\frac{\alpha r^{2(n-1)}u^{2}_{xx}}{v}+r^{2(n-1)}u^{2}_{x}\right)dxds,
\end{eqnarray}
from which one can deduce that
\begin{eqnarray}\label{c7}
K_{1}&&\leq\frac{1}{6}\int_{0}^{t}\int_{0}^{\infty}\frac{\alpha r^{2(n-1)}u^{2}_{xx}}{v}dxds
+C\int_{0}^{t}\int_{0}^{\infty}r^{2(n-1)}u^{2}_{x}dxds\nonumber\\
&&\leq\frac{1}{6}\int_{0}^{t}\int_{0}^{\infty}\frac{\alpha r^{2(n-1)}u^{2}_{xx}}{v}dxds+C\left\|\theta\right\|_{\infty}
\end{eqnarray}
and
\begin{eqnarray}\label{c8}
K_{2}&\leq&\frac{1}{6}\int_{0}^{t}\int_{0}^{\infty}\frac{\alpha r^{2(n-1)}u^{2}_{xx}}{v}dxds +C\left(1+\left\|\theta\right\|^{2l_{1}}_{\infty}\right) \int_{0}^{t}\int_{0}^{\infty}r^{2(n-1)}u^{2}_{x}dxds\nonumber\\
&\leq&\frac{1}{6}\int_{0}^{t}\int_{0}^{\infty}\frac{\alpha r^{2(n-1)}u^{2}_{xx}}{v}dxds +C\left(1+\left\|\theta\right\|^{2l_{1}+1}_{\infty}\right).
\end{eqnarray}
Thus we have from \eqref{c4}-\eqref{c8} that
\begin{eqnarray}\label{c9}
\int_{0}^{t}\int_{0}^{\infty}r^{2(n-1)}u_{x}^{2}v_{x}^{2}dxds \leq\frac{1}{3}\int_{0}^{t}\int_{0}^{\infty}\frac{\alpha r^{2(n-1)}u_{xx}^{2}}{v}dxds+C\left(1+\left\|\theta\right\|^{2l_{1}+1}_{\infty}\right),
\end{eqnarray}
and \eqref{c1} follows from \eqref{c3} and \eqref{c9}. This completes the proof of Lemma 3.8.
\end{proof}

Now we establish the bounds for $\left\|u\right\|_{L^{\infty}\left([0,\infty)\right)}$ and $\int_{0}^{t}\left\|r^{n-1}u_{x}\right\|^{2}_{L^{\infty}\left([0,\infty)\right)}ds$, which will be frequently used later on.
\begin{lemma} Under the assumptions listed in Lemma 2.1, we have for all $0\leq t\leq T$ that
\begin{eqnarray}\label{c15}
\left\|u\right\|_{L^{\infty}\left([0,\infty)\right)}\leq C+C\left\|\theta\right\|^{\frac{l_{2}}{4}}_{\infty}
\end{eqnarray}
and
\begin{eqnarray}\label{c17}
 \int_{0}^{t}\left\|r^{n-1}u_{x}\right\|^{2}_{L^{\infty}\left([0,\infty)\right)}ds\leq C+C\left\|\theta\right\|^{l_{2}}_{\infty}.
\end{eqnarray}
\end{lemma}

\begin{proof} Firstly, we have from \eqref{bb13} and \eqref{c1} that
\begin{eqnarray*}
\left\|u\right\|_{L^{\infty}\left([0,\infty)\right)}^{2}\leq C\left\|u\right\|\left\|u_{x}\right\|\leq C\left\|u_{x}\right\|\leq C+C\left\|\theta\right\|^{\frac{l_{2}}{2}}_{\infty},
\end{eqnarray*}
from which \eqref{c15} follows immediately.

As to the proof of \eqref{c17}, one can conclude from \eqref{c1}, \eqref{c4} and \eqref{c7} that
\begin{eqnarray}\label{c22}
\int_{0}^{t}\left\|r^{n-1}u_{x}\right\|^{2}_{L^{\infty}\left([0,\infty)\right)}ds&&\leq C\int_{0}^{t}\left\|r^{n-1}u_{x}\right\| \left\|\left(r^{n-1}u_{x}\right)_{x}\right\|ds\nonumber\\
&&\leq C\int_{0}^{t}\int_{\Omega}\frac{\alpha r^{2(n-1)}u^{2}_{xx}}{v}dxds+C\left\|\theta\right\|_{\infty}\nonumber\\
&&\leq C+C\left\|\theta\right\|_{\infty}^{l_{2}}.\nonumber
\end{eqnarray}
This is \eqref{c17} and thus completes the proof of Lemma 3.9.
\end{proof}
The next lemma aims to derive the bound on the term $\int_{0}^{\infty}\frac{r^{2(n-1)}v^{2}_{x}}{v^{2}}dx$.
\begin{lemma} Under the assumptions listed in Lemma 2.1, we have for all $0\leq t\leq T$ that
\begin{eqnarray}\label{c23}
\int_{0}^{\infty}\frac{r^{2(n-1)}v^{2}_{x}}{v^{2}}dx+\int_{0}^{t}\int_{0}^{\infty}\frac{r^{2(n-1)}\theta v^{2}_{x}}{v^{3}}dxds\leq C+C\left\|\theta\right\|^{l_{3}}_{\infty},
\end{eqnarray}
where
\begin{eqnarray}\label{cb24}
l_{3}=\max\left\{1, \frac{l_{2}}{2}, \left(7-b\right)_{+}\right\}.
\end{eqnarray}
\end{lemma}
\begin{proof}
Noticing that $\eqref{a20}_{2}$ can be rewritten as
\begin{eqnarray}\label{c24}
\alpha\left(\frac{r^{n-1}v_{x}}{v}\right)_{t}+\frac{Rr^{n-1}\theta v_{x}}{v^{2}}=u_{t}+\frac{Rr^{n-1}\theta_{x}}{v}
-\alpha(n-1)r^{n-2}u\frac{v_{x}}{v}+\frac{4a}{3}r^{n-1}\theta^{3}\theta_{x}.
\end{eqnarray}

Multiplying \eqref{c24} by $\frac{r^{n-1}v_{x}}{v}$ and integrating the resulting identity with respect to $t$ and $x$ over $[0,t)\times[0,\infty)$, we get by using Cauchy's inequality that
\begin{eqnarray}\label{c25}
&&\frac{\alpha}{2}\displaystyle\int_{0}^{\infty}\frac{r^{2(n-1)}v^{2}_{x}}{v^{2}}dx+\int_{0}^{t}\int_{0}^{\infty}\frac{Rr^{2(n-1)}\theta v^{2}_{x}}{v^{3}}dxds\\
&\leq& C+\frac{1}{8}\int_{0}^{t}\int_{0}^{\infty}\frac{Rr^{2(n-1)}\theta v^{2}_{x}}{v^{3}}dxds+C\int_{0}^{t}\int_{0}^{\infty}\frac{r^{2(n-1)}\theta^{2}_{x}}{v\theta}dxds
+\int_{0}^{t}\int_{0}^{\infty}\frac{u_{t}r^{n-1}v_{x}}{v}dxds\nonumber\\
&&-\alpha(n-1)\int_{0}^{t}\int_{0}^{\infty}\frac{r^{2(n-1)-1}uv^{2}_{x}}{v^{2}}dxds
+\int_{0}^{t}\int_{0}^{\infty}\frac{4ar^{2(n-1)}\theta^{3}v_{x}\theta_{x}}{3v}dxds.\nonumber
\end{eqnarray}

The terms in the right hand side of \eqref{c25} will be estimated below. Firstly, for the third term in the right hand side of \eqref{c25}, it is easy to see that
\begin{eqnarray}\label{c26}
\int_{0}^{t}\int_{0}^{\infty}\frac{r^{2(n-1)}\theta^{2}_{x}}{v\theta}dxds\leq C\int_{0}^{t}\int_{0}^{\infty}\frac{\kappa r^{2(n-1)}\theta^{2}_{x}}{v\theta^{2}}\cdot\frac{\theta}{1+\theta^{b}}dxds\leq C,
\end{eqnarray}
As for the fourth term in the right hand side of \eqref{c25}, one can deduce from \eqref{c15} that
\begin{eqnarray}\label{c27}
&&\int_{0}^{t}\int_{0}^{\infty}\frac{u_{t}r^{n-1}v_{x}}{v}dxds\nonumber\\
&=&\int_{0}^{\infty}\frac{ur^{n-1}v_{x}}{v}dx
-\int_{0}^{\infty}\frac{ur^{n-1}v_{x}}{v}(0,x)dx -(n-1)\int_{0}^{t}\int_{0}^{\infty}\frac{u^{2}r^{n-2}v_{x}}{v}dxds\nonumber\\
&&+\int_{0}^{t}\int_{0}^{\infty}\frac{\left|\left(r^{n-1}u\right)_{x}\right|^{2}}{v}dxds\\
&\leq& C+\frac{\alpha}{4}\int_{0}^{\infty}\frac{r^{2(n-1)}v^{2}_{x}}{v^{2}}dx+\frac{1}{8}\int_{0}^{t}\int_{0}^{\infty}\frac{Rr^{2(n-1)}\theta v^{2}_{x}}{v^{3}}dxds+C\left\|\theta\right\|_{\infty}\nonumber\\
&&+C\int_{0}^{t}\left\|u\right\|_{L^{\infty}\left([0,\infty)\right)}^{2}\int_{0}^{\infty}\frac{u^{2}v}{r^{2}\theta}dxds\nonumber\\
&\leq& C+\frac{\alpha}{4}\int\frac{r^{2(n-1)}v^{2}_{x}}{v^{2}}dx+\frac{1}{8}\int_{0}^{t}\int_{0}^{\infty}\frac{Rr^{2(n-1)}\theta v^{2}_{x}}{v^{3}}dxds+C\left\|\theta\right\|_{\infty}+C\left\|\theta\right\|_{\infty}^{\frac{l_{2}}{2}}.\nonumber
\end{eqnarray}

Moreover, by exploiting the Cauchy inequality, the fifth term in the right hand side of \eqref{c25} can be estimates as follows
\begin{eqnarray}\label{cc28}
&&-\alpha(n-1)\int_{0}^{t}\int_{0}^{\infty}\frac{r^{2(n-1)-1}uv^{2}_{x}}{v^{2}}dxds\nonumber\\
&\leq&\frac{1}{8}\int_{0}^{t}\int_{0}^{\infty}\frac{Rr^{2(n-1)}\theta v^{2}_{x}}{v^{^{3}}}dxds
+C\int_{0}^{t}V(s)\int_{0}^{\infty}\frac{r^{2(n-1)}v^{2}_{x}}{v^{2}}dxds\\
&&+C\int_{0}^{t}\left\|\frac{u}{r}\right\|^{2}_{L^{\infty}\left([0,\infty)\right)}\int\frac{r^{2(n-1)}v^{2}_{x}}{v^{2}}dxds.\nonumber
\end{eqnarray}
Here we have used \eqref{b19} with $m=\frac{1}{2}$.

Finally, the last term in the right hand side of \eqref{c25} can be controlled as in the following
\begin{eqnarray}\label{c29}
&&\int_{0}^{t}\int_{0}^{\infty}\frac{4ar^{2(n-1)}\theta^{3}v_{x}\theta_{x}}{3v}dxds\nonumber\\
&\leq&\frac{1}{8}\int_{0}^{t}\int_{0}^{\infty}\frac{Rr^{2(n-1)}\theta v^{2}_{x}}{v^{3}}dxds+C\int_{0}^{t}\int_{0}^{\infty}r^{2(n-1)}\theta^{5}\theta_{x}^{2}dxds\\
&\leq&\frac{1}{8}\int_{0}^{t}\int_{0}^{\infty}\frac{Rr^{2(n-1)}\theta v^{2}_{x}}{v^{3}}dxds +C\int_{0}^{t}\int_{0}^{\infty}\frac{\kappa\left(r^{n-1}\theta_{x}\right)^{2}}{v\theta^{2}}\cdot\frac{\theta^{7}}{1+\theta^{b}}dxds\nonumber\\
&\leq&\frac{1}{8}\int_{0}^{t}\int_{0}^{\infty}\frac{Rr^{2(n-1)}\theta v^{2}_{x}}{v^{3}}dxds+C\left\|\theta\right\|_{\infty}^{(7-b)_{+}}.\nonumber
\end{eqnarray}
Thus we can conclude \eqref{c23} from \eqref{cc16}, \eqref{c25}-\eqref{c29} and Gronwall's inequality. This completes the proof of Lemma 3.10.
\end{proof}

The next lemma aims to estimate the term $\int_{0}^{\infty}\frac{r^{2(n-1)}u^{2}_{x}}{v^{2}}dx$.
\begin{lemma} Under the assumptions listed in Lemma 2.1, we have for all $0\leq t\leq T$ that
\begin{eqnarray}\label{c31}
\int_{0}^{\infty} r^{2(n-1)}u^{2}_{x}dx+\int_{0}^{t}\int_{0}^{\infty}\frac{r^{4(n-1)}u^{2}_{xx}}{v}dxds\leq C+C\left\|\theta\right\|^{l_{2}+l_{3}}_{\infty}.
\end{eqnarray}
\end{lemma}
\begin{proof}
Multiplying $\eqref{a20}_{2}$ by $-r^{2n-2}u_{xx}$, we can obtain
\begin{eqnarray}\label{c32}
&&\partial_{t}\left(\frac{r^{2n-2}u^{2}_{x}}{2}\right)-\left(r^{2n-2}u_{t}u_{x}\right)_{x} +\frac{\alpha r^{4n-4}u^{2}_{xx}}{v}\nonumber\\
&=&r^{3n-3}u_{xx}\left[P_{x}-2\alpha(n-1)r^{-1}u_{x}+\alpha(n-1)r^{-n-1}uv
+\frac{\alpha r^{n-1}u_{x}v_{x}}{v^{2}}\right]\\
&&+(n-1)r^{2n-3}uu_{x}^{2}-2(n-1)vr^{n-2}u_{x}u_{t}.\nonumber
\end{eqnarray}

Integrating the above identity with respect to $t$ and $x$ over $[0,t)\times[0,\infty)$ and by using Cauchy's inequality and boundary condition \eqref{a22}, we can infer that
\begin{eqnarray}\label{c33}
&&\displaystyle\int_{0}^{\infty} r^{2(n-1)}u^{2}_{x}dx+\int_{0}^{t}\int_{0}^{\infty}\frac{r^{4(n-1)}u^{2}_{xx}}{v}dxds\nonumber\\
&\leq & C+\frac{1}{4}\int_{0}^{t}\int_{0}^{\infty}\frac{r^{4(n-1)}u^{2}_{xx}}{v}dxds+\underbrace{C\int_{0}^{t}\int_{0}^{\infty}r^{2(n-1)}P^{2}_{x}dxds}_{I_4}
\nonumber\\
&&+\underbrace{C\int_{0}^{t}\int_{0}^{\infty}r^{2n-4}u_{x}^{2}dxds}_{I_5}
+\underbrace{C\int_{0}^{t}\int_{0}^{\infty}r^{-4}u^{2}dxds}_{I_6}\\
&&+\underbrace{C\int_{0}^{t}\int_{0}^{\infty}r^{4(n-1)}u_{x}^{2}v_{x}^{2}dxds}_{I_7}
+\underbrace{C\int_{0}^{t}\int_{0}^{\infty}r^{2n-3}|u|u^{2}_{x}dxds}_{I_{8}}\nonumber\\
&&+\underbrace{C\left|\int_{0}^{t}\int_{0}^{\infty}vr^{n-2}u_{x}u_{t}dxds\right|}_{I_9}.\nonumber
\end{eqnarray}

Now we turn to estimate the terms $I_{k}$ ($4\leq k\leq 9$) term by term. To this end, we first have from \eqref{bb13}, \eqref{b43}, \eqref{bbb44}, \eqref{c23} that
\begin{eqnarray}\label{c34}
I_{4}&\leq& C\int_{0}^{t}\int_{0}^{\infty}r^{2(n-1)}\left(\theta^{2}_{x}+\theta^{2}v^{2}_{x}+\theta^{6}\theta^{2}_{x}\right)dxds\nonumber\\
&\leq& C\int_{0}^{t}\int_{0}^{\infty}\frac{\kappa r^{2(n-1)}\theta^{2}_{x}}{v\theta^{2}}\cdot\frac{\theta^{2}}{1+\theta^{b}}dxds
+C\left\|\theta\right\|_{\infty}\int_{0}^{t}\int_{0}^{\infty}r^{2(n-1)}\theta v^{2}_{x}dxds\\
&&+C\int_{0}^{t}\int_{0}^{\infty}\frac{\kappa r^{2(n-1)}\theta^{2}_{x}}{v\theta^{2}}\cdot\frac{\theta^{8}}{1+\theta^{b}}dxds\nonumber\\
&\leq& C+C\left\|\theta\right\|^{l_{3}+1}_{\infty}+C\left\|\theta\right\|^{(8-b)_{+}}_{\infty},\nonumber
\end{eqnarray}
\begin{eqnarray}\label{c35}
I_{5}&\leq& C\int_{0}^{t}\int_{0}^{\infty}\frac{r^{2(n-1)}u^{2}_{x}}{r^{2}v\theta}\cdot\theta dxds\leq C\left\|\theta\right\|_{\infty},
\end{eqnarray}
and
\begin{eqnarray}\label{c36}
I_{6}&\leq& C\int_{0}^{t}\int_{0}^{\infty}\frac{u^{2}v}{r^{2}\theta}\cdot\theta dxds\leq C\left\|\theta\right\|_{\infty}.
\end{eqnarray}

Secondly, \eqref{c17} together with \eqref{c23} tell us that
\begin{eqnarray}\label{c37}
I_{7}&\leq& C\int_{0}^{t}\left\|r^{n-1}u_{x}\right\|^{2}_{L^{\infty}\left([0,\infty)\right)}\int_{0}^{\infty} r^{2(n-1)}v^{2}_{x}dxds\nonumber\\
&\leq& C+C\left\|\theta\right\|_{\infty}^{l_{2}+l_{3}},
\end{eqnarray}
and we can conclude from \eqref{c15} that
\begin{eqnarray}\label{c38}
I_{8}&&\leq C\int_{0}^{t}\left\|u\right\|_{L^{\infty}\left([0,\infty)\right)}\int_{0}^{\infty}\frac{r^{2(n-1)}u^{2}_{x}}{v\theta}\cdot\theta dxds\nonumber\\
&&\leq C+C\left\|\theta\right\|_{\infty}^{\frac{l_{2}}{4}+1}.
\end{eqnarray}

Thirdly, for the term $I_{9}$, noticing that
\begin{eqnarray}\label{c39}
I_{9}&\leq&C\left|\int_{0}^{t}\int_{0}^{\infty}vr^{n-2}u_{x}r^{n-1}\left(\alpha\frac{\left(r^{n-1}u\right)_{x}}{v}-P\right)_xdxds\right|\nonumber\\
&\leq&\underbrace{C\left|\int_{0}^{t}\int_{0}^{\infty}\frac{\alpha(n-1)v^{2}uu_{x}r^{n}}{r^{4}}dxds\right|}_{I_9^1}
+\underbrace{C\left|\int_{0}^{t}\int_{0}^{\infty}\frac{2\alpha(n-1)vr^{2n-2}u^{2}_{x}}{r^{2}}dxds\right|}_{I_9^2}\nonumber\\
&&+\underbrace{C\left|\int_{0}^{t}\int_{0}^{\infty}
r^{2n-3}vu_{x}P_{x}dxds\right|}_{I_9^3}
+\underbrace{C\left|\int_{0}^{t}\int_{0}^{\infty}\alpha r^{3n-4}u_{x}u_{xx}dxds\right|}_{I_9^4}\\
&&+\underbrace{C\left|\int_{0}^{t}\int_{0}^{\infty}\frac{\alpha r^{3n-4}v_{x}u^{2}_{x}}{v}dxds\right|}_{I_9^5},\nonumber
\end{eqnarray}
we can deduce from \eqref{a25}, \eqref{bb13}, \eqref{c17} and \eqref{c23} that
\begin{eqnarray}\label{c40}
I_{9}^{1}\leq C\int_{0}^{t}\int_{0}^{\infty}r^{n-4}\left|uu_{x}\right|dxds\leq C\int_{0}^{t}\int_{0}^{\infty}\left(\frac{r^{2(n-1)}u^{2}_{x}}{v\theta}\cdot\theta+\frac{u^{2}}{r^{6}\theta}\cdot\theta\right)dxds\leq C\left\|\theta\right\|_{\infty},
\end{eqnarray}
\begin{eqnarray}\label{c41}
I_{9}^{2}\leq C\int_{0}^{t}\int_{0}^{\infty}\frac{r^{2(n-1)}u^{2}_{x}}{r^{2}\theta}\cdot\theta dxds\leq C\left\|\theta\right\|_{\infty},
\end{eqnarray}
\begin{eqnarray}\label{c43}
I_{9}^{3}&\leq& C\int_{0}^{t}\int_{0}^{\infty}r^{2n-3}\left|u_{x}\right|\left(\left|\theta_{x}\right|+\theta\left|v_{x}\right|
+\theta^{3}\left|\theta_{x}\right|\right)dxds\nonumber\\
&\leq& C\int_{0}^{t}\int_{0}^{\infty}\frac{r^{2n-2}u^{2}_{x}}{v\theta}\cdot\theta dxds
+C\int_{0}^{t}\int_{0}^{\infty}\frac{\kappa r^{2n-2}\theta^{2}_{x}}{v\theta^{2}}\left(\frac{\theta^{2}}{1+\theta^{b}}+\frac{\theta^{8}}{1+\theta^{b}}\right)dxds\\
&&+C\int_{0}^{t}\int_{0}^{\infty}r^{2n-2}\theta^{2}v_{x}^{2}dxds\nonumber\\
&\leq& C+C\left\|\theta\right\|_{\infty}^{l_{3}+1}+C\left\|\theta\right\|_{\infty}^{(8-b)_{+}},\nonumber
\end{eqnarray}
\begin{eqnarray}\label{c44}
I_{9}^{4}&&\leq \frac{1}{4}\int_{0}^{t}\int_{0}^{\infty}\frac{r^{4(n-1)}u^{2}_{xx}}{v}dxds+C\int_{0}^{t}\int_{0}^{\infty}\frac{r^{2n-2}u^{2}_{x}}{r^{2}\theta}\cdot\theta dxds\nonumber\\
&&\leq \frac{1}{4}\int_{0}^{t}\int_{0}^{\infty}\frac{r^{4(n-1)}u^{2}_{xx}}{v}dxds+C\left\|\theta\right\|_{\infty},
\end{eqnarray}
and
\begin{eqnarray}\label{c45}
I_{9}^{5}&&\leq C\int_{0}^{t}\int_{0}^{\infty}\frac{r^{2n-2}u^{2}_{x}}{r^{2}v\theta}\cdot\theta dxds+C\int_{0}^{t}\int_{0}^{\infty}r^{4n-4}u_{x}^{2}v_{x}^{2}dxds\nonumber\\
&&\leq C\left\|\theta\right\|_{\infty}+C\int_{0}^{t}\left\|r^{n-1}u_{x}\right\|^{2}_{L^{\infty}\left([0,\infty)\right)}\int r^{2(n-1)}v^{2}_{x}dxds\\
&&\leq C+C\left\|\theta\right\|_{\infty}^{l_{2}+l_{3}},\nonumber
\end{eqnarray}
Thus it follows from \eqref{c39}-\eqref{c45} that
\begin{eqnarray}\label{c46}
I_{9}\leq C+\frac{1}{4}\int_{0}^{t}\int_{0}^{\infty}\frac{r^{4(n-1)}u^{2}_{xx}}{v}dxds+C\left\|\theta\right\|_{\infty}^{l_{2}+l_{3}}.
\end{eqnarray}

Inserting the above estimates on $I_k (k=4,5,6,7,8,9)$ into \eqref{c33}, we can deduce \eqref{c31}. This completes the proof of Lemma 3.11.
\end{proof}

\section{Upper bound of $\theta\left(t,x\right)$}
The main purpose of this section is to derive an estimate on the upper bound of $\theta\left(t,x\right)$. To this end, we set
\begin{eqnarray}\label{b56}
X(t):&=&\int_{0}^{t}\int_{0}^{\infty}\left(1+\theta^{b+3}(s,x)\right)\theta_{t}^{2}(s,x)dxds,\nonumber\\ Y(t):&=&\max\limits_{s\in(0,t)}\int_{0}^{\infty}r^{2n-2}\left(1+\theta^{2b}(s,x)\right)\theta_{x}^{2}(s,x)dx,\\
Z(t):&=&\max\limits_{s\in(0,t)}\int_{0}^{\infty}u_{xx}^{2}(s,x)dx,\nonumber
\end{eqnarray}
and then try to deduce certain estimates among them by employing the structure of the system  under our consideration.

Firstly, for each $x\in[0,\infty)$, there exists an integer $k\in\mathbb{N}$ such that $x\in\left[k, k+1\right]$ and we can assume without loss of generality that $x\geq b_{k}\left(t\right)$. Observe first that
\begin{eqnarray}\label{b57}
 \left(\theta(t,x)-1\right)^{2b+6}&=&\left(\theta\left(t,b_{k}(t)\right)-1\right)^{2b+6}
 +\int_{b_{k}\left(t\right)}^{x}\left(2b+6\right)\left(\theta(t,y)-1\right)^{2b+5}\theta_{x}(t,y)dy\nonumber\\
 &\leq& C+C\int_{0}^{\infty}\left|\theta(t,x)-1\right|^{2b+5}|\theta_{x}(t,x)|dx\nonumber\\
 &\leq& C+C\|\theta(t)-1\|_{L^{\infty}\left([0,\infty)\right)}^{b+3}\left(\int_{0}^{\infty}\left(\theta(t,x)-1\right)^{4} dx\right)^{\frac{1}{2}}\cdot\\
 &&\times\left(\int_{0}^{\infty}\left(\theta(t,x)-1\right)^{2b}\theta_{x}^{2}(t,x)r^{2n-2}dx\right)^{\frac{1}{2}} \nonumber\\
 &\leq& C+C\|\theta(t)-1\|^{b+3}_{L^{\infty}\left([0,\infty)\right)}Y^{\frac{1}{2}}(t),\nonumber
\end{eqnarray}
which implies
\begin{eqnarray}\label{b58}
\|\theta(t)\|_{L^{\infty}\left([0,\infty)\right)} \leq C+CY(t)^{\frac{1}{2b+6}},
\end{eqnarray}
where we have used the fact that
\begin{eqnarray}\label{bz58}
\left(\theta-1\right)^{4}\leq \left(\theta-1\right)^{2}\left(3\theta^{2}+2\theta+1\right),\quad\left(\theta-1\right)^{2b}\leq C\left(1+\theta^{2b}\right).
\end{eqnarray}

Secondly, by the Gagliardo-Nirenberg inequality, we infer that
\begin{eqnarray*}
\|u_{x}(t)\|\leq C\|u(t)\|^{\frac{1}{2}}\|u_{xx}(t)\|^{\frac{1}{2}}\leq C\|u_{xx}(t)\|^{\frac{1}{2}},
\end{eqnarray*}
which implies that
\begin{eqnarray}\label{b60}
\max\limits_{s\in(0,t)}\int_{0}^{\infty}u_{x}^{2}(s,x)dx\leq C+CZ(t)^{\frac{1}{2}}.
\end{eqnarray}

Furthermore, by the Sobolev inequality, we can get that
\begin{eqnarray}\label{b61}
\|u_{x}(t)\|_{L^{\infty}\left([0,\infty)\right)}&&\leq C\|u_{x}(t)\|^{\frac{1}{2}}\|u_{xx}(t)\|^{\frac{1}{2}}\nonumber\\
&&\leq C\left(1+Z(t)^{\frac{1}{8}}\right)Z(t)^{\frac{1}{4}}\\
&&\leq C+CZ(t)^{\frac{3}{8}}.\nonumber
  \end{eqnarray}

With the above preparations in hand, our next result is to show that $X(t)$ and $Y(t)$ can be controlled by $Z(t)$.
\begin{lemma} Under the assumptions listed in Lemma 2.1, we have for $0\leq t\leq T$ that
\begin{eqnarray}\label{b62}
X(t)+Y(t)\leq C\left(1+Z(t)^{\lambda_{1}}\right).
\end{eqnarray}
Here 
$\lambda_1$ is given by
\begin{eqnarray}\label{bz136}
\lambda_{1}=\max\left\{\frac{b+3}{4(b+4)},\; \frac{3(b+3)}{2(3b+9-2l_{3})}\right\}.
\end{eqnarray}
It is easy to see that $\lambda_1\in(0,1)$ when $b>\frac{19}{4}$.
\end{lemma}
\begin{proof}
In the same manner as in Kawohl \cite{Kawohl-JDE-1985}, if we set
\begin{eqnarray}\label{b63}
K\left(v,\theta\right)=\int_{0}^{\theta}\frac{\kappa\left(v,\xi\right)}{v}d\xi=\frac{\kappa_1\theta}{v}+\frac{\kappa_2\theta^{b+1}}{b+1},
  \end{eqnarray}
then it is easy to verify from the estimate \eqref{b43} obtained in Lemma 3.4 that
\begin{eqnarray}
K_{t}(v,\theta)&=&K_{v}(v,\theta)\left(r^{n-1}u\right)_{x}+\frac{\kappa(v,\theta)\theta_{t}}{v},\label{b64}\\
K_{xt}(v,\theta)&=&\left[\frac{\kappa(v,\theta)\theta_{x}}{v}\right]_{t}+K_{v}(v,\theta)\left(r^{n-1}u\right)_{xx}\nonumber\\
&&+K_{vv}(v,\theta)v_{x}\left(r^{n-1}u\right)_{x}
+\left(\frac{\kappa(v,\theta)}{v}\right)_{v}v_{x}\theta_{t},\label{b65}\\
\left|K_{v}(v,\theta)\right|&+&\left|K_{vv}(v,\theta)\right|\leq C\theta.\label{b66}
\end{eqnarray}

Noticing that $\eqref{a20}_{3}$ can be rewritten as
\begin{eqnarray}\label{b117}
e_{\theta}\theta_{t}+\theta P_{\theta}\left(r^{n-1}u\right)_{x}-\frac{\alpha\left|\left(r^{n-1}u\right)_{x}\right|^{2}}{v}=\left(\frac{r^{2n-2}\kappa\theta_{x}}{v}\right)_{x}
-2\mu\left(n-1\right)\left(r^{n-2}u^{2}\right)_{x}+\lambda\phi z,
\end{eqnarray}
where
\begin{eqnarray}\label{f3}
 e_{\theta}(v,\theta)=C_{v}+4av\theta^{3},\quad P_{\theta}(v,\theta)=\frac{R}{v}+\frac{4}{3}a\theta^{3},
\end{eqnarray}
we can get by multiplying \eqref{b117} by $K_{t}$ and integrating the resulting identity over $\left[0,t\right)\times[0,\infty)$ and by using the boundary conditions \eqref{a22} that
\begin{eqnarray}\label{b67}
&&\int_{0}^{t}\int_{0}^{\infty}\left(e_{\theta}(v,\theta)\theta_{t}+\theta P_{\theta}(v,\theta)\left(r^{n-1}u\right)_{x}-\frac{\alpha\left|\left(r^{n-1}u\right)_{x}\right|^{2}}{v}\right)K_{t}(v,\theta)dxd\tau\nonumber\\ &&+\int_{0}^{t}\int_{0}^{\infty}\frac{r^{2n-2}\kappa\left(v,\theta\right)\theta_{x}}{v}K_{tx}(v,\theta)dxds\\
&&+2\mu\left(n-1\right)\int_{0}^{t}\int_{0}^{\infty}\left(r^{n-2}u^{2}\right)_{x}K_{t}(v,\theta)dxds\nonumber\\
&=&\int_{0}^{t}\int_{0}^{\infty}\lambda\phi zK_{t}(v,\theta)dxds.\nonumber
\end{eqnarray}

Combining (\ref{b63})-(\ref{b67}), we have
\begin{eqnarray}\label{b68}
&&\int_{0}^{t}\int_{0}^{\infty}\frac{e_{\theta}(v,\theta)\kappa\left(v,\theta\right)\theta_{t}^{2}}{v}dxds
+\int_{0}^{t}\int_{0}^{\infty}\frac{\kappa\left(v,\theta\right)\theta_{x}}{v}\left(\frac{\kappa\left(v,\theta\right)\theta_{x}}{v}\right)_{t}dxds\nonumber\\
&\leq& C+\underbrace{\left|\int_{0}^{t}\int_{0}^{\infty}e_{\theta}(v,\theta)\theta_{t}K_{v}(v,\theta)\left(r^{n-1}u\right)_{x}dxds\right|}_{I_{10}}
+\underbrace{\left|\int_{0}^{t}\int_{0}^{\infty}\theta p_{\theta}(v,\theta)\left|\left(r^{n-1}u\right)_{x}\right|^{2}K_{v}(v,\theta)dxds\right|}_{I_{11}}\nonumber\\
&&+\underbrace{\left|\int_{0}^{t}\int_{0}^{\infty}\frac{\theta p_{\theta}(v,\theta)\kappa\left(v,\theta\right)\left(r^{n-1}u\right)_{x}\theta_{t}}{v}dxds\right|}_{I_{12}}
+\underbrace{\left|\int_{0}^{t}\int_{0}^{\infty}\frac{\alpha\left|\left(r^{n-1}u\right)_{x}\right|^{2}K_{t}(v,\theta)}{v}dxds\right|}_{I_{13}}\nonumber\\
&&+\underbrace{\left|\int_{0}^{t}\int_{0}^{\infty}\frac{r^{2n-2}\kappa(v,\theta)\theta_{x}}{v}\left(K_{vv}(v,\theta)v_{x}\left(r^{n-1}u\right)_{x} +K_{v}(v,\theta)\left(r^{n-1}u\right)_{xx}\right)dxds\right|}_{I_{14}}\\
&&+\underbrace{\left|\int_{0}^{t}\int_{0}^{\infty}\frac{r^{2n-2}\kappa(v,\theta)\theta_{x}}{v} \left(\frac{\kappa(v,\theta)}{v}\right)_{v}v_{x}\theta_{t}dxds\right|}_{I_{15}}\nonumber\\
&&+\underbrace{\left|\int_{0}^{t}\int_{0}^{\infty}\lambda\phi zK_{v}(v,\theta)\left(r^{n-1}u\right)_{x}dxds\right|}_{I_{16}}
+\underbrace{\left|\int_{0}^{t}\int_{0}^{\infty}\frac{\lambda\phi z\kappa\left(v,\theta\right)\theta_{t}}{v}dxds\right| }_{I_{17}}\nonumber\\
&&+\underbrace{C\left|\int_{0}^{t}\int_{0}^{\infty}\left(r^{n-2}u^{2}\right)_{x}\left(K_{v}(v,\theta)\left(r^{n-1}u\right)_{x}
+\frac{\kappa(v,\theta)\theta_{t}}{v}\right)dxds\right|}_{I_{18}}.\nonumber
\end{eqnarray}
We now turn to control $I_k (k=10,11,\cdots, 18)$ term by term. To do so, we first have
\begin{eqnarray}\label{b69}
\int_{0}^{t}\int_{0}^{\infty}\frac{e_{\theta}(v,\theta)\kappa\left(v,\theta\right)\theta_{t}^{2}}{v}dxds&&\geq C\int_{0}^{t}\int_{0}^{\infty}\left(1+\theta^{3}\right)\left(1+\theta^{b}\right)\theta_{t}^{2}dxds \nonumber\\
&&\geq CX(t),
\end{eqnarray}
\begin{eqnarray}\label{b70}
&&\int_{0}^{t}\int_{0}^{\infty}\frac{r^{2n-2}\kappa\left(v,\theta\right)\theta_{x}}{v}\left(\frac{\kappa\left(v,\theta\right)\theta_{x}}{v}\right)_{t}dxds\nonumber\\
&=&\frac{1}{2}\int_{0}^{t}\int_{0}^{\infty}r^{2n-2}\left[\left(\frac{\kappa\theta_{x}}{v}\right)^{2}\right]_{t}dxds\nonumber\\
&=&\frac{1}{2}\int_{0}^{t}\int_{0}^{\infty}\left[r^{2n-2}\left(\frac{\kappa\theta_{x}}{v}\right)^{2}\right]_{t}dxds
-\left(n-1\right)\int_{0}^{t}\int_{0}^{\infty}\left(\frac{\kappa\theta_{x}}{v}\right)^{2}r^{2n-3}udxds\\
&=&\frac{1}{2}\int_{0}^{\infty}r^{2n-2}\left(\frac{\kappa\theta_{x}}{v}\right)^{2}dx-C
-\left(n-1\right)\int_{0}^{t}\int_{0}^{\infty}\left(\frac{\kappa\theta_{x}}{v}\right)^{2}r^{2n-3}udxds\nonumber\\
&\geq& CY-C-\left(n-1\right)\int_{0}^{t}\int_{0}^{\infty}\left(\frac{\kappa\theta_{x}}{v}\right)^{2}r^{2n-3}udxds,\nonumber
\end{eqnarray}
and
\begin{eqnarray}\label{bb70}
&&\left|\left(n-1\right)\int_{0}^{t}\int_{0}^{\infty} \left(\frac{\kappa\theta_{x}}{v}\right)^{2}r^{2n-3}udxds\right|\nonumber\\
&\leq& C\int_{0}^{t}\left\|u\right\|_{L^{\infty}\left([0,\infty)\right)}\int_{0}^{\infty}\left(1+\theta^{b}\right)
\frac{\kappa\left(r^{n-1}\theta_{x}\right)^{2}}{v^{2}\theta^{2}}\frac{\theta^{2}}{r}dxds\nonumber\\
&\leq& C\left(1+\left\|\theta\right\|_{\infty}^{b+2}\right)\int_{0}^{t}\left\|u\right\|^{\frac{1}{2}}
\left\|u_{x}\right\|^{\frac{1}{2}}\int_{0}^{\infty}\frac{\kappa\left(r^{n-1}\theta_{x}\right)^{2}}{v^{2}\theta^{2}}dxds\\
&\leq & C\left(1+Y(t)^{\frac{b+2}{2b+6}}\right)\left(1+Z(t)^{\frac{1}{8}}\right)\nonumber\\
&\leq& C\left(\epsilon\right)\left(1+Z(t)^{\frac{b+3}{4(b+4)}}\right)+\epsilon Y(t),\nonumber
\end{eqnarray}
where we have used the fact that
\begin{eqnarray*}
Y(t)^{\frac{b+2}{2b+6}}Z(t)^{\frac{1}{8}}&&\leq\epsilon Y(t)+C\left(\epsilon\right)Z(t)^{\frac{b+3}{4(b+4)}}.
\end{eqnarray*}

On the other hand, it follows from Cauchy's inequality, \eqref{a25} and \eqref{bb13} that
\begin{eqnarray}\label{b105}
I_{10}&&\leq C\int_{0}^{t}\int_{0}^{\infty} \left(1+\theta\right)^{4}\left|\theta_{t}\left(r^{n-1}u\right)_{x}\right|dxds \nonumber\\
&&\leq \epsilon X(t) +C\left(\epsilon\right)\left(1+\|\theta\|_{\infty}^{(6-b)_{+}}\right)\int_{0}^{t}\int_{0}^{\infty}
\frac{\left|\left(r^{n-1}u\right)_{x}\right|^{2}}{v\theta}\cdot vdxds\\
&&\leq \epsilon X(t)+C\left(\epsilon\right)\left(1+Y(t)^{\frac{(6-b)_{+}}{2b+6}}\right)\nonumber\\
&&\leq \epsilon (X(t)+Y(t))+C\left(\epsilon\right),\nonumber
\end{eqnarray}
\begin{eqnarray}\label{b106}
I_{11}&&\leq C\left|\int_{0}^{t}\int_{0}^{\infty}\left(1+\theta\right)^{5}\left|\left(r^{n-1}u\right)_{x}\right|^{2}dxds\right|\nonumber\\
&&\leq C\int_{0}^{t}\int_{0}^{\infty}\left(1+\theta\right)^{6}\cdot\frac{\left|\left(r^{n-1}u\right)_{x}\right|^{2}}{v\theta}dxds\\
&&\leq C+C\|\theta\|^{6}_{\infty}\nonumber\\
&&\leq C\left(\epsilon\right)+\epsilon Y(t),\nonumber
\end{eqnarray}
and
\begin{eqnarray}\label{b107}
I_{12}&&\leq C\int_{0}^{t}\int_{0}^{\infty}\left(1+\theta\right)^{b+4}\left|\left(r^{n-1}u\right)_{x}\theta_{t}\right|dxds\nonumber\\
&&\leq \epsilon X(t)+C\left(\epsilon\right) \left(1+\|\theta\|^{b+6}_{\infty}\right)\int_{0}^{t}\int_{0}^{\infty}\frac{\left|\left(r^{n-1}u\right)_{x}\right|^{2}}{v\theta}dxds\\
&&\leq \epsilon X(t)+CY(t)^{\frac{b+6}{2b+6}} \nonumber\\
&&\leq \epsilon\left(X(t)+Y(t)\right)+C\left(\epsilon\right).\nonumber
\end{eqnarray}
Here and in the rest of this paper, $\epsilon>0$ is any positive constant which can be chosen as small as we wanted.

As for the term $I_{13}$, it is easy to see that
\begin{eqnarray}\label{b108}
I_{13}&&=\left|\int_{0}^{t}\int_{0}^{\infty}\frac{\alpha\left|\left(r^{n-1}u\right)_{x}\right|^{2}}{v}
\left(K_{v}\left(r^{n-1}u\right)_{x}+\frac{\kappa\left(v, \theta\right)\theta_{t}}{v}\right)dxds\right|\nonumber\\
&&\leq C\int_{0}^{t}\int_{0}^{\infty}\left|\left(r^{n-1}u\right)_{x}\right|^{3}\theta dxds
+C\int_{0}^{t}\int_{0}^{\infty}\left|\left(r^{n-1}u\right)_{x}\right|^{2}\left(1+\theta^{b}\right)\left|\theta_{t}\right|dxds\\
&&:=I_{13}^{1}+I_{13}^{2},\nonumber
\end{eqnarray}
we can deduce from \eqref{cc16}, \eqref{c17}, \eqref{c31}, the fact that $2l_{2}+l_{3}<b+9$ (since $b>\frac{19}{4}$) and
\begin{eqnarray}\label{bb126}
\left(r^{n-1}u\right)_{x}=\frac{(n-1)vu}{r}+r^{n-1}u_{x}
\end{eqnarray}
that
\begin{eqnarray}\label{b109}
\left|I_{13}^{1}\right|&&\leq C\int_{0}^{t}\int_{0}^{\infty}\frac{\left|\left(r^{n-1}u\right)_{x}\right|^{2}}{v\theta}\cdot\theta^{3}dxds
+C\int_{0}^{t}\int_{0}^{\infty}\left|\left(r^{n-1}u\right)_{x}\right|^{4}dxds\nonumber\\
&\leq& C\|\theta\|^{3}_{\infty}+\int_{0}^{t}\int_{0}^{\infty}\left(\frac{u^{4}}{r^{4}}+r^{4n-4}u_{x}^{4}\right)dxds\nonumber\\
&\leq& C+CY(t)^{\frac{3}{2b+6}} +\int_{0}^{t}\left\|\left(\frac{u}{r}\right)(s)\right\|^{2}_{L^{\infty}\left([0,\infty)\right)} \left(\int_{0}^{\infty} u^{2}dx\right)ds\nonumber\\
&&+\int_{0}^{t}\left\|\left(r^{n-1}u_{x}\right)(s)\right\|^{2}_{L^{\infty}\left([0,\infty)\right)} \left(\int_{0}^{\infty} r^{2(n-1)}u^{2}_{x}dx\right)ds\nonumber\\
&\leq& C+CY(t)^{\frac{3}{2b+6}}+CY(t)^{\frac{2l_{2}+l_{3}}{2b+6}}\nonumber\\
&\leq& \epsilon Y(t)+C\left(\epsilon\right)
\end{eqnarray}
and
\begin{eqnarray}\label{b110}
\left|I_{13}^{2}\right|&&\leq\epsilon X(t)+C\left(\epsilon\right)\int_{0}^{t}\int_{0}^{\infty}\left(1+\theta^{b-3}\right)\left(r^{n-1}u\right)_{x}^{4}dxds\nonumber\\
&&\leq\epsilon X(t)+C\left(\epsilon\right)\int_{0}^{t}\int_{0}^{\infty}\left(1+\theta^{b-3}\right)
\left(\frac{u^{4}}{r^{4}}+r^{4n-4}u_{x}^{4}\right)dxds\\
&&\leq\epsilon X(t)+C\left(\epsilon\right)\left(1+\left\|\theta\right\|_{\infty}^{b-3}\right)
\left(1+\int_{0}^{t}\left\|\left(r^{n-1}u_{x}\right)(s)\right\|^{2}_{L^{\infty}\left([0,\infty)\right)} \left(\int_{0}^{\infty} r^{2(n-1)}u^{2}_{x}dx\right)ds\right)\nonumber\\
&&\leq\epsilon X(t)+C\left(\epsilon\right)\left(1+\left\|\theta\right\|_{\infty}^{2l_{2}+l_{3}+b-3}\right)\nonumber\\
&&\leq \epsilon \left(X(t)+Y(t)\right)+C\left(\epsilon\right).\nonumber
\end{eqnarray}
Thus the combination of \eqref{b108}, \eqref{b109}, \eqref{b110} and the assumption $b>\frac{19}{4}$ give birth to
\begin{eqnarray}\label{b111}
I_{13}\leq \epsilon \left(X(t)+Y(t)\right)+C\left(\epsilon\right).
\end{eqnarray}

Now for the term $I_{14}$, we have from \eqref{b66} that
\begin{eqnarray}\label{b112}
I_{14}&\leq& C\int_{0}^{t}\int_{0}^{\infty}r^{2n-2}\left(1+\theta^{b+1}\right)\left|\theta_{x}v_{x}\left(r^{n-1}u\right)_{x}\right|dxds\nonumber\\
&&+C\int_{0}^{t}\int_{0}^{\infty}r^{2n-2}\left(1+\theta^{b+1}\right)\left|\theta_{x}\left(r^{n-1}u\right)_{xx}\right|dxds\\
&:=&I_{14}^{1}+I_{14}^{2},\nonumber
\end{eqnarray}
then we can obtain from \eqref{bb13}, \eqref{cc16}, \eqref{c1}, \eqref{c17}, \eqref{c23} and the fact that
\begin{eqnarray}\label{bd9}
\left(r^{n-1}u\right)_{xx}=2(n-1)r^{-1}vu_{x}+(n-1)r^{-1}v_{x}u-(n-1)r^{-1-n}v^{2}u+r^{n-1}u_{xx},
\end{eqnarray}
that
\begin{eqnarray}\label{b113}
I_{14}^{1}&\leq& C\int_{0}^{t}\int_{0}^{\infty}\frac{\kappa r^{2n-2}\theta^{2}_{x}}{v\theta^{2}}\cdot\frac{1+\theta^{2b+4}}{1+\theta^{b}}dxds+C\int_{0}^{t}\int_{0}^{\infty}r^{2n-2} v^{2}_{x}\left|\left(r^{n-1}u\right)_{x}\right|^{2}dxds\nonumber\\
&\leq& C+C\left\|\theta\right\|_{\infty}^{b+4}+C\int_{0}^{t}\int_{0}^{\infty}r^{2n-2} v^{2}_{x}\left(\frac{u^{2}}{r^{2}}+r^{2n-2}u_{x}^{2}\right)dxds\nonumber\\
&\leq& C+CY(t)^{\frac{b+4}{2b+6}}+\int_{0}^{t}\left\|\frac{u}{r}\right\|^{2}_{L^{\infty}\left([0,\infty)\right)}\int_{0}^{\infty} r^{2n-2}v^{2}_{x}dxds\nonumber\\
&&+C\int_{0}^{t}\left\|r^{n-1}u_{x}\right\|^{2}_{L^{\infty}\left([0,\infty)\right)}\int_{0}^{\infty} r^{2(n-1)}v^{2}_{x}dxds\nonumber\\
&\leq& C+CY(t)^{\frac{b+4}{2b+6}}+CY(t)^{\frac{l_{2}+l_{3}}{2b+6}}\nonumber\\
&\leq& \epsilon Y(t)+C\left(\epsilon\right),
\end{eqnarray}
and
\begin{eqnarray}\label{b114}
I_{14}^{2}&\leq& C\int_{0}^{t}\int_{0}^{\infty}\frac{\kappa r^{2n-2}\theta^{2}_{x}}{v\theta^{2}}\cdot\frac{1+\theta^{2b+4}}{1+\theta^{b}}dxds
+C\int_{0}^{t}\int_{0}^{\infty}\frac{r^{2n-2}\left|\left(r^{n-1}u\right)_{xx}\right|^{2}}{v}dxds\nonumber\\
&\leq& C+C\left\|\theta\right\|_{\infty}^{b+4}+C\int_{0}^{t}\int_{0}^{\infty}\left(r^{2n-4}u_{x}^{2}+r^{2n-4}u^{2}v_{x}^{2}
+\frac{u^{2}}{r^{4}}+r^{4n-4}u_{xx}^{2}\right)dxds\nonumber\\
&\leq& C+CY(t)^{\frac{b+4}{2b+6}}+C\int_{0}^{t}\int_{0}^{\infty}\left(\frac{r^{2n-2}u_{x}^{2}}{v\theta}\cdot\frac{v\theta}{r^{2}}
+\frac{vu^{2}}{r^{2}\theta}\cdot\frac{\theta}{vr^{2}}\right)dxds\nonumber\\
&&+\int_{0}^{t}\left\|\frac{u}{r}\right\|^{2}_{L^{\infty}\left([0,\infty)\right)}\int_{0}^{\infty} r^{2n-2}v^{2}_{x}dxds
+C\int_{0}^{t}\int_{0}^{\infty}r^{4n-4}u_{xx}^{2}dxds\nonumber\\
&\leq& C+CY(t)^{\frac{b+4}{2b+6}}+CY(t)^{\frac{l_{2}+l_{3}}{2b+6}}\nonumber\\
&\leq&\epsilon Y(t)+C\left(\epsilon\right).\nonumber
\end{eqnarray}

Putting \eqref{b112}-\eqref{b114} together, we can see that
\begin{eqnarray}\label{b115}
I_{14}\leq \epsilon Y(t)+C\left(\epsilon\right).
\end{eqnarray}


As to the term $I_{15}$, it is easy to see that
\begin{eqnarray}\label{b116}
I_{15}&&\leq C\int_{0}^{t}\int_{0}^{\infty}r^{2n-2}\left(1+\theta^{b}\right)\left|\theta_{x}\theta_{t}v_{x}\right|dxds\nonumber\\
&&\leq\epsilon X(t)+C\left(\epsilon\right)\int_{0}^{t}\int_{0}^{\infty}r^{4n-4} \left(1+\theta^{b-3}\right)\theta^{2}_{x}v^{2}_{x}dxds\\
&&\leq\epsilon X(t)+C\left(\epsilon\right)\int_{0}^{t}\left\|\frac{r^{n-1}\kappa(v,\theta)\theta_{x}}{v}\right\|^{2}_{L^{\infty}\left([0,\infty)\right)}
\left\|r^{n-1}v_{x}\right\|^{2}ds.\nonumber
\end{eqnarray}

On the other hand, we can deduce from \eqref{b117} that
\begin{eqnarray}\label{b118}
\left(\frac{r^{n-1}\kappa\theta_{x}}{v}\right)_{x}&=&\frac{e_{\theta}\theta_{t}}{r^{n-1}}+\frac{\theta P_{\theta}\left(r^{n-1}u\right)_{x}}{r^{n-1}}-\frac{\alpha\left|\left(r^{n-1}u\right)_{x}\right|^{2}}{vr^{n-1}}-(n-1)r^{-1}\kappa\theta_{x}\nonumber\\
&&+\frac{2\mu\left(n-1\right)\left(r^{n-2}u^{2}\right)_{x}}{r^{n-1}}-\frac{\lambda\phi z}{r^{n-1}}.
\end{eqnarray}

Furthermore, we have
\begin{eqnarray}\label{b119}
\max\limits_{x\in[0,\infty)}\left\{\left(\frac{r^{n-1}\kappa(v,\theta)\theta_{x}}{v}\right)^{2}(t,x)\right\}\leq C\int_{0}^{\infty}\left|\frac{r^{n-1}\kappa(v,\theta)\theta_{x}}{v}\right|\left|\left(\frac{r^{n-1}\kappa(v,\theta)\theta_{x}}{v}\right)_{x}\right|dx.
\end{eqnarray}
Thus for the last term in the right hand side of \eqref{b116}, we can conclude from \eqref{b118}, \eqref{b119} that
\begin{eqnarray}\label{b120}
&&\int_{0}^{t}\left\|\frac{r^{n-1}\kappa(v,\theta)\theta_{x}}{v}\right\|^{2}_{L^{\infty}\left([0,\infty)\right)}
\left\|r^{n-1}v_{x}\right\|^{2}ds\nonumber\\
&\leq& C\left(1+\left\|\theta\right\|^{l_{3}}_{\infty}\right)\int_{0}^{t}\int_{0}^{\infty}\left|\frac{r^{n-1}\kappa(v,\theta)\theta_{x}}{v}\right| \left|\left(\frac{r^{n-1}\kappa(v,\theta)\theta_{x}}{v}\right)_{x}\right|dxds\nonumber\\
&\leq& C\left(1+\left\|\theta\right\|^{l_{3}}_{\infty}\right)\left(\int_{0}^{t}\int_{0}^{\infty}\frac{r^{2n-2}\kappa(v,\theta)\theta_{x}^{2}} {v\theta^{2}}dxds\right)^{\frac{1}{2}}
\left(\int_{0}^{t}\int_{0}^{\infty}\left(1+\theta\right)^{b+2}\left|\left(\frac{r^{n-1}\kappa(v,\theta) \theta_{x}}{v}\right)_{x}\right|^{2}dxds\right)^{\frac{1}{2}} \nonumber\\
&\leq& C\left(1+\left\|\theta\right\|^{l_{3}}_{\infty}\right)\left(\int_{0}^{t}\int_{0}^{\infty}\left(1+\theta\right)^{b+2}
\left(e_{\theta}^{2}(v,\theta)\theta_{t}^{2}+\theta^{2}P_{\theta}^{2}(v,\theta)u_{x}^{2}+u_{x}^{4}+\phi^{2}z^{2} \right)dxds\right)^{\frac{1}{2}}\\
&\leq& C\left(1+Y(t)^{\frac{l_{3}}{2b+6}}\right)\left(\int_{0}^{t}\int_{0}^{\infty} \left(1+\theta\right)^{b+2}
\left(\frac{e_{\theta}^{2}(v,\theta)\theta_{t}^{2}}{r^{2(n-1)}}+\frac{\theta^{2} P_{\theta}^{2}\left|\left(r^{n-1}u\right)_{x}\right|^{2}}{r^{2(n-1)}} +\frac{\left|\left(r^{n-1}u\right)_{x}\right|^{4}}{r^{2(n-1)}}\right.\right.\nonumber\\
&&\left.\left.+r^{-2}\kappa^{2}\theta_{x}^{2}+\frac{\left|\left(r^{n-2}u^{2}\right)_{x}\right|^{2}}{r^{2(n-1)}} +\frac{\phi^{2}z^{2}}{r^{2(n-1)}}\right)dxds
\right)^{\frac{1}{2}}.\nonumber
\end{eqnarray}
To control the last term in the right hand side of \eqref{b120}, one can deduce from \eqref{a25}, \eqref{cc16}, \eqref{f3} and \eqref{bb126} that
\begin{eqnarray}\label{b124}
&&\int_{0}^{t}\int_{0}^{\infty}\frac{(1+\theta)^{b+2}e_{\theta}^{2}(v,\theta)\theta_{t}^{2}}{r^{2(n-1)}}dxds\nonumber\\
&\leq& C\int_{0}^{t}\int_{0}^{\infty}(1+\theta)^{b+8}\theta_{t}^{2}dxds\\
&\leq& C\left(1+\|\theta\|_{\infty}^{5}\right)X(t)\nonumber\\
&\leq& C\left(1+Y(t)^{\frac{5}{2b+6}}\right)X(t),\nonumber
\end{eqnarray}
\begin{eqnarray}\label{b125}
&&\int_{0}^{t}\int_{0}^{\infty}\frac{(1+\theta)^{b+2}\theta^{2}P_{\theta}^{2}(v,\theta)\left|\left(r^{n-1}u\right)_{x}\right|^{2}}{r^{2(n-1)}}dxds\nonumber\\
&\leq& C\int_{0}^{t}\int_{0}^{\infty}\frac{\left(1+\theta\right)^{b+11}\left|\left(r^{n-1}u\right)_{x}\right|^{2}}{v\theta}dxds\\
&\leq& C\left(1+\|\theta\|_{\infty}^{b+11}\right)\nonumber\\
&\leq& C+CY(t)^{\frac{b+11}{2b+6}},\nonumber
\end{eqnarray}
\begin{eqnarray}\label{b126}
&&\int_{0}^{t}\int_{0}^{\infty} \frac{(1+\theta)^{b+2}\left|\left(r^{n-1}u\right)_{x}\right|^{4}}{r^{2(n-1)}}dxds\nonumber\\
&\leq&
C\int_{0}^{t}\int_{0}^{\infty}\frac{\left(1+\theta\right)^{b+2}\left(r^{-4}u^{4} +r^{4(n-1)}u^{4}_{x}\right)}{r^{2(n-1)}}dxds\nonumber\\
&\leq& C\left(1+\left\|\theta\right\|_{\infty}^{b+2}\right)\int_{0}^{t}\left\|\frac{u}{r}\right\|^{2}_{L^{\infty}\left([0,\infty)\right)}\int_{0}^{\infty}u^{2}dxds\\
&&+C\left(1+\left\|\theta\right\|_{\infty}^{b+3}\right)\int_{0}^{t}\left\|u_{x}\right\|^{2}_{L^{\infty}\left([0,\infty)\right)}
\int_{0}^{\infty}\frac{r^{2(n-1)}u^{2}_{x}}{v\theta}dxds\nonumber\\
 &\leq&C\left(1+Y(t)^{\frac{1}{2}}\right)\left(1+Z(t)^{\frac{3}{4}}\right),\nonumber
\end{eqnarray}
and
\begin{eqnarray}\label{b127}
&&\int_{0}^{t}\int_{0}^{\infty}\left(1+\theta\right)^{b+2}r^{-2}\kappa^{2}\theta_{x}^{2}dxds\nonumber\\
&\leq&
C\int_{0}^{t}\int_{0}^{\infty}\frac{\kappa\left(r^{n-1}\theta_{x}\right)^{2}}{v\theta^{2}}\cdot\left(1+\theta\right)^{b+2}
\left(1+\theta\right)^{b}\theta^{2}dxds\\
&\leq& C\left(1+\left\|\theta\right\|_{\infty}^{2b+4}\right)\nonumber\\
&\leq& C\left(1+Y(t)^{\frac{2b+4}{2b+6}}\right).\nonumber
\end{eqnarray}

Moreover, from \eqref{b13} and the fact that
\begin{eqnarray}\label{b128}
\left(r^{n-2}u^{2}\right)_{x}=(n-2)r^{n-3}r_{x}u^{2}+2r^{n-2}uu_{x}=(n-2)r^{-2}vu^{2}+2r^{n-2}uu_{x},
\end{eqnarray}
we have
\begin{eqnarray}\label{b129}
&&\int_{0}^{t}\int_{0}^{\infty}\frac{\left(1+\theta\right)^{b+2}\left|\left(r^{n-2}u^{2}\right)_{x}\right|^{2}}{r^{2(n-1)}}dxds\nonumber\\
&\leq& C\int_{0}^{t}\int_{0}^{\infty}\left[\frac{\left(1+\theta\right)^{b+2}u^{4}}{r^{2(n-1)+4}}+
\frac{r^{2(n-2)}\left(1+\theta\right)^{b+2}u^{2}u^{2}_{x}}{r^{2(n-1)}}\right]dxds\nonumber\\
&\leq& C\left(1+\left\|\theta\right\|_{\infty}^{b+2}\right)\int_{0}^{t}\int_{0}^{\infty}\frac{u^{4}}{r^{2}}dxds\\
&&+C\left(1+\left\|\theta\right\|_{\infty}^{b+3}\right)\int_{0}^{t}\left\|u_{x}\right\|^{2}_{L^{\infty}\left([0,\infty)\right)}
\int\frac{u^{2}}{r^{2}\theta}dxds\nonumber\\
 &\leq&C\left(1+Y(t)^{\frac{1}{2}}\right)\left(1+Z(t)^{\frac{3}{4}}\right)\nonumber
\end{eqnarray}
and
\begin{eqnarray}\label{b130}
&&\int_{0}^{t}\int_{0}^{\infty}\frac{\left(1+\theta\right)^{b+2}\phi^{2}z^{2}}{r^{2(n-1)}}dxds\nonumber\\
&\leq& C\left(1+\left\|\theta\right\|_{\infty}^{b+2+\beta}\right)\int_{0}^{t}\int_{0}^{\infty}\phi z^{2}dxds\\
&\leq& C+CY(t)^{\frac{b+2+\beta}{2b+6}}.\nonumber
\end{eqnarray}

Consequently, we obtain by combining the estimates (\ref{b120})-(\ref{b130}) that
\begin{eqnarray}\label{b131}
&&\int_{0}^{t}\left\|\frac{r^{n-1}\kappa(v,\theta)\theta_{x}}{v}\right\|^{2}_{L^{\infty}\left([0,\infty)\right)}
\left\|r^{n-1}v_{x}\right\|^{2}ds\nonumber\\
&\leq& C\left(1+Y(t)^{\frac{l_{3}}{2b+6}}\right)\left[\left(1+Y(t)^{\frac{5}{2b+6}}\right)X(t)\right.\nonumber\\
&&\left.+Y(t)^{\frac{b+11}{2b+6}}+\left(1+Y(t)^{\frac{1}{2}}\right)\left(1+Z(t)^{\frac{3}{4}}\right)+Y(t)^{\frac{2b+4}{2b+6}} +Y(t)^{\frac{b+2+\beta}{2b+6}}\right]^{\frac{1}{2}}\nonumber\\
&\leq& C\left(1+Y(t)^{\frac{l_{3}}{2b+6}}\right)\bigg(1+X(t)^{\frac{1}{2}}+X(t)^{\frac{1}{2}}Y(t)^{\frac{5}{4b+12}}+
Y(t)^{\frac{b+11}{4b+12}}+Z(t)^{\frac{3}{8}}\\
&&+Y(t)^{\frac{1}{4}}Z(t)^{\frac{3}{8}}+Y(t)^{\frac{2b+4}{4b+12}}
+Y(t)^{\frac{b+2+\beta}{4b+12}}\bigg)\nonumber\\
&\leq& C\bigg(1+X(t)^{\frac{1}{2}}+X(t)^{\frac{1}{2}}Y(t)^{\frac{2l_{3}+5}{4b+12}}+Y(t)^{\frac{2l_{3}+b+11}{4b+12}}
+Y(t)^{\frac{2l_{3}+b+3}{4b+12}}Z(t)^{\frac{3}{8}}\nonumber\\
&&+Y(t)^{\frac{2b+2l_{3}+4}{4b+12}}+Y(t)^{\frac{b+2+\beta+2l_{3}}{4b+12}}\bigg)\nonumber\\
&\leq&\epsilon\left(X(t)+Y(t)\right)+C\left(\epsilon\right) \left(1+Z(t)^{\frac{3(b+3)}{2(3b+9-2l_{3})}}\right),\nonumber
\end{eqnarray}
where we have used the facts that
\begin{eqnarray*}
X(t)^{\frac{1}{2}}Y(t)^{\frac{2l_{3}+5}{4b+12}}&&\leq\epsilon X(t)+C\left(\epsilon\right)Y(t)^{\frac{5+2l_{3}}{2b+6}}\\
&&\leq \epsilon \left(X(t)+Y(t)\right)+C\left(\epsilon\right),\\
Y(t)^{\frac{b+11+2l_{3}}{4b+12}}&&\leq \epsilon Y(t)+C\left(\epsilon\right),\\
Y(t)^{\frac{2l_{3}+b+3}{4b+12}}Z(t)^{\frac{3}{8}}&&\leq\epsilon Y(t)+C\left(\epsilon\right)Z(t)^{\frac{3(b+3)}{2(3b+9-2l_{3})}},\\
Y(t)^{\frac{2b+4+2l_{3}}{4b+12}}&&\leq \epsilon Y(t)+C\left(\epsilon\right),\\
Y(t)^{\frac{b+2+\beta+2l_{3}}{4b+12}}&&\leq \epsilon Y(t)+C\left(\epsilon\right)
\end{eqnarray*}
together with the assumption $0\leq\beta<b+9$.

Combining (\ref{b116})-(\ref{b131}), we can get the following estimate on $I_{15}$
\begin{eqnarray}\label{b132}
|I_{15}|&&\leq\epsilon\left(X(t)+Y(t)\right)+C\left(\epsilon\right)\left(1+Z(t)^{\frac{3(b+3)}{2(3b+9-2l_{3})}}\right).
\end{eqnarray}

Now for $I_k (k=16,17)$, we can get from \eqref{b15} and the assumption $b>\frac{19}{4}$ and $0\leq\beta<b+9$ that
\begin{eqnarray}\label{b133}
I_{16}&&=\left|\int_{0}^{t}\int_{0}^{\infty}\lambda\phi zK_{v}(v,\theta)\left(r^{n-1}u\right)_{x}dxds\right| \nonumber\\
&&\leq C\left(1+\|\theta\|_{\infty}^{\frac{\beta+3}{2}}\right)\left(\int_{0}^{t}\int_{0}^{\infty}\phi z^{2}dxds\right)^{\frac{1}{2}}\left(\int_{0}^{t}\int_{0}^{\infty}\frac{\left|\left(r^{n-1}u\right)_{x}\right|^{2}}{v\theta}dxds\right)^{\frac{1}{2}}\\
&&\leq C+CY(t)^{\frac{\beta+3}{4b+12}}\nonumber\\
&&\leq\epsilon Y(t)+C\left(\epsilon\right)\nonumber
\end{eqnarray}
and
\begin{eqnarray}\label{bb134}
I_{17}&&=\left|\int_{0}^{t}\int_{0}^{\infty}\frac{\lambda\phi z\kappa\left(v,\theta\right)\theta_{t}}{v}dxds\right| \nonumber\\
&&\leq \epsilon X(t)+C\left(\epsilon\right)\int_{0}^{t}\int_{0}^{\infty}\left(1+\theta^{b+\beta-3}\right)\phi z^{2}dxds\\
&&\leq \epsilon X(t)+C\left(\epsilon\right)\left(1+\|\theta\|^{b+\beta-3}_{\infty}\right)\int_{0}^{t}\int_{0}^{\infty}\phi z^{2}dxds\nonumber\\
&&\leq \epsilon (X(t)+Y(t))+C\left(\epsilon\right).\nonumber
\end{eqnarray}

Finally, for the term $I_{18}$, one can infer from \eqref{b128} that
\begin{eqnarray}\label{b135}
I_{18}&\leq& C\int_{0}^{t}\int_{0}^{\infty}\left(r^{-2}u^{2}+r^{n-2}\left|uu_{x}\right|\right)\left[\theta\left|\left(r^{n-1}u\right)_{x}\right|+
\left(1+\theta^{b}\right)\left|\theta_{t}\right|\right]dxds\\
&\leq& C\int_{0}^{t}\int_{0}^{\infty}\bigg[r^{-2}u^{2}\theta\left|\left(r^{n-1}u\right)_{x}\right|
+r^{-2}u^{2}\left(1+\theta^{b}\right)\left|\theta_{t}\right|\nonumber\\
&&+r^{n-2}\left(1+\theta^{b}\right)\left|uu_{x}\theta_{t}\right|+r^{n-2}\theta \left|uu_{x}\left(r^{n-1}u\right)_{x}\right|\bigg]dxds\nonumber\\
&:=&I_{18}^{1}+I_{18}^{2}+I_{18}^{3}+I_{18}^{4}.\nonumber
\end{eqnarray}
$I_{18}^j (j=1,2,3,4)$ can be estimated term by term as in the following: Firstly, we have from \eqref{cc16} that
\begin{eqnarray}\label{b136}
I_{18}^{1}&&\leq C\int_{0}^{t}\int_{0}^{\infty}\frac{\left|\left(r^{n-1}u\right)_{x}\right|^{2}}{v\theta}\cdot\theta^{3}dxds
+C\int_{0}^{t}\int_{0}^{\infty}\frac{u^{4}}{r^{4}}dxds\nonumber\\
&&\leq C\left\|\theta\right\|_{\infty}^{3}+\int_{0}^{t}\left\|\frac{u}{r}\right\|^{2}_{L^{\infty}\left([0,\infty)\right)}
\int_{0}^{\infty}u^{2}dxds\nonumber\\
&&\leq C+CY(t)^{\frac{3}{2b+6}}\nonumber\\
&&\leq \epsilon Y(t)+C\left(\epsilon\right),
\end{eqnarray}
and
\begin{eqnarray}\label{bb137}
I_{18}^{2}&&\leq \epsilon X(t)+C\left(\epsilon\right)\int_{0}^{t}\int_{0}^{\infty}r^{-4}u^{4}\left(1+\theta^{b-3}\right)dxds\nonumber\\
&&\leq\epsilon X(t)+C\left(\epsilon\right)\left(1+\left\|\theta\right\|^{b-3}_{\infty}\right)\\
&&\leq\epsilon X(t)+C\left(\epsilon\right)\left(1+Y(t)^{\frac{b-3}{2b+6}}\right)\nonumber\\
&&\leq\epsilon\left(X(t)+Y(t)\right)+C\left(\epsilon\right).\nonumber
\end{eqnarray}

As for the term $I_{18}^{3}$, we have from \eqref{cc16} and \eqref{c31} that
\begin{eqnarray}\label{bbc138}
I_{18}^{3}&&\leq \epsilon X(t)+C\left(\epsilon\right)\int_{0}^{t}\int_{0}^{\infty}\frac{r^{2(n-1)}u^{2}u^{2}_{x}\left(1+\theta^{b-3}\right)}{r^{2}}dxds\nonumber\\
&&\leq\epsilon X(t)+C\left(\epsilon\right)\left(1+\left\|\theta\right\|^{b-3}_{\infty}\right)
\int_{0}^{t}\left\|\frac{u}{r}\right\|^{2}_{L^{\infty}\left([0,\infty)\right)}\int_{0}^{\infty}r^{2(n-1)}u^{2}_{x}dxds\nonumber\\
&&\leq\epsilon X(t)+C\left(\epsilon\right)\left(1+\left\|\theta\right\|^{b-3+l_{2}+l_{3}}_{\infty}\right)\\
&&\leq\epsilon X(t)+C\left(\epsilon\right)\left(1+Y(t)^{\frac{b-3+l_{2}+l_{3}}{2b+6}}\right)\nonumber\\
&&\leq\epsilon\left(X(t)+Y(t)\right)+C\left(\epsilon\right).\nonumber
\end{eqnarray}

For $I^4_{18}$, we can get from \eqref{bb13}, \eqref{cc16} and the fact that $l_{2}<2b+6$ since $b>\frac{19}{4}$ that
\begin{eqnarray}\label{bg139}
I_{18}^{4}&&\leq C\int_{0}^{t}\int_{0}^{\infty}\frac{\left|\left(r^{n-1}u\right)_{x}\right|^{2}}{v\theta}\cdot\theta^{3}dxds
+C\int_{0}^{t}\int_{0}^{\infty}r^{2n-4}u^{2}u^{2}_{x}dxds\nonumber\\
&&\leq C\left\|\theta\right\|^{3}_{\infty}
+C\int_{0}^{t}\left\|r^{n-1}u_{x}\right\|^{2}_{L^{\infty}\left([0,\infty)\right)}\int_{0}^{\infty}u^{2}dxds\\
&&\leq C+CY(t)^{\frac{3}{2b+6}}+CY(t)^{\frac{l_{2}}{2b+6}}\nonumber\\
&&\leq\epsilon Y\left(t\right)+C\left(\epsilon\right).\nonumber
\end{eqnarray}
Thus it follows from \eqref{b135}-\eqref{bg139} that
\begin{eqnarray}\label{bg140}
I_{18}&&\leq \epsilon\left(X(t)+Y(t)\right)+C\left(\epsilon\right).
\end{eqnarray}

With the above estimates in hand, if we define $\lambda_1$ as in \eqref{bz136}, then combining all the above estimates and by choosing $\epsilon>0$ small enough, we can get the estimate \eqref{b62} immediately. Since $b>\frac{19}{4}$, one can deduce from \eqref{bd58}, \eqref{cc1}, \eqref{cb24} and \eqref{bz136} that $0<\lambda_1<1$. This completes the proof of Lemma 4.1.
\end{proof}

Our next result in this section is to show that $Z(t)$ can be bounded by $X(t)$ and $Y(t)$.
\begin{lemma} Under the assumptions listed in Lemma 2.1, we have for all $0\leq t\leq T$ that
\begin{eqnarray}\label{b137}
Z(t)\leq C\left(1+X(t)+Y(t)+Z(t)^{\lambda_{2}}\right).
\end{eqnarray}
Here $\lambda_{2}$ is given by
\begin{eqnarray}\label{b150}
\lambda_{_{2}}=\max\left\{\frac{b+3}{2(2b+5)},\;\frac{3(b+3)}{2(2b+6-l_{1})}\right\}.
\end{eqnarray}
It is easy to see that $\lambda_2\in(0,1)$ provided that $b>\frac{19}{4}$.
\end{lemma}
\begin{proof}
Differentiating $(\ref{a20})_{2}$ with respect to $t$, multiplying it by $u_{t}$, we obtain that
\begin{eqnarray}
\partial_{t}\left(\frac{u_{t}^{2}}{2}\right)&=&\left\{\left[r^{n-1}\left(\alpha\frac{\left(r^{n-1}u\right)_{x}}{v}-P\right)\right]_{t}u_{t}\right\}_{x}
-\left[r^{n-1}\left(\alpha\frac{\left(r^{n-1}u\right)_{x}}{v}-P\right)\right]_{t}u_{tx}\nonumber\\
&&-\left[\left(r^{n-1}\right)_{x}\left(\alpha\frac{\left(r^{n-1}u\right)_{x}}{v}-P\right)\right]_{t}u_{t}.\nonumber
\end{eqnarray}
Then integrating the above identity with respect to $t$ and $x$ over $\left[0,t\right)\times[0,\infty)$ yields
\begin{eqnarray}\label{b138}
&&\frac{\|u_{t}(t)\|^{2}}{2}+\int_{0}^{t}\int_{0}^{\infty}\frac{\alpha r^{2(n-1)}u_{tx}^{2}}{v}dxds+\alpha(n-1)^{2}\int_{0}^{t}\int_{0}^{\infty}\frac{vu_{t}^{2}}{r^{2}}dxds\nonumber\\
&\leq& C\underbrace{-\int_{0}^{t}\int_{0}^{\infty}\alpha(n-1)(n-2)r^{n-3}u^{2}u_{tx}dxds}_{I_{19}} \underbrace{-2\alpha(n-1)\int_{0}^{t}\int_{0}^{\infty}r^{n-2}u_{t}u_{tx}dxds}_{I_{20}}\nonumber\\
&&\underbrace{-2\alpha(n-1)\int_{0}^{t}\int_{0}^{\infty}\frac{r^{2n-3}uu_{x}u_{tx}}{v}dxds}_{I_{21}} +\underbrace{\alpha\int_{0}^{t}\int_{0}^{\infty}\frac{r^{2(n-1)}u_{x}u_{tx}v_{t}}{v^{2}}dxds}_{I_{22}}\nonumber\\
&&+\underbrace{\int_{0}^{t}\int_{0}^{\infty}r^{n-1}P_{t}u_{tx}dxds}_{I_{23}}
+\underbrace{(n-1)\int_{0}^{t}\int_{0}^{\infty}r^{n-2}Puu_{tx}dxds}_{I_{24}}\\
&&+\underbrace{2\alpha(n-1)^{2}\int_{0}^{t} \int_{0}^{\infty}\frac{vu^{2}u_{t}}{r^{3}}dxds}_{I_{25}}
\underbrace{-\alpha(n-1)^{2}\int_{0}^{t}\int_{0}^{\infty}r^{-2}uu_{t}v_{t}dxds}_{I_{26}}\nonumber\\
&&\underbrace{-\alpha(n-1)(n-2)\int_{0}^{t}\int_{0}^{\infty}r^{n-3}uu_{x}u_{t}dxds}_{I_{27}}
+\underbrace{(n-1)\int_{0}^{t}\int_{0}^{\infty}r^{-1}Pu_{t}v_{t}dxds}_{I_{28}}\nonumber\\
&&\underbrace{-(n-1)\int_{0}^{t}\int_{0}^{\infty}r^{-2}vuPu_{t}dxds}_{I_{29}} +\underbrace{(n-1)\int_{0}^{t}\int_{0}^{\infty}vr^{-1}u_{t}P_{t}dxds}_{I_{30}}.\nonumber
\end{eqnarray}
Now we turn to estimate $I_k (k=19, 20, \cdots, 30)$ term by term. To this end, we compute from \eqref{a23}, \eqref{bb13} and \eqref{cc16} that
\begin{eqnarray}\label{bb139}
I_{19}&&\leq\epsilon\int_{0}^{t}\int_{0}^{\infty}\frac{\alpha r^{2(n-1)}u_{tx}^{2}}{v}dxds+C(\epsilon)\int_{0}^{t}\left\|\frac{u}{r}\right\|^{2}_{L^{\infty}\left([0,\infty)\right)}\int_{0}^{\infty}u^{2}dxds\nonumber\\
&&\leq\epsilon\int_{0}^{t}\int_{0}^{\infty}\frac{\alpha r^{2(n-1)}u_{tx}^{2}}{v}dxds+C(\epsilon)
\end{eqnarray}
and
\begin{eqnarray}\label{bb140}
I_{20}&&=-\alpha(n-1)\int_{0}^{t}\int_{0}^{\infty}r^{n-2}\left(u_{t}^{2}\right)_{x}dxds\nonumber\\
&&=\alpha(n-1)\int_{0}^{t}\int_{0}^{\infty}\left(r^{n-2}\right)_{x}u_{t}^{2}dxds\\
&&=\alpha(n-1)(n-2)\int_{0}^{t}\int_{0}^{\infty}r^{n-3}r_{x}u_{t}^{2}dxds\nonumber\\
&&=\alpha(n-1)(n-2)\int_{0}^{t}\int_{0}^{\infty}r^{-2}vu_{t}^{2}dxds.\nonumber
\end{eqnarray}

As for the term $I_{21}$, due to
\begin{eqnarray}\label{bb141}
I_{21}&&\leq\epsilon\int_{0}^{t}\int_{0}^{\infty}\frac{\alpha r^{2(n-1)}u_{tx}^{2}}{v}dxds+C(\epsilon)\int_{0}^{t}\int_{0}^{\infty}r^{2(n-1)-2}u^{2}u_{x}^{2}dxds,
\end{eqnarray}
and noticing
\begin{eqnarray}\label{bc141}
&&\int_{0}^{t}\int_{0}^{\infty}r^{2(n-1)-2}u^{2}u_{x}^{2}dxds\nonumber\\
&\leq& C\int_{0}^{t}\int_{0}^{\infty}\frac{r^{2(n-1)}u_{x}^{2}}{v\theta}\cdot\frac{u^{2}\theta}{r^{2}}dxds\nonumber\\
&\leq& C\left\|\theta\right\|_{\infty}\int_{0}^{t}\left\|u\right\|_{L^{\infty}\left([0,\infty)\right)}^{2}\int_{0}^{\infty}\frac{r^{2(n-1)}u_{x}^{2}}{v\theta}dxds\\
&\leq& C\left\|\theta\right\|_{\infty}\int_{0}^{t}\left\|u\right\|\left\|u_{x}\right\|\int_{0}^{\infty}\frac{r^{2(n-1)}u_{x}^{2}}{v\theta}dxds\nonumber\\
&\leq& C\left(1+Y^{\frac{1}{2b+6}}(t)\right)\left(1+Z^{\frac{1}{4}}(t)\right),\nonumber\\
&\leq& C\left(1+Y(t)+Z^{\frac{b+3}{2(2b+5)}}(t)\right),\nonumber
\end{eqnarray}
we can deduce that
\begin{eqnarray}\label{bb142}
I_{21}\leq\epsilon\int_{0}^{t}\int_{0}^{\infty}\frac{\alpha r^{2(n-1)}u_{tx}^{2}}{v}dxds+C(\epsilon)\left(1+Y(t)+Z^{\frac{b+3}{2(2b+5)}}(t)\right).
\end{eqnarray}

For the term $I_{22}$, since
\begin{eqnarray}\label{bb143}
I_{22}&&\leq\epsilon\int_{0}^{t}\int_{0}^{\infty}\frac{\alpha r^{2(n-1)}u_{tx}^{2}}{v}dxds+C(\epsilon)\int_{0}^{t}\int_{0}^{\infty} \frac{r^{2(n-1)}u_{x}^{2}\left(r^{n-1}u\right)^{2}_{x}}{v^{2}}dxds,
\end{eqnarray}
and due to
\begin{eqnarray}\label{bb145}
\int_{0}^{t}\int_{0}^{\infty}\frac{r^{2(n-1)}u_{x}^{2}\left|\left(r^{n-1}u\right)_{x}\right|^{2}}{v^{2}}dxds&&\leq C\int_{0}^{t}\int_{0}^{\infty}\left(\frac{r^{2(n-1)}u^{2}_{x}u^{2}}{r^{2}} +r^{4(n-1)}u^{4}_{x}\right)dxds
\end{eqnarray}
and
\begin{eqnarray}\label{bb147}
&&\int_{0}^{t}\int_{0}^{\infty}r^{4(n-1)}u_{x}^{4}dxds\nonumber\\
&\leq& C\int_{0}^{t}\left\|r^{n-1}u_{x}\right\|^{2}_{L^{\infty}\left([0,\infty)\right)}\int_{0}^{\infty} r^{2(n-1)}u^{2}_{x}dxds\nonumber\\
&\leq& C+C\left\|\theta\right\|^{2l_{2}+l_{3}}_{\infty}\\
&\leq& C+CY^{\frac{2l_{2}+l_{3}}{2b+6}}\nonumber\\
&\leq& C\left(1+Y(t)\right),\nonumber
\end{eqnarray}
we have by combining \eqref{bc141}, \eqref{bb145}-\eqref{bb147} that
\begin{eqnarray}\label{bb148}
\int_{0}^{t}\int_{0}^{\infty}\frac{r^{2(n-1)}u_{x}^{2}\left|\left(r^{n-1}u\right)_{x}\right|^{2}}{v^{2}}dxds\leq C\left(1+Y(t)+Z^{\frac{b+3}{2(2b+5)}}(t)\right).
\end{eqnarray}
Consequently, we can conclude from \eqref{bb143}-\eqref{bb148} that
\begin{eqnarray}\label{bb149}
I_{22}&&\leq\epsilon\int_{0}^{t}\int_{0}^{\infty}\frac{\alpha r^{2(n-1)}u_{tx}^{2}}{v}dxds+C(\epsilon)\left(1+Y(t)+Z^{\frac{b+3}{2(2b+5)}}(t)\right).
\end{eqnarray}

Finally for $I_k (k=23, 24, \cdots, 30)$, noticing that
\begin{eqnarray}\label{bz138}
\frac{\partial P(v,\theta)} {\partial t} =\left(\frac{R}{v}+\frac{4}{3}a\theta^{3}\right)\theta_{t}-\frac{R\theta\left(r^{n-1}u\right)_{x}}{v^{2}},
\end{eqnarray}
then it follows from \eqref{a8} and \eqref{b43} that
\begin{eqnarray}\label{bb150}
I_{23}&&\leq\epsilon\int_{0}^{t}\int_{0}^{\infty}\frac{\alpha r^{2(n-1)}u_{tx}^{2}}{v}dxds+C(\epsilon)\int_{0}^{t}\int_{0}^{\infty}P_{t}^{2}dxds\nonumber\\
&&\leq\epsilon\int_{0}^{t}\int_{0}^{\infty}\frac{\alpha r^{2(n-1)}u_{tx}^{2}}{v}dxds
+C(\epsilon)\int_{0}^{t}\int_{0}^{\infty}\left[\left(1+\theta^{6}\right)\theta_{t}^{2}
+\theta^{2}\left|\left(r^{n-1}u\right)_{x}\right|^{2}\right]dxds\\
&&\leq\epsilon\int_{0}^{t}\int_{0}^{\infty}\frac{\alpha r^{2(n-1)}u_{tx}^{2}}{v}dxds
+C(\epsilon)\left(X(t)+\int_{0}^{t}\int_{0}^{\infty}\frac{\left|\left(r^{n-1}u\right)_{x}\right|^{2}}{v\theta}\cdot\theta^{3}\right)dxds\nonumber\\
&&\leq\epsilon\int_{0}^{t}\int_{0}^{\infty}\frac{\alpha r^{2(n-1)}u_{tx}^{2}}{v}dxds
+C(\epsilon)\left(1+X(t)+Y(t)\right),\nonumber
\end{eqnarray}
\begin{eqnarray}\label{bb151}
I_{24}&&\leq\epsilon\int_{0}^{t}\int_{0}^{\infty}\frac{\alpha r^{2(n-1)}u_{tx}^{2}}{v}dxds+C(\epsilon)\int_{0}^{t}\int_{0}^{\infty}\frac{P^{2}u^{2}}{r^{2}}dxds\nonumber\\
&&\leq\epsilon\int_{0}^{t}\int_{0}^{\infty}\frac{\alpha r^{2(n-1)}u_{tx}^{2}}{v}dxds
+C(\epsilon)\int_{0}^{t}\int_{0}^{\infty}\frac{u^{2}}{r^{2}\theta} \cdot\theta\left(\theta^{2}+\theta^{8}\right)dxds\\
&&\leq\epsilon\int_{0}^{t}\int_{0}^{\infty}\frac{\alpha r^{2(n-1)}u_{tx}^{2}}{v}dxds
+C(\epsilon)\left(1+Y^{\frac{9}{2b+6}}(t)\right)\nonumber\\
&&\leq\epsilon\int_{0}^{t}\int_{0}^{\infty}\frac{\alpha r^{2(n-1)}u_{tx}^{2}}{v}dxds
+C(\epsilon)\left(1+Y(t)\right),\nonumber
\end{eqnarray}
\begin{eqnarray}\label{bb152}
I_{25}&&\leq\epsilon\int_{0}^{t}\int_{0}^{\infty}\frac{vu_{t}^{2}}{r^{2}}dxds
+C(\epsilon)\int_{0}^{t}\int_{0}^{\infty}\frac{u^{4}}{r^{4}}dxds\nonumber\\
&&\leq\epsilon\int_{0}^{t}\int_{0}^{\infty}\frac{vu_{t}^{2}}{r^{2}}dxds+C(\epsilon),
\end{eqnarray}
\begin{eqnarray}\label{bb153}
I_{26}&&\leq\epsilon\int_{0}^{t}\int_{0}^{\infty}\frac{vu_{t}^{2}}{r^{2}}dxds
+C(\epsilon)\int_{0}^{t}\int_{0}^{\infty}\frac{u^{2}\left|\left(r^{n-1}u\right)_{x}\right|^{2}}{r^{2}}dxds\nonumber\\
&&\leq\epsilon\int_{0}^{t}\int_{0}^{\infty}\frac{vu_{t}^{2}}{r^{2}}dxds
+C(\epsilon)\left\|\theta\right\|_{\infty}\int_{0}^{t}\left\|u\right\|^{2}_{L^{\infty}\left([0,\infty)\right)}
\int_{0}^{\infty}\frac{\left|\left(r^{n-1}u\right)_{x}\right|^{2}}{v\theta}dxds\nonumber\\
&&\leq\epsilon\int_{0}^{t}\int_{0}^{\infty}\frac{vu_{t}^{2}}{r^{2}}dxds
+C(\epsilon)\left\|\theta\right\|_{\infty}\int_{0}^{t}\left\|u\right\|\left\|u_{x}\right\|
\int_{0}^{\infty}\frac{\left|\left(r^{n-1}u\right)_{x}\right|^{2}}{v\theta}dxds\\
&&\leq\epsilon\int_{0}^{t}\int_{0}^{\infty}\frac{vu_{t}^{2}}{r^{2}}dxds
+C(\epsilon)\left(1+Y^{\frac{1}{2b+6}}(t)\right)\left(1+Z^{\frac{1}{4}}(t)\right)\nonumber\\
&&\leq\epsilon\int_{0}^{t}\int_{0}^{\infty}\frac{vu_{t}^{2}}{r^{2}}dxds
+C(\epsilon)\left(1+Y(t)+Z^{\frac{b+3}{2(2b+5)}}(t)\right),\nonumber
\end{eqnarray}
\begin{eqnarray}\label{bb154}
I_{27}&&\leq\epsilon\int_{0}^{t}\int_{0}^{\infty}\frac{vu_{t}^{2}}{r^{2}}dxds
+C(\epsilon)\int_{0}^{t}\int_{0}^{\infty}\frac{r^{2(n-1)}u_{x}^{2}}{v\theta}\cdot\frac{v\theta u^{2}}{r^{2}}dxds\nonumber\\
&&\leq\epsilon\int_{0}^{t}\int_{0}^{\infty}\frac{vu_{t}^{2}}{r^{2}}dxds
+C(\epsilon)\left\|\theta\right\|_{\infty}\int_{0}^{t}\left\|u\right\|^{2}_{L^{\infty}\left([0,\infty)\right)}
\int_{0}^{\infty}\frac{r^{2(n-1)}u_{x}^{2}}{v\theta}dxds\\
&&\leq\epsilon\int_{0}^{t}\int_{0}^{\infty}\frac{vu_{t}^{2}}{r^{2}}dxds
+C(\epsilon)\left(1+Y(t)+Z^{\frac{b+3}{2(2b+5)}}(t)\right),\nonumber
\end{eqnarray}
\begin{eqnarray}\label{bb155}
I_{28}&&\leq\epsilon\int_{0}^{t}\int_{0}^{\infty}\frac{vu_{t}^{2}}{r^{2}}dxds
+C(\epsilon)\int_{0}^{t}\int_{0}^{\infty}\left|\left(r^{n-1}u\right)_{x}\right|^{2}P^{2}dxds\nonumber\\
&&\leq\epsilon\int_{0}^{t}\int_{0}^{\infty}\frac{vu_{t}^{2}}{r^{2}}dxds
+C(\epsilon)\int_{0}^{t}\int_{0}^{\infty}\frac{\left|\left(r^{n-1}u\right)_{x}\right|^{2}}{v\theta}
\cdot\theta\left(\theta^{2}+\theta^{8}\right)dxds\nonumber\\
&&\leq\epsilon\int_{0}^{t}\int_{0}^{\infty}\frac{vu_{t}^{2}}{r^{2}}dxds
+C(\epsilon)\left(1+\left\|\theta\right\|^{9}_{\infty}\right)\\
&&\leq\epsilon\int_{0}^{t}\int_{0}^{\infty}\frac{vu_{t}^{2}}{r^{2}}dxds
+C(\epsilon)\left(1+Y^{\frac{9}{2b+6}}(t)\right)\nonumber\\
&&\leq\epsilon\int_{0}^{t}\int_{0}^{\infty}\frac{vu_{t}^{2}}{r^{2}}dxds
+C(\epsilon)\left(1+Y(t)\right),\nonumber
\end{eqnarray}
\begin{eqnarray}\label{bb156}
I_{29}&&\leq\epsilon\int_{0}^{t}\int_{0}^{\infty}\frac{vu_{t}^{2}}{r^{2}}dxds
+C(\epsilon)\int_{0}^{t}\int_{0}^{\infty}\frac{u^{2}P^{2}}{r^{2}}dxds\nonumber\\
&&\leq\epsilon\int_{0}^{t}\int_{0}^{\infty}\frac{vu_{t}^{2}}{r^{2}}dxds
+C(\epsilon)\int_{0}^{t}\int_{0}^{\infty}\frac{u^{2}}{r^{2}\theta}
\cdot\theta\left(\theta^{2}+\theta^{8}\right)dxds\\
&&\leq\epsilon\int_{0}^{t}\int_{0}^{\infty}\frac{vu_{t}^{2}}{r^{2}}dxds
+C(\epsilon)\left(1+Y(t)\right),\nonumber
\end{eqnarray}
and
\begin{eqnarray}\label{bb157}
I_{30}&&\leq\epsilon\int_{0}^{t}\int_{0}^{\infty}\frac{vu_{t}^{2}}{r^{2}}dxds
+C(\epsilon)\int_{0}^{t}\int_{0}^{\infty}P_{t}^{2}dxds\nonumber\\
&&\leq\epsilon\int_{0}^{t}\int_{0}^{\infty}\frac{vu_{t}^{2}}{r^{2}}dxds
+C(\epsilon)\left(1+X(t)+Y(t)\right).
\end{eqnarray}

Combining the estimates \eqref{b138}-\eqref{bb157} and then by further choosing $\epsilon>0$ small enough, we arrive at
\begin{eqnarray}\label{bb158}
&&\left\|u_{t}(t)\right\|^{2}+\int_{0}^{t}\int_{0}^{\infty}\frac{ r^{2(n-1)}u_{tx}^{2}}{v}dxds+\int_{0}^{t}\int_{0}^{\infty}\frac{vu_{t}^{2}}{r^{2}}dxds\nonumber\\
&\leq& C\left(1+X(t)+Y(t)+Z^{\frac{b+3}{2(2b+5)}}(t)\right).
\end{eqnarray}

To yield an estimate on $\left\|u_{xx}(t)\right\|$, $(\ref{a20})_{2}$ tells us that
\begin{eqnarray}\label{b140}
u_{xx}=\frac{u_{t}v}{\alpha r^{2(n-1)}}+\frac{(n-1)v^{2}u}{r^{2n}}-\frac{2(n-1)vu_{x}}{r^{n}}+v^{-1}u_{x}v_{x}+\frac{P_{x}v}{\alpha r^{n-1}}.
\end{eqnarray}
Thus one gets by combining \eqref{bbb44} and \eqref{b140} that
\begin{eqnarray}\label{b142}
\|u_{xx}(t)\|^{2}&&\leq C\int_{0}^{\infty}\left(\frac{u_{t}^{2}}{r^{4(n-1)}}+\frac{u^{2}}{r^{4n}}+\frac{u_{x}^{2}}{r^{2n}}+u_{x}^{2}v^{2}_{x}+\frac{P^{2}_{x}}{r^{2(n-1)}}        \right)dx\nonumber\\
&&\leq C\left(1+\int_{0}^{\infty}\left( u_{t}^{2}+P_{x}^{2}+u^{2}_{x}+u_{x}^{2}v_{x}^{2}\right)dx\right)\\
&&\leq C\left(1+\|u_{t}(t)\|^{2}+\int_{0}^{\infty}\left(1+\theta^{6}\right)\theta_{x}^{2}dx +\int_{0}^{\infty}\left(\theta^{2}+u_{x}^{2}\right)v_{x}^{2}dx+\int_{0}^{\infty}u_{x}^{2}dx\right)\nonumber\\
&&\leq C\left(1+X(t)+Y(t)+Z^{\frac{b+3}{2(2b+5)}}(t)+\|\theta\|_{\infty}^{2}\|v_{x}(t)\|^{2} +\|u_{x}\|^{2}_{\infty}\|v_{x}(t)\|^{2}+\left\|u_{x}\right\|^{2}\right),\nonumber
\end{eqnarray}
while the last three terms in the right hand side of \eqref{b142} can be bounded as follows:
\begin{eqnarray}\label{b143}
\|\theta\|_{\infty}^{2}\|v_{x}(t)\|^{2}&&\leq C+C\|\theta\|_{\infty}^{2+l_{1}}\nonumber\\
&&\leq C+CY(t)^{\frac{2+l_{1}}{2b+6}}\\
&&\leq C+CY(t),\nonumber
\end{eqnarray}
\begin{eqnarray}\label{b144}
\|u_{x}\|^{2}_{\infty}\|v_{x}(t)\|^{2}&&\leq C\left(1+\|\theta\|_{\infty}^{l_{1}}\right)\left(1+Z(t)^{\frac{3}{4}}\right)\nonumber\\
&&\leq C\left(1+Z(t)^{\frac{3}{4}}+Y(t)^{\frac{l_{1}}{2b+6}}+Y(t)^{\frac{l_{1}}{2b+6}}Z(t)^{\frac{3}{4}}\right)\\
&&\leq C+CY(t)+CZ(t)^{\frac{3(b+3)}{2(2b+6-l_{1})}},\nonumber
\end{eqnarray}
and
\begin{eqnarray}\label{bd145}
\left\|u_{x}\right\|^{2}&&\leq C\left(1+\|\theta\|_{\infty}^{l_{2}}\right)\nonumber\\
&&\leq C\left(1+Y(t)^{\frac{l_{2}}{2b+6}}\right)\\
&&\leq C\left(1+Y(t)\right).\nonumber
\end{eqnarray}
Thus it follows from (\ref{b142})-(\ref{bd145}) that
\begin{eqnarray}\label{b145}
\|u_{xx}(t)\|^{2}\leq C\left(1+X(t)+Y(t)+Z(t)^{\lambda_{_{2}}}\right),
\end{eqnarray}
where $\lambda_2$ is given by \eqref{b150}.

Having obtained the estimate \eqref{b145}, we can get the estimate \eqref{b137} by using the definition of $Z(t)$, thus the proof of Lemma 4.2 is complete.
\end{proof}

Combining Lemma 4.1 with Lemma 4.2, we can deduce that $Y(t)\leq C$, thus we can get the desired upper bounds on $\theta\left(t,x\right)$ from \eqref{b58}. In fact, we can deduce from Lemma 2.1 and Lemma 4.2 that
\begin{lemma} Under the assumptions listed in Lemma 2.1, there exists a positive constant $\overline{\Theta}$ which depends only on the fixed constants $\mu$, $\lambda_{1}$, $\lambda$, $K$, $A$, $d$, $R$, $c_{v}$, $a$, $\kappa_{1}$, $\kappa_{2}$, $n$ and the initial data $(v_0(x), u_0(x),$ $\theta_0(x), z_0(x))$, such that
  \begin{eqnarray}\label{b151}
 \theta\left(t,x\right) \leq\overline{\Theta},\quad \forall
 \left(t,x\right)\in[0,T] \times[0,\infty).
  \end{eqnarray}
Moreover, we have for $0\leq t\leq T$ that
\begin{eqnarray}\label{b152}
 &&\left\|\left(v-1, u, \theta-1, z\right)\left(t\right)\right\|^{2}+\left\|r^{n-1}\left(v_{x}, u_{x}, \theta_{x} \right)\left(t\right)\right\|^{2}
 +\left\|u_{xx}(t)\right\|^{2}+\left\|z(t)\right\|_{L^{1}([0,\infty))}\nonumber\\
 &&+\int_{0}^{t}\left\|\left(r^{2n-2}u_{xx}, r^{n-1}\sqrt{\theta}v_{x},r^{n-1}u_{tx},r^{n-1}u_{x},r^{n-1}\theta_{x},\theta_{t},r^{n-1}z_{x}\right)(s)\right\|^{2}ds\leq C
\end{eqnarray}
and
\begin{eqnarray}\label{b152-2}
\left\|u_{x}(t)\right\|_{L^{\infty}\left([0,\infty)\right)}\leq C,\quad \|u(t)\|_{L^{\infty}\left([0,\infty)\right)}\leq C,\quad  \int_{0}^{t}\left\|r^{n-1}u_{x}\right\|^{2}_{L^{\infty}\left([0,\infty)\right)}ds\leq C.
\end{eqnarray}
Recall that the constants $C$ in \eqref{b151}, \eqref{b152} and \eqref{b152-2} depend only on the fixed constants $\mu$, $\lambda_{1}$, $\lambda$, $K$, $A$, $d$, $R$, $c_{v}$, $a$, $\kappa_{1}$, $\kappa_{2}$, $n$ and the initial data $(v_0(x), u_0(x),$ $ \theta_0(x), z_0(x))$.
\end{lemma}

The main purpose of the following lemma is to deduce a nice bound on $\int_{0}^{t}\left\|r^{n-1}\theta_{xx}(s)\right\|_{L^{\infty}\left([0,\infty)\right)}^2ds$.
\begin{lemma} Under the assumptions listed in Lemma 2.1, we have for $0\leq t\leq T$ that
\begin{eqnarray}\label{b160}
\left\|\theta_{x}(t)\right\|^{2}+\int_{0}^{t}\left\|r^{n-1}\theta_{xx}(s)\right\|^2ds\leq C
\end{eqnarray}
and
\begin{eqnarray}\label{bb160}
\int_{0}^{t}\left\|r^{n-1}\theta_{x}(s)\right\|^2_{L^{\infty}\left([0,\infty)\right)}ds\leq C.
\end{eqnarray}
\end{lemma}
 \begin{proof}
Multiplying (\ref{b117}) by $-\frac{\theta_{xx}}{e_{\theta}}$, one has
\begin{eqnarray}\label{b161}
&&\frac{1}{2}\left(\theta_{x}^{2}\right)_{t}-\left(\theta_{t}\theta_{x}\right)_{x}+\frac{r^{2n-2}\kappa\theta^{2}_{xx}}{ve_{\theta}}\nonumber\\
&=&\frac{\theta P_{\theta}\left(r^{n-1}u\right)_{x}\theta_{xx}}{e_{\theta}}-\frac{\alpha\left(r^{n-1}u\right)^{2}_{x}\theta_{xx}}{ve_{\theta}}
-\frac{(2n-2)r^{2n-2}\kappa\theta_{x}\theta_{xx}}{e_{\theta}}\\
&&-\frac{r^{2n-2}\kappa_{v}v_{x}\theta_{x}\theta_{xx}}{ve_{\theta}}-\frac{r^{2n-2}\kappa_{\theta}\theta^{2}_{x}\theta_{xx}}{ve_{\theta}}
+\frac{r^{2n-2}\kappa v_{x}\theta_{x}\theta_{xx}}{v^{2}e_{\theta}}\nonumber\\
&&+\frac{2\mu(n-1)\left(r^{n-2}u^{2}\right)_{x}\theta_{xx}}{e_{\theta}}-\frac{\lambda\phi z\theta_{xx}}{e_{\theta}}.\nonumber
\end{eqnarray}
Integrating the above identity with respect to $x$ over $[0,\infty)$ and by using the boundary conditions \eqref{a22}, one has
\begin{eqnarray}\label{b162}
&&\frac{1}{2}\frac{d}{dt}\|\theta_{x}(t)\|^{2}+\int_{0}^{\infty}\frac{r^{2n-2}\kappa(v,\theta)\theta^{2}_{xx}}{ve_{\theta}(v,\theta)}dx\nonumber\\
&\leq& \underbrace{\left|\int_{0}^{\infty}\frac{\theta P_{\theta}(v,\theta)\left(r^{n-1}u\right)_{x}\theta_{xx}}{e_{\theta}(v,\theta)}dx\right|}_{I_{31}}
+\underbrace{\left|\int_{0}^{\infty}\frac{\alpha\left|\left(r^{n-1}u\right)_{x}\right|^{2}\theta_{xx}}{ve_{\theta}(v,\theta)}dx\right|}_{I_{32}}\nonumber\\
&&+\underbrace{\left|\int_{0}^{\infty}\frac{r^{2n-2}\kappa_{v}(v,\theta)v_{x}\theta_{x}\theta_{xx}}{ve_{\theta}(v,\theta)}dx\right|}_{I_{33}}
+\underbrace{\left|\int_{0}^{\infty}\frac{r^{2n-2}\kappa_{\theta}(v,\theta)\theta^{2}_{x}\theta_{xx}}{ve_{\theta}(v,\theta)}dx\right|}_{I_{34}}\\
&&+\underbrace{\left|\int_{0}^{\infty}\frac{r^{2n-2}\kappa(v,\theta)\theta_{x}v_{x}\theta_{xx}}{v^{2}e_{\theta}(v,\theta)}dx\right| }_{I_{35}}
+\underbrace{\left|\int_{0}^{\infty}\frac{\lambda\phi z\theta_{xx}}{e_{\theta}(v,\theta)}dx\right|}_{I_{36}}\nonumber\\
&&+\underbrace{\left|\int_{0}^{\infty}\frac{(2n-2)r^{2n-2}\kappa\theta_{x}\theta_{xx}}{e_{\theta}(v,\theta)}dx\right|}_{I_{37}}
+\underbrace{\left|\int_{0}^{\infty}\frac{2\mu(n-1)\left(r^{n-2}u^{2}\right)_{x}\theta_{xx}} {e_{\theta}(v,\theta)}dx\right|}_{I_{38}}.\nonumber
\end{eqnarray}

Now we are in a position to bound the terms $I_k (31\leq k\leq 38)$. To begin with, we first have from \eqref{f3} that
\begin{eqnarray}\label{b163}
I_{31}
&&\leq\epsilon\int_{0}^{\infty}\frac{r^{2n-2}\kappa(v,\theta)\theta^{2}_{xx}}{ve_{\theta}(v,\theta)}dx +C\left(\epsilon\right)\int_{0}^{\infty}\frac{\left|\left(r^{n-1}u\right)_{x}\right|^{2}}{v\theta}\cdot\frac{\theta^{3}P^{2}_{\theta}(v,\theta)} {\kappa(v,\theta)r^{2n-2}e_{\theta}(v,\theta)}dx\\
&&\leq\epsilon\int_{0}^{\infty}\frac{r^{2n-2}\kappa(v,\theta)\theta^{2}_{xx}}{ve_{\theta}(v,\theta)}dx +C\left(\epsilon\right)V\left(t\right).\nonumber
\end{eqnarray}

Secondly, \eqref{a25}, \eqref{bb126}, \eqref{b151} and \eqref{b152-2} imply that
\begin{eqnarray}\label{b164}
I_{32}
&\leq&\epsilon\int_{0}^{\infty}\frac{r^{2n-2}\kappa(v,\theta)\theta^{2}_{xx}}{ve_{\theta}(v,\theta)}dx +C\left(\epsilon\right)\int_{0}^{\infty}\frac{\left|\left(r^{n-1}u\right)_{x}\right|^{4}}{r^{2n-2}\kappa(v,\theta)e_{\theta}(v,\theta)}dx\nonumber\\
&\leq&\epsilon\int_{0}^{\infty}\frac{r^{2n-2}\kappa(v,\theta)\theta^{2}_{xx}}{ve_{\theta}(v,\theta)}dx
+C\left(\epsilon\right)\int_{0}^{\infty}\frac{u^{4}+r^{4(n-1)}u^{4}_{x}}{r^{2n-2}\kappa(v,\theta)e_{\theta}(v,\theta)}dx\\
&\leq&\epsilon\int_{0}^{\infty}\frac{r^{2n-2}\kappa(v,\theta)\theta^{2}_{xx}}{ve_{\theta}(v,\theta)}dx
+C\left(\epsilon\right)\left(\left\|\frac{u}{r}\right\|^{2}_{L^{\infty}\left([0,\infty)\right)}+\int_{0}^{\infty}\frac{r^{2(n-1)}u^{2}_{x}}{v\theta}\cdot\frac{u^{2}_{x}\theta}
{\left(1+\theta^{b}\right)\left(1+\theta^{3}\right)}dx\right)\nonumber\\
&\leq&\epsilon\int_{0}^{\infty}\frac{r^{2n-2}\kappa(v,\theta)\theta^{2}_{xx}}{ve_{\theta}(v,\theta)}dx
+C\left(\epsilon\right)\left(\left\|\frac{u}{r}\right\|^{2}_{L^{\infty}\left([0,\infty)\right)}+V(t)\right).\nonumber
\end{eqnarray}

Thirdly, by virtue of Sobolev's inequality and Lemma 4.3, we can infer that
\begin{eqnarray}\label{b165}
I_{33}
&&\leq\epsilon\int_{0}^{\infty}\frac{r^{2n-2}\kappa(v,\theta)\theta^{2}_{xx}}{ve_{\theta}(v,\theta)}dx
+C\left(\epsilon\right)\int_{0}^{\infty}r^{2n-2}v^{2}_{x}\cdot\frac{\kappa_{v}^{2}(v,\theta)\theta^{2}_{x}}{e_{\theta}(v,\theta) \kappa(v,\theta)}dx\nonumber\\
&&\leq\epsilon\int_{0}^{\infty}\frac{r^{2n-2}\kappa(v,\theta)\theta^{2}_{xx}}{ve_{\theta}(v,\theta)}dx
+C\left(\epsilon\right)\left\|\theta_{x}(t)\right\|\left\|\theta_{xx}(t)\right\|\\
&&\leq 2\epsilon\int_{0}^{\infty}\frac{r^{2n-2}\kappa(v,\theta)\theta^{2}_{xx}}{ve_{\theta}(v,\theta)}dx
+C\left(\epsilon\right)\int_{0}^{\infty}\frac{r^{2n-2}\kappa(v,\theta)\theta^{2}_{x}}{v\theta^{2}}
\cdot\frac{\theta^{2}e_{\theta}}{r^{4n-4}\kappa^{2}}dx\nonumber\\
&&\leq 2\epsilon\int_{0}^{\infty}\frac{r^{2n-2}\kappa(v,\theta)\theta^{2}_{xx}}{ve_{\theta}(v,\theta)}dx
+C\left(\epsilon\right)V(t),\nonumber
\end{eqnarray}
\begin{eqnarray}\label{b168}
I_{34}
&&\leq\epsilon\int_{0}^{\infty}\frac{r^{2n-2}\kappa(v,\theta)\theta^{2}_{xx}}{ve_{\theta}(v,\theta)}dx +C\left(\epsilon\right)\int_{}\frac{r^{2n-2}\theta^{4}_{x}\kappa^{2}_{\theta}(v,\theta)}{\kappa(v,\theta) e_{\theta}(v,\theta)}dx \nonumber\\
&&\leq\epsilon\int_{0}^{\infty}\frac{r^{2n-2}\kappa(v,\theta)\theta^{2}_{xx}}{ve_{\theta}(v,\theta)}dx
+C\left(\epsilon\right)\left\|\theta_{x}(t)\right\|\left\|\theta_{xx}(t)\right\|\int_{0}^{\infty}r^{2n-2}\theta^{2}_{x}dx\\
&&\leq\epsilon\int_{0}^{\infty}\frac{r^{2n-2}\kappa(v,\theta)\theta^{2}_{xx}}{ve_{\theta}(v,\theta)}dx
+C\left(\epsilon\right)\left\|\theta_{x}(t)\right\|\left\|\theta_{xx}(t)\right\|\nonumber\\
&&\leq2\epsilon\int_{0}^{\infty}\frac{r^{2n-2}\kappa(v,\theta)\theta^{2}_{xx}}{ve_{\theta}(v,\theta)}dx +C\left(\epsilon\right)V\left(t\right),\nonumber
\end{eqnarray}
and
\begin{eqnarray}\label{b169}
I_{35}
&&\leq\epsilon\int_{0}^{\infty}\frac{r^{2n-2}\kappa(v,\theta)\theta^{2}_{xx}}{ve_{\theta}(v,\theta)}dx
+C\left(\epsilon\right)\int_{0}^{\infty}\frac{r^{2n-2}\kappa(v,\theta)\theta^{2}_{x}v^{2}_{x}}{e_{\theta}(v,\theta)}dx\nonumber\\
&&\leq\epsilon\int_{0}^{\infty}\frac{r^{2n-2}\kappa(v,\theta)\theta^{2}_{xx}}{ve_{\theta}(v,\theta)}dx
+C\left(\epsilon\right)\left\|\theta_{x}\right\|^{2}_{L^{\infty}\left([0,\infty)\right)}\int_{0}^{\infty}r^{2n-2}v_{x}^{2}dx\\
&&\leq\epsilon\int_{0}^{\infty}\frac{r^{2n-2}\kappa(v,\theta)\theta^{2}_{xx}}{ve_{\theta}(v,\theta)}dx
+C\left(\epsilon\right)\|\theta_{x}(t)\|\|\theta_{xx}(t)\|\nonumber\\
&&\leq2\epsilon\int_{0}^{\infty}\frac{r^{2n-2}\kappa(v,\theta)\theta^{2}_{xx}}{ve_{\theta}(v,\theta)}dx +C\left(\epsilon\right)V\left(t\right).\nonumber
\end{eqnarray}

Finally, it follows from \eqref{a2}, \eqref{a25}, \eqref{f3}, \eqref{b151} and \eqref{b152} that
\begin{eqnarray}\label{b170}
I_{36}
&&\leq\epsilon\int_{0}^{\infty}\frac{r^{2n-2}\kappa(v,\theta)\theta^{2}_{xx}}{ve_{\theta}(v,\theta)}dx
+C\left(\epsilon\right)\int_{0}^{\infty}\frac{v\phi^{2}z^{2}}{r^{2n-2}\kappa(v,\theta) e_{\theta}(v,\theta)}dx\nonumber\\
&&\leq\epsilon\int_{0}^{\infty}\frac{r^{2n-2}\kappa(v,\theta)\theta^{2}_{xx}}{ve_{\theta}(v,\theta)}dx
+C\left(\epsilon\right)\int_{0}^{\infty}\phi z^{2}dx,
\end{eqnarray}
\begin{eqnarray}\label{bb170}
I_{37}
&&\leq\epsilon\int_{0}^{\infty}\frac{r^{2n-2}\kappa(v,\theta)\theta^{2}_{xx}}{ve_{\theta}(v,\theta)}dx
+C\left(\epsilon\right)\int_{0}^{\infty}\frac{r^{2n-2}\kappa\theta^{2}_{x}}{v\theta^{2}}
\cdot\frac{\theta^{2}}{e_{\theta}(v,\theta)}dx\nonumber\\
&&\leq\epsilon\int_{0}^{\infty}\frac{r^{2n-2}\kappa(v,\theta)\theta^{2}_{xx}}{ve_{\theta}(v,\theta)}dx
+C\left(\epsilon\right)V(t),
\end{eqnarray}
and
\begin{eqnarray}\label{bb171}
I_{38}&&\leq C\int_{0}^{\infty}\frac{\left(r^{-2}u^{2}+r^{n-2}\left|uu_{x}\right|\right) \theta_{xx}}{e_{\theta}(v,\theta)}dx\\
&&\leq\epsilon\int_{0}^{\infty}\frac{r^{2n-2}\kappa(v,\theta)\theta^{2}_{xx}}{ve_{\theta}(v,\theta)}dx
+C\left(\epsilon\right)\int_{0}^{\infty}\left(\frac{u^{4}}{r^{2n+2}\kappa}+\frac{r^{2n-2}u^{2}_{x}}{v\theta}
\cdot\frac{\theta}{r^{2n}e_{\theta}(v,\theta)\kappa(v,\theta)}\right)dx\nonumber\\
&&\leq\epsilon\int_{0}^{\infty}\frac{r^{2n-2}\kappa(v,\theta)\theta^{2}_{xx}}{ve_{\theta}(v,\theta)}dx
+C\left(\epsilon\right)\left(\left\|\frac{u}{r}\right\|^{2}_{L^{\infty}\left([0,\infty)\right)}+V(t)\right).\nonumber
\end{eqnarray}

Combining (\ref{b162})-(\ref{bb171}) and by choosing $\epsilon>0$ small enough, we arrive at
\begin{eqnarray}\label{b171}
\frac{d}{dt}\|\theta_{x}(t)\|^{2}+\left\|r^{n-1}\theta_{xx}(t)\right\|^2\leq C\left(V\left(t\right)+\left\|\frac{u}{r}\right\|^{2}_{L^{\infty}\left([0,\infty)\right)}+\int_{0}^{\infty}\phi(t,x) z^{2}(t,x)dx\right).
\end{eqnarray}

Integrating the inequality (\ref{b171}) with respect to $t$ over $\left(0, t\right)$ and using the estimates \eqref{b15}, \eqref{bb13} and \eqref{cc16}, we can obtain \eqref{b160}.

Moreover, noticing that
\begin{eqnarray}\label{bc171}
\left(r^{n-1}\theta_{x}\right)_{x}=\frac{\left(n-1\right)v\theta_{x}}{r}+r^{n-1}\theta_{xx},
\end{eqnarray}
thus taking advantage of \eqref{bb13}, Lemma 4.3, \eqref{b160} and Sobolev's inequality, we find that
\begin{eqnarray}\label{bc172}
\int_{0}^{t}\left\|r^{n-1}\theta_{x}(s)\right\|^2_{L^{\infty}\left([0,\infty)\right)}ds&&\leq C\int_{0}^{t}\left\|r^{n-1}\theta_{x}\right\|\left\|\left(r^{n-1}\theta_{x}\right)_{x}\right\|ds\nonumber\\
&&\leq C\int_{0}^{t}\left\|r^{n-1}\theta_{x}\right\|^{2}ds+C\int_{0}^{t}\left\|\left(r^{n-1}\theta_{x}\right)_{x}\right\|^{2}ds\nonumber\\
&&\leq C+C\int_{0}^{t}\int_{0}^{\infty} \left(\frac{\theta^{2}_{x}}{r^{2}}+r^{2(n-1)}\theta^{2}_{xx}\right)dxds\\
&&\leq C+C\int_{0}^{t}\int_{0}^{\infty}\frac{\kappa r^{2(n-1)}\theta^{2}_{x}}{v\theta^{2}}\cdot\theta^{2}dxds\nonumber\\
&&\leq C.\nonumber
\end{eqnarray}
This completes the proof of Lemma 4.4.
\end{proof}

The next lemma is concerned with estimates on $\left\|r^{n-1}z_{x}(t)\right\|^{2}$ and $\int_{0}^{t}\left\|r^{n-1}z_{xx}(s)\right\|^{2}ds$.
\begin{lemma} Under the assumptions listed in Lemma 2.1, we can get for any $0\leq t\leq T$ that
\begin{eqnarray}\label{b172}
\left\|r^{n-1}z_{x}(t)\right\|^{2}+\int_{0}^{t}\left\|z_{t}(s)\right\|^2ds\leq C
\end{eqnarray}
and
\begin{eqnarray}\label{bb172}
\left\|z_{x}(t)\right\|^{2}+\int_{0}^{t}\left\|r^{n-1}z_{xx}(s)\right\|^{2}ds\leq C.
\end{eqnarray}

\end{lemma}
\begin{proof} To prove \eqref{b172}, multiplying the equation $\eqref{a20}_{4}$ by $z_{t}$, we get
\begin{eqnarray}\label{b173}
\partial_{t}\left(\frac{dr^{2n-2}z^{2}_{x}}{2v^{2}}\right)+z_{t}^{2}=d\left(\frac{r^{2n-2}z_{x}z_{t}}{v^{2}}\right)_{x}
+\frac{d(n-1)r^{2n-3}uz^{2}_{x}}{v^{2}}-\frac{dr^{2n-2}\left(r^{n-1}u\right)_{x}z^{2}_{x}}{v^{3}}-\phi zz_{t}.\nonumber
\end{eqnarray}

Integrating the above equation with respect to $t$ and $x$ over $\left[0,t\right)\times[0,\infty)$ and by using the boundary condition \eqref{a22}, we have
\begin{eqnarray}\label{b174}
&&\frac{d}{2}\int_{0}^{\infty}\frac{dr^{2n-2}z^{2}_{x}}{2v^{2}}dx+\int_{0}^{t}\int_{0}^{\infty}z_{t}^{2}dxds\nonumber\\
&=&\frac{d}{2}\int_{0}^{\infty}\frac{dr^{2n-2}z^{2}_{x}}{2v^{2}}\left(0,x\right)dx
+\underbrace{\int_{0}^{t}\int_{0}^{\infty}\frac{d(n-1)r^{2n-3}uz^{2}_{x}}{v^{2}}dxds}_{I_{39}}\\
&&-\underbrace{\int_{0}^{t}\int_{0}^{\infty}\frac{dr^{2n-2}\left(r^{n-1}u\right)_{x}z^{2}_{x}}{v^{3}}dxds}_{I_{40}}
-\underbrace{\int_{0}^{t}\int_{0}^{\infty}\phi zz_{t}dxds}_{I_{41}}.\nonumber
\end{eqnarray}

To control $I_k (k=39, 40, 41)$, we can get from \eqref{a25}, \eqref{b15} and \eqref{b152-2} that
\begin{eqnarray}\label{b175}
\left |I_{39}\right|&&\leq C\int_{0}^{t}\int_{0}^{\infty}\frac{r^{2n-2}z^{2}_{x}}{v^{2}}dxds\leq C,
\end{eqnarray}
and by employing \eqref{bb126}, we can infer that
\begin{eqnarray}\label{b176}
\left|I_{40}\right|&&\leq C\int_{0}^{t}\int_{0}^{\infty}\frac{r^{2n-2}z^{2}_{x}}{v^{2}}\left(r^{-1}\left|u\right|+r^{n-1}\left|u_{x}\right|\right)dxds\nonumber\\
&&\leq C\int_{0}^{t}\left(1+\left\|r^{n-1}u_{x}\right\|_{L^{\infty}\left([0,\infty)\right)}^{2}\right)
\int_{0}^{\infty}\frac{r^{2n-2}z^{2}_{x}}{v^{2}}dxds\\
&&\leq C+C\int_{0}^{t}\left\|r^{n-1}u_{x}\right\|_{L^{\infty}\left([0,\infty)\right)}^{2}\int_{0}^{\infty} \frac{r^{2n-2}z^{2}_{x}}{v^{2}}dxds.\nonumber
\end{eqnarray}

As to the term $I_{41}$, we get from \eqref{b15} that
\begin{eqnarray}\label{b178}
\left|I_{41}\right|&&\leq\frac{1}{2}\int_{0}^{t}\int_{0}^{\infty}z_{t}^{2}dxds+C\int_{0}^{t}\int_{0}^{\infty}\phi^{2}z^{2}dxds\nonumber\\
&&\leq\frac{1}{2}\int_{0}^{t}\int_{0}^{\infty}z_{t}^{2}dxds+C\int_{0}^{t}\int_{0}^{\infty}\phi z^{2}dxds\\
&&\leq\frac{1}{2}\int_{0}^{t}\int_{0}^{\infty}z_{t}^{2}dxds+C.\nonumber
\end{eqnarray}

Substituting \eqref{b175}, \eqref{b176} and \eqref{b178} into \eqref{b174}  and by using Gronwall's inequality, we can obtain \eqref{b172}.

Now we turn to prove \eqref{bb172}. For this purpose, multiplying $\eqref{a20}_{4}$ by $z_{xx}$, we have
\begin{eqnarray}\label{bc173}
\partial_{t}\left(\frac{z^{2}_{x}}{2}\right)-\left(z_{t}z_{x}\right)_{x}+\frac{dr^{2n-2}z^{2}_{xx}}{v^{2}}=-\frac{(2n-2)dr^{n-2}z_{x}z_{xx}}{v^{2}}
+\frac{2dr^{2n-2}z_{x}v_{x}z_{xx}}{v^{3}}+\phi zz_{xx}.
\end{eqnarray}

Integrating the above equation with respect to $t$ and $x$ over $\left[0,t\right)\times[0,\infty)$ and by using the boundary condition \eqref{a22} and Lemma 3.4, we arrive at
\begin{eqnarray}\label{bc174}
&&\left\|z_{x}(t)\right\|^{2}+\int_{0}^{t}\left\|r^{n-1}z_{xx}(s)\right\|^{2}ds\nonumber\\
&\leq& C+\underbrace{\int_{0}^{t}\int_{0}^{\infty}r^{n-2}\left|z_{x}z_{xx}\right|dxds}_{I_{42}}
+\underbrace{\int_{0}^{t}\int_{0}^{\infty}r^{2n-2}\left|z_{x}v_{x}z_{xx}\right|dxds}_{I_{43}}\\
&&+\underbrace{\int_{0}^{t}\int_{0}^{\infty}\phi z\left|z_{xx}\right|dxds}_{I_{44}}.\nonumber
\end{eqnarray}
The terms $I_k (42\leq k\leq 44)$ can be estimated term by term by employing Cauchy's inequality, Sobolev's inequality, Lemma 2.1 and Lemma 4.3 as follows:
\begin{eqnarray}\label{bc175}
\left|I_{42}\right|&&\leq\frac{1}{8}\int_{0}^{t}\left\|r^{n-1}z_{xx}(s)\right\|^{2}ds+C\int_{0}^{t}\int_{0}^{\infty}r^{-2}z_{x}^{2}dxds\nonumber\\
&&\leq\frac{1}{8}\int_{0}^{t}\left\|r^{n-1}z_{xx}(s)\right\|^{2}ds+C,
\end{eqnarray}
\begin{eqnarray}\label{bc176}
\left|I_{43}\right|&&\leq\frac{1}{8}\int_{0}^{t}\left\|r^{n-1}z_{xx}(s)\right\|^{2}ds+C\int_{0}^{t}\int_{0}^{\infty}r^{2n-2}v_{x}^{2}z_{x}^{2}dxds\nonumber\\
&&\leq\frac{1}{8}\int_{0}^{t}\left\|r^{n-1}z_{xx}(s)\right\|^{2}ds+C\int_{0}^{t}\left\|z_{x}\right\|^{2}_{L^{\infty}\left([0,\infty)\right)}ds\nonumber\\
&&\leq\frac{1}{8}\int_{0}^{t}\left\|r^{n-1}z_{xx}(s)\right\|^{2}ds+C\int_{0}^{t}\left\|z_{x}\right\|\left\|z_{xx}\right\|ds\\
&&\leq\frac{1}{4}\int_{0}^{t}\left\|r^{n-1}z_{xx}(s)\right\|^{2}ds+C,\nonumber
\end{eqnarray}
and
\begin{eqnarray}\label{bc177}
\left|I_{44}\right|&&\leq\frac{1}{8}\int_{0}^{t}\left\|r^{n-1}z_{xx}(s)\right\|^{2}ds+C\int_{0}^{t}\int_{0}^{\infty}\phi^{2}z^{2}dxds\nonumber\\
&&\leq\frac{1}{8}\int_{0}^{t}\left\|r^{n-1}z_{xx}(s)\right\|^{2}ds+C\left\|\theta\right\|^{\beta}_{\infty}\int_{0}^{t}\int_{0}^{\infty}\phi z^{2}ds\\
&&\leq\frac{1}{8}\int_{0}^{t}\left\|r^{n-1}z_{xx}(s)\right\|^{2}ds+C.\nonumber
\end{eqnarray}
Then \eqref{bb172} follows from \eqref{bc174}-\eqref{bc177} and this completes the proof of Lemma 4.5.
\end{proof}

Finally, the following lemma gives a nice bound on the term $\int_{0}^{t}\left\|r^{n-1}v_{x}(s)\right\|^2_{L^{\infty}\left([0,\infty)\right)}ds$, which will be used in Section 5.
\begin{lemma} Under the assumptions listed in Lemma 2.1, we can get for any $0\leq t\leq T$ that
\begin{eqnarray}\label{bd172}
\int_{0}^{\infty}\frac{r^{2n-2}v^{2}_{xx}}{v^{2}}dx+\int_{0}^{t}\int_{0}^{\infty}\frac{Rr^{2n-2}\theta v^{2}_{xx}}{v^{3}}dxds\leq C
\end{eqnarray}
and
\begin{eqnarray}\label{bd173}
\int_{0}^{t}\left\|r^{n-1}v_{x}(s)\right\|^2_{L^{\infty}\left([0,\infty)\right)}ds\leq C.
\end{eqnarray}

\end{lemma}

\begin{proof}
Differentiating $\eqref{bd44}$ with respect to $x$ once and multiplying it by $\frac{r^{n-1}v_{xx}}{v}$, we obtain that
\begin{eqnarray}\label{bd1}
&&\frac{\alpha}{2}\partial_{t}\left[\frac{r^{2n-2}v^{2}_{xx}}{v^{2}}\right]+\frac{Rr^{2n-2}\theta v^{2}_{xx}}{v^{3}}\nonumber\\
&=&\frac{r^{n-1}u_{tx}v_{xx}}{v}-\alpha(n-1)r^{n-2}v_{xx}\left(\frac{\left(r^{n-1}u\right)_{x}}{v}\right)_{x}\nonumber\\
&&+\frac{\alpha(n-1)r^{2n-3}uv^{2}_{xx}}{v^{2}}
+\frac{\alpha r^{2n-2}v_{xx}}{v}\left(\frac{v^{2}_{x}}{v^{2}}\right)_{t}\\
&&+(n-1)r^{n-2}v_{xx}\left(\frac{R\theta_{x}}{v}-\frac{R\theta v_{x}}{v^{2}}+\frac{4}{3}a\theta^{3}\theta_{x}\right)\nonumber\\
&&+\frac{Rr^{2n-2}v_{xx}\theta_{xx}}{v^{2}}
-\frac{2Rr^{2n-2}v_{xx}v_{x}\theta_{x}}{v^{3}}\nonumber\\
&&+\frac{2Rr^{2n-2}\theta v_{xx}v^{2}_{x}}{v^{4}}+\frac{4ar^{2n-2}\theta^{2}v_{xx}\theta^{2}_{x}}{v}+\frac{4a\theta^{3}r^{2n-2}v_{xx}\theta_{xx}}{3v},\nonumber
\end{eqnarray}
where we have used \eqref{bbb44} and the fact that
\begin{eqnarray*}
 \frac{\partial^{2} P(v,\theta)}{\partial x^{2}}=\frac{R\theta_{xx}}{v}-\frac{2R\theta_{x}v_{x}}{v^{2}}-\frac{R\theta v_{xx}}{v^{2}}+\frac{2R\theta v^{2}_{x}}{v^{3}}+4a\theta^{2}\theta^{2}_{x}
+\frac{4}{3}a\theta^{3}\theta_{xx}.
\end{eqnarray*}

Integrating \eqref{bd1} with respect to $t$ and $x$ over $\left[0,t\right)\times[0, +\infty)$, we arrive at
\begin{eqnarray}\label{bd2}
&&\int_{0}^{\infty}\frac{r^{2n-2}v^{2}_{xx}}{v^{2}}dx+\int_{0}^{t}\int_{0}^{\infty}\frac{Rr^{2n-2}\theta v^{2}_{xx}}{v^{3}}dxds\nonumber\\
&\leq&\int_{0}^{\infty}\frac{r^{2n-2}v^{2}_{xx}}{v^{2}}\left(0,x\right)dx
+\underbrace{C\int_{0}^{t}\int_{0}^{\infty}\frac{r^{n-1}\left|u_{tx}v_{xx}\right|}{v}dxds}_{I_{45}}\nonumber\\
&&+\underbrace{C\int_{0}^{t}\int_{0}^{\infty}r^{n-2}\left|v_{xx}\left(\frac{\left(r^{n-1}u\right)_{x}}{v}\right)_{x}\right|dxds}_{I_{46}}
+\underbrace{C\int_{0}^{t}\int_{0}^{\infty}\frac{r^{2n-3}\left|u\right|v^{2}_{xx}}{v^{2}}dxds}_{I_{47}}\nonumber\\
&&+\underbrace{C\int_{0}^{t}\int_{0}^{\infty}\left|\frac{r^{2n-2}v_{xx}}{v}\left(\frac{v^{2}_{x}}{v^{2}}\right)_{t}\right|dxds}_{I_{48}}
+\underbrace{C\int_{0}^{t}\int_{0}^{\infty}\frac{r^{2n-2}\left|v_{xx}\theta_{xx}\right|}{v^{2}}dxds}_{I_{49}}\\
&&+\underbrace{C\int_{0}^{t}\int_{0}^{\infty}r^{n-2}\left|v_{xx}\left(\frac{R\theta_{x}}{v}-\frac{R\theta v_{x}}{v^{2}}+\frac{4}{3}a\theta^{3}\theta_{x}\right)\right|dxds}_{I_{50}}\nonumber\\
&&+\underbrace{C\int_{0}^{t}\int_{0}^{\infty}\frac{r^{2n-2}\left|v_{xx}\theta_{x}v_{x}\right|}{v^{3}}dxds}_{I_{51}}
+\underbrace{C\int_{0}^{t}\int_{0}^{\infty}\frac{r^{2n-2}\theta v^{2}_{x}\left|v_{xx}\right|}{v^{4}}dxds}_{I_{52}}\nonumber\\
&&+\underbrace{C\int_{0}^{t}\int_{0}^{\infty}\frac{r^{2n-2}\theta^{2}\theta^{2}_{x}\left|v_{xx}\right|}{v}dxds}_{I_{53}}
+\underbrace{C\int_{0}^{t}\int_{0}^{\infty}\frac{r^{2n-2}\theta^{3}\left|v_{xx}\theta_{xx}\right|}{v}dxds}_{I_{54}}.\nonumber
\end{eqnarray}

Now we turn to estimate the terms $I_{45}$-$I_{54}$ term by term. For this purpose, we can get first from \eqref{bb158}, Lemma 4.3 and \eqref{b19} with $m=\frac{1}{2}$ that
\begin{eqnarray}\label{bd3}
I_{45}&&\leq\frac{1}{20}\int_{0}^{t}\int_{0}^{\infty}\frac{Rr^{2(n-1)}\theta v^{2}_{xx}}{v^{^{3}}}dxds
+C\int_{0}^{t}V(s)\int_{0}^{\infty}\frac{r^{2n-2}v^{2}_{xx}}{v^{2}}dxds+C\int_{0}^{t}\int_{0}^{\infty}\frac{r^{2n-2}u_{tx}^{2}}{v}\nonumber\\
&&\leq\frac{1}{20}\int_{0}^{t}\int_{0}^{\infty}\frac{Rr^{2(n-1)}\theta v^{2}_{xx}}{v^{^{3}}}dxds
+C\int_{0}^{t}V(s)\int_{0}^{\infty}\frac{r^{2n-2}v^{2}_{xx}}{v^{2}}dxds+C.
\end{eqnarray}

Next, due to the fact that
\begin{eqnarray*}
\left(\frac{\left(r^{n-1}u\right)_{x}}{v}\right)_{x}=\frac{2(n-1)u_{x}}{r}-\frac{(n-1)uv}{r^{n+1}}-\frac{r^{n-1}u_{x}v_{x}}{v^{2}}+\frac{r^{n-1}u_{xx}}{v},
\end{eqnarray*}
we can obtain from \eqref{bb13}, \eqref{b19}, \eqref{bz57}, \eqref{c1}, \eqref{c17} and \eqref{b151} that
\begin{eqnarray}\label{bd4}
I_{46}&\leq& C\int_{0}^{t}\int_{0}^{\infty}\left(\frac{v\left|uv_{xx}\right|}{r^{3}}+r^{n-3}\left|u_{x}v_{xx}\right|
+\frac{r^{2n-3}\left|v_{xx}v_{x}u_{x}\right|}{v^{2}}+\frac{r^{2n-3}\left|v_{xx}u_{xx}\right|}{v}\right)dxds\nonumber\\
&\leq&\frac{1}{20}\int_{0}^{t}\int_{0}^{\infty}\frac{Rr^{2(n-1)}\theta v^{2}_{xx}}{v^{^{3}}}dxds
+C\int_{0}^{t}V(s)\int_{0}^{\infty}\frac{r^{2n-2}v^{2}_{xx}}{v^{2}}dxds\nonumber\\
&&+C\int_{0}^{t}\int_{0}^{\infty}\left(\frac{vu^{2}}{r^{2}\theta}+\frac{r^{2(n-1)}u^{2}_{x}}{v\theta}+r^{2n-2}u^{2}_{x}v^{2}_{x}
+r^{2n-2}u^{2}_{xx}\right)dxds\\
&\leq&\frac{1}{20}\int_{0}^{t}\int_{0}^{\infty}\frac{Rr^{2(n-1)}\theta v^{2}_{xx}}{v^{^{3}}}dxds
+C\int_{0}^{t}V(s)\int_{0}^{\infty}\frac{r^{2n-2}v^{2}_{xx}}{v^{2}}dxds+C\nonumber\\
&&+C\int_{0}^{t}\left\|r^{n-1}u_{x}\right\|^{2}_{L^{\infty}\left([0,\infty)\right)}\int_{0}^{\infty}v_{x}^{2}dxds\nonumber\\
&\leq&\frac{1}{20}\int_{0}^{t}\int_{0}^{\infty}\frac{Rr^{2(n-1)}\theta v^{2}_{xx}}{v^{^{3}}}dxds
+C\int_{0}^{t}V(s)\int_{0}^{\infty}\frac{r^{2n-2}v^{2}_{xx}}{v^{2}}dxds+C.\nonumber
\end{eqnarray}

Now, it is apparent that
\begin{eqnarray}\label{bd5}
I_{47}&\leq& \frac{1}{20}\int_{0}^{t}\int_{0}^{\infty}\frac{Rr^{2(n-1)}\theta v^{2}_{xx}}{v^{^{3}}}dxds
+C\int_{0}^{t}V(s)\int_{0}^{\infty}\frac{r^{2n-2}v^{2}_{xx}}{v^{2}}dxds
+C\int_{0}^{t}\int_{0}^{\infty}\frac{r^{2n-4}u^{2}v^{2}_{xx}}{v^{2}}\nonumber\\
&\leq& \frac{1}{20}\int_{0}^{t}\int_{0}^{\infty}\frac{Rr^{2(n-1)}\theta v^{2}_{xx}}{v^{^{3}}}dxds
+C\int_{0}^{t}\left(V(s)+\left\|\frac{u}{r}\right\|^{2}_{L^{\infty}\left([0,\infty)\right)}\right)\int_{0}^{\infty}\frac{r^{2n-2}v^{2}_{xx}}{v^{2}}dxds,
\end{eqnarray}
where we have also used \eqref{b19} with $m=\frac{1}{2}$.

For the term $I_{48}$, since
\begin{eqnarray*}
\left(\frac{v^{2}_{x}}{v^{2}}\right)_{t}=\frac{2v_{x}}{v}\cdot\frac{v_{tx}v-v_{x}v_{t}}{v^{2}}
=\frac{2v_{x}\left(r^{n-1}u\right)_{xx}}{v^{2}}-\frac{2v_{x}^{2}\left(r^{n-1}u\right)_{x}}{v^{^{2}}},
\end{eqnarray*}
then we have
\begin{eqnarray}\label{bd6}
I_{48}\leq \underbrace{C\int_{0}^{t}\int_{0}^{\infty}\frac{r^{2n-2}\left|v^{2}_{x}v_{xx}\left(r^{n-1}u\right)_{x}\right|}{v^{4}}dxds}_{I_{48}^{1}}
+\underbrace{C\int_{0}^{t}\int_{0}^{\infty}\frac{r^{2n-2}\left|v_{x}v_{xx}\left(r^{n-1}u\right)_{xx}\right|}{v^{3}}dxds}_{I_{48}^{2}}.
\end{eqnarray}
As for the term $I_{48}^{1}$, we can deduce that
\begin{eqnarray}\label{bd7}
I_{48}^{1}&\leq& \frac{1}{80}\int_{0}^{t}\int_{0}^{\infty}\frac{Rr^{2(n-1)}\theta v^{2}_{xx}}{v^{^{3}}}dxds
+C\int_{0}^{t}V(s)\int_{0}^{\infty}\frac{r^{2n-2}v^{2}_{xx}}{v^{2}}dxds\nonumber\\
&&+C\int_{0}^{t}\int_{0}^{\infty}r^{2n-2}v_{x}^{4}\left[\left(r^{n-1}u\right)_{x}\right]^{2}dxds\nonumber\\
&\leq& \frac{1}{80}\int_{0}^{t}\int_{0}^{\infty}\frac{Rr^{2(n-1)}\theta v^{2}_{xx}}{v^{^{3}}}dxds
+C\int_{0}^{t}V(s)\int_{0}^{\infty}\frac{r^{2n-2}v^{2}_{xx}}{v^{2}}dxds\nonumber\\
&&+C\int_{0}^{t}\left\|v_{x}\right\|^2_{L^{\infty}\left([0,\infty)\right)} \left\|\left(r^{n-1}u\right)_{x}\right\|^2_{L^{\infty}\left([0,\infty)\right)}ds\\
&\leq& \frac{1}{80}\int_{0}^{t}\int_{0}^{\infty}\frac{Rr^{2(n-1)}\theta v^{2}_{xx}}{v^{^{3}}}dxds
+C\int_{0}^{t}V(s)\int_{0}^{\infty}\frac{r^{2n-2}v^{2}_{xx}}{v^{2}}dxds\nonumber\\
&&+C\int_{0}^{t}\left\|v_{x}\right\|\left\|v_{xx}\right\|\left\|\left(r^{n-1}u\right)_{x}\right\|\left\|\left(r^{n-1}u_{x}\right)_{xx}\right\|ds\nonumber\\
&\leq& \frac{1}{40}\int_{0}^{t}\int_{0}^{\infty}\frac{Rr^{2(n-1)}\theta v^{2}_{xx}}{v^{^{3}}}dxds
+C\int_{0}^{t}V(s)\int_{0}^{\infty}\frac{r^{2n-2}v^{2}_{xx}}{v^{2}}dxds\nonumber\\
&&+C\int_{0}^{t}\int_{0}^{\infty}\left[\left(r^{n-1}u\right)_{xx}\right]^{2}dxds.\nonumber
\end{eqnarray}
Here we have used \eqref{b152} and the fact that
\begin{eqnarray}\label{bd8}
\left\|\left(r^{n-1}u\right)_{x}\right\|^{2}\leq C\int_{0}^{\infty}\left(\frac{u^{2}}{r^{2}}+r^{2n-2}u_{x}^{2}\right)dx\leq C.
\end{eqnarray}
On the other hand, we can infer from \eqref{a25}, \eqref{bb13}, \eqref{cc16}, \eqref{bd9}, \eqref{b151} and \eqref{b152} that
\begin{eqnarray}\label{bd10}
&&\int_{0}^{t}\int_{0}^{\infty}\left[\left(r^{n-1}u_{x}\right)_{xx}\right]^{2}dxds\nonumber\\
&\leq& C\int_{0}^{t}\int_{0}^{\infty}\left(\frac{u^{2}}{r^{2n+2}}+\frac{u^{2}v_{x}^{2}}{r^{2}}+\frac{u_{x}^{2}}{r^{2}}+r^{2n-2}u^{2}_{xx}\right)dxds\nonumber\\
&\leq& C+C\int_{0}^{t}\int_{0}^{\infty}\frac{u^{2}}{r^{2}\theta}\cdot\theta dxds+C\int_{0}^{t}\left\|\frac{u}{r}\right\|^2_{L^{\infty}\left([0,\infty)\right)}\int_{0}^{\infty}v_{x}^{2}dxds\\
&&+C\int_{0}^{t}\int_{0}^{\infty}\frac{r^{2n-2}u_{x}^{2}}{v\theta}\cdot\theta dxds\nonumber\\
&\leq& C.\nonumber
\end{eqnarray}

The combination of \eqref{bd7} and \eqref{bd10} shows that
\begin{eqnarray}\label{bd11}
I_{48}^{1}
&\leq& \frac{1}{40}\int_{0}^{t}\int_{0}^{\infty}\frac{Rr^{2(n-1)}\theta v^{2}_{xx}}{v^{^{3}}}dxds
+C\int_{0}^{t}V(s)\int_{0}^{\infty}\frac{r^{2n-2}v^{2}_{xx}}{v^{2}}dxds+C.
\end{eqnarray}

Moreover, it follows from \eqref{b19}, \eqref{cc16}, \eqref{b152}, \eqref{b152-2} and \eqref{bd9} that
\begin{eqnarray}\label{bd12}
I_{48}^{2}
&\leq& \frac{1}{80}\int_{0}^{t}\int_{0}^{\infty}\frac{Rr^{2(n-1)}\theta v^{2}_{xx}}{v^{^{3}}}dxds
+C\int_{0}^{t}V(s)\int_{0}^{\infty}\frac{r^{2n-2}v^{2}_{xx}}{v^{2}}dxds\nonumber\\
&&+C\int_{0}^{t}\int_{0}^{\infty}r^{2n-2}v^{2}_{x}\left[\left(r^{n-1}u\right)_{xx}\right]^{2}dxds\nonumber\\
&\leq& \frac{1}{80}\int_{0}^{t}\int_{0}^{\infty}\frac{Rr^{2(n-1)}\theta v^{2}_{xx}}{v^{^{3}}}dxds
+C\int_{0}^{t}V(s)\int_{0}^{\infty}\frac{r^{2n-2}v^{2}_{xx}}{v^{2}}dxds\nonumber\\
&&+C\int_{0}^{t}\int_{0}^{\infty}\left(\frac{v^{2}_{x}u^{2}}{r^{4}}+\frac{r^{2n-2}u^{2}v_{x}^{4}}{r^{2}}+\frac{r^{2n-2}u^{2}_{x}v_{x}^{2}}{r^{2}}
+r^{4n-4}v^{2}_{x}u_{xx}^{2}\right)\nonumber\\
&\leq& \frac{1}{80}\int_{0}^{t}\int_{0}^{\infty}\frac{Rr^{2(n-1)}\theta v^{2}_{xx}}{v^{^{3}}}dxds
+C\int_{0}^{t}V(s)\int_{0}^{\infty}\frac{r^{2n-2}v^{2}_{xx}}{v^{2}}dxds\nonumber\\
&&+C\int_{0}^{t}\left\|\frac{u}{r}\right\|^2_{L^{\infty}\left([0,\infty)\right)}\int_{0}^{\infty}v_{x}^{2}dxds
+C\int_{0}^{t}\left\|v_{x}\right\|^2_{L^{\infty}\left([0,\infty)\right)} \int_{0}^{\infty}r^{2n-2}v_{x}^{2}dxds\\
&&+C\int_{0}^{t}\left\|v_{x}\right\|^2_{L^{\infty}\left([0,\infty)\right)}\int_{0}^{\infty}r^{2n-2}u_{x}^{2}dxds
+C\int_{0}^{t}\left\|v_{x}\right\|^2_{L^{\infty}\left([0,\infty)\right)}\int_{0}^{\infty}r^{4n-4}u_{xx}^{2}dxds\nonumber\\
&\leq& \frac{1}{80}\int_{0}^{t}\int_{0}^{\infty}\frac{Rr^{2(n-1)}\theta v^{2}_{xx}}{v^{^{3}}}dxds
+C\int_{0}^{t}V(s)\int_{0}^{\infty}\frac{r^{2n-2}v^{2}_{xx}}{v^{2}}dxds\nonumber\\
&&+C+C\int_{0}^{t}\left\|v_{x}\right\|\left\|v_{xx}\right\|ds+C\int_{0}^{t}\left\|v_{x}\right\|\left\|v_{xx}\right\|\int_{0}^{\infty}r^{4n-4}u_{xx}^{2}dxds\nonumber\\
&\leq& \frac{1}{40}\int_{0}^{t}\int_{0}^{\infty}\frac{Rr^{2(n-1)}\theta v^{2}_{xx}}{v^{^{3}}}dxds
+C\int_{0}^{t}V(s)\int_{0}^{\infty}\frac{r^{2n-2}v^{2}_{xx}}{v^{2}}dxds\nonumber\\
&&+C+C\int_{0}^{t}\int_{0}^{\infty}v_{x}^{2}dxds
+C\int_{0}^{t}\left(\int_{0}^{\infty}r^{4n-4}u_{xx}^{2}dx\right) \left(\int_{0}^{\infty}\frac{r^{2n-2}v^{2}_{xx}}{v^{2}}dx\right)ds.\nonumber
\end{eqnarray}
Furthermore, \eqref{b19} and \eqref{b152} tell us that
\begin{eqnarray}\label{bd13}
\int_{0}^{t}\int_{0}^{\infty}v_{x}^{2}dxds&&\leq C\int_{0}^{t}\int_{0}^{\infty}\theta v_{x}^{2}dxds+C\int_{0}^{t}V(s)\int_{0}^{\infty}v_{x}^{2}dxds
\leq C,
\end{eqnarray}
and we can get from \eqref{bd12} and \eqref{bd13} that
\begin{eqnarray}\label{bd14}
I_{48}^{2}
&\leq& \frac{1}{40}\int_{0}^{t}\int_{0}^{\infty}\frac{Rr^{2(n-1)}\theta v^{2}_{xx}}{v^{^{3}}}dxds
+C\int_{0}^{t}V(s)\int_{0}^{\infty}\frac{r^{2n-2}v^{2}_{xx}}{v^{2}}dxds\nonumber\\
&&+C+C\int_{0}^{t}\left(\int_{0}^{\infty}r^{4n-4}u_{xx}^{2}dx\right)\left(\int_{0}^{\infty}\frac{r^{2n-2}v^{2}_{xx}}{v^{2}}dx\right)ds.
\end{eqnarray}

Substituting \eqref{bd7} and \eqref{bd14} into \eqref{bd6}, gives
\begin{eqnarray}\label{bd15}
I_{48}
&\leq& \frac{1}{20}\int_{0}^{t}\int_{0}^{\infty}\frac{Rr^{2(n-1)}\theta v^{2}_{xx}}{v^{^{3}}}dxds
+C\int_{0}^{t}V(s)\int_{0}^{\infty}\frac{r^{2n-2}v^{2}_{xx}}{v^{2}}dxds\nonumber\\
&&+C+C\int_{0}^{t}\left(\int_{0}^{\infty}r^{4n-4}u_{xx}^{2}dx\right)\left(\int_{0}^{\infty}\frac{r^{2n-2}v^{2}_{xx}}{v^{2}}dx\right)ds.
\end{eqnarray}

For $I_{49}$, by making use of \eqref{b19} and \eqref{b160}, we find that
\begin{eqnarray}\label{bd16}
I_{49}
&\leq& \frac{1}{20}\int_{0}^{t}\int_{0}^{\infty}\frac{Rr^{2(n-1)}\theta v^{2}_{xx}}{v^{^{3}}}dxds
+C\int_{0}^{t}V(s)\int_{0}^{\infty}\frac{r^{2n-2}v^{2}_{xx}}{v^{2}}dxds\nonumber\\
&&+C\int_{0}^{t}\int_{0}^{\infty}r^{2n-2}\theta_{xx}^{2}dxds\\
&\leq& \frac{1}{20}\int_{0}^{t}\int_{0}^{\infty}\frac{Rr^{2(n-1)}\theta v^{2}_{xx}}{v^{^{3}}}dxds
+C\int_{0}^{t}V(s)\int_{0}^{\infty}\frac{r^{2n-2}v^{2}_{xx}}{v^{2}}dxds+C.\nonumber
\end{eqnarray}

As for the term $I_{50}$, we have from \eqref{bb13}, \eqref{b151} and \eqref{b152} that
\begin{eqnarray}\label{bd17}
I_{50}
&\leq& C\int_{0}^{t}\int_{0}^{\infty}\left(\frac{r^{n-2}\left|v_{xx}\theta_{x}\right|}{v}+\frac{r^{n-2}\theta\left|v_{xx}v_{x}\right|}{v^{2}}
+\theta^{3}\left|v_{xx}\theta_{x}\right|\right)\nonumber\\
&\leq& \frac{1}{20}\int_{0}^{t}\int_{0}^{\infty}\frac{Rr^{2(n-1)}\theta v^{2}_{xx}}{v^{^{3}}}dxds
+C\int_{0}^{t}\int_{0}^{\infty}\left[\frac{r^{2n-2}\theta^{2}_{x}}{v\theta^{2}}\cdot\left(\theta+\theta^{7}\right)+\theta v^{2}_{x}\right]dxds\\
&\leq& \frac{1}{20}\int_{0}^{t}\int_{0}^{\infty}\frac{Rr^{2(n-1)}\theta v^{2}_{xx}}{v^{^{3}}}dxds+C.\nonumber
\end{eqnarray}

Now for $I_{51}$ and $I_{52}$, in view of \eqref{b152} and \eqref{bd13} and similar to that of \eqref{bd12}, we can conclude that
\begin{eqnarray}\label{bd18}
I_{51}
&\leq& \frac{1}{40}\int_{0}^{t}\int_{0}^{\infty}\frac{Rr^{2(n-1)}\theta v^{2}_{xx}}{v^{^{3}}}dxds
+C\int_{0}^{t}V(s)\int_{0}^{\infty}\frac{r^{2n-2}v^{2}_{xx}}{v^{2}}dxds\nonumber\\
&&+C\int_{0}^{t}\left\|v_{x}\right\|^2_{L^{\infty}\left([0,\infty)\right)}\int_{0}^{\infty}r^{2n-2}\theta_{x}^{2}dxds\nonumber\\
&\leq& \frac{1}{40}\int_{0}^{t}\int_{0}^{\infty}\frac{Rr^{2(n-1)}\theta v^{2}_{xx}}{v^{^{3}}}dxds
+C\int_{0}^{t}V(s)\int_{0}^{\infty}\frac{r^{2n-2}v^{2}_{xx}}{v^{2}}dxds\\
&&+C\int_{0}^{t}\left\|v_{x}\right\|\left\|v_{xx}\right\|ds\nonumber\\
&\leq& \frac{1}{20}\int_{0}^{t}\int_{0}^{\infty}\frac{Rr^{2(n-1)}\theta v^{2}_{xx}}{v^{^{3}}}dxds
+C\int_{0}^{t}V(s)\int_{0}^{\infty}\frac{r^{2n-2}v^{2}_{xx}}{v^{2}}dxds+C\nonumber
\end{eqnarray}
and
\begin{eqnarray}\label{bd19}
I_{52}
&\leq& \frac{1}{40}\int_{0}^{t}\int_{0}^{\infty}\frac{Rr^{2(n-1)}\theta v^{2}_{xx}}{v^{^{3}}}dxds
+C\int_{0}^{t}\left\|v_{x}\right\|^2_{L^{\infty}\left([0,\infty)\right)}\int_{0}^{\infty}r^{2n-2}v_{x}^{2}dxds\nonumber\\
&\leq& \frac{1}{40}\int_{0}^{t}\int_{0}^{\infty}\frac{Rr^{2(n-1)}\theta v^{2}_{xx}}{v^{^{3}}}dxds
+C\int_{0}^{t}\left\|v_{x}\right\|\left\|v_{xx}\right\|ds\\
&\leq& \frac{1}{20}\int_{0}^{t}\int_{0}^{\infty}\frac{Rr^{2(n-1)}\theta v^{2}_{xx}}{v^{^{3}}}dxds
+C\int_{0}^{t}V(s)\int_{0}^{\infty}\frac{r^{2n-2}v^{2}_{xx}}{v^{2}}dxds+C.\nonumber
\end{eqnarray}

For the term $I_{53}$, one can get from \eqref{b152} and \eqref{bb160} that
\begin{eqnarray}\label{bd20}
I_{53}
&\leq& \frac{1}{20}\int_{0}^{t}\int_{0}^{\infty}\frac{Rr^{2(n-1)}\theta v^{2}_{xx}}{v^{^{3}}}dxds
+C\int_{0}^{t}\int_{0}^{\infty}r^{2n-2}\theta^{4}_{x}dxds\nonumber\\
&\leq& \frac{1}{20}\int_{0}^{t}\int_{0}^{\infty}\frac{Rr^{2(n-1)}\theta v^{2}_{xx}}{v^{^{3}}}dxds
+C\int_{0}^{t}\left\|r^{n-1}\theta_{x}(s)\right\|^2_{L^{\infty}\left([0,\infty)\right)}\int_{0}^{\infty}\theta_{x}^{2}dxds\\
&\leq& \frac{1}{20}\int_{0}^{t}\int_{0}^{\infty}\frac{Rr^{2(n-1)}\theta v^{2}_{xx}}{v^{^{3}}}dxds
+C.\nonumber
\end{eqnarray}

Furthermore, we can conclude from \eqref{b160} that
\begin{eqnarray}\label{bd21}
I_{54}
&\leq& \frac{1}{20}\int_{0}^{t}\int_{0}^{\infty}\frac{Rr^{2(n-1)}\theta v^{2}_{xx}}{v^{^{3}}}dxds
+C\int_{0}^{t}\int_{0}^{\infty}r^{2n-2}\theta^{2}_{xx}dxds\nonumber\\
&\leq& \frac{1}{20}\int_{0}^{t}\int_{0}^{\infty}\frac{Rr^{2(n-1)}\theta v^{2}_{xx}}{v^{^{3}}}dxds
+C.
\end{eqnarray}

Combining \eqref{bd2}-\eqref{bd21} and taking full advantage of \eqref{cc16}, \eqref{b152} and Gronwall's inequality, we can obtain \eqref{bd172}.

It remains to estimate the term $\int_{0}^{t}\left\|r^{n-1}v_{x}(s)\right\|^2_{L^{\infty}\left([0,\infty)\right)}ds$. To this end, noticing that
\begin{eqnarray*}
\left(r^{n-1}v_{x}\right)_{x}=(n-1)r^{-1}vv_{x}+r^{n-1}v_{xx},
\end{eqnarray*}
then we have
\begin{eqnarray}\label{bd23}
\int_{0}^{t}\left\|r^{n-1}v_{x}\right\|^2_{L^{\infty}\left([0,\infty)\right)}ds
&\leq& C\int_{0}^{t}\left\|r^{n-1}v_{x}\right\|\left\|\left(r^{n-1}v_{x}\right)_{x}\right\|ds\nonumber\\
&\leq& C\int_{0}^{t}\int_{0}^{\infty}\left(r^{2n-2}v_{x}^{2}+v_{x}^{2}+r^{2n-2}v_{xx}^{2}\right)dxds\nonumber\\
&\leq& C+C\int_{0}^{t}\left(\int_{0}^{\infty}\theta r^{2n-2}v_{x}^{2}+\theta r^{2n-2}v_{xx}^{2}\right)dxds\\
&&+C\int_{0}^{t}V\left(s\right)\left(\int_{0}^{\infty}r^{2n-2}v_{x}^{2}dx+\int_{0}^{\infty}r^{2n-2}v_{xx}^{2}dx\right)ds\nonumber\\
&\leq& C.\nonumber
\end{eqnarray}
Here we have used \eqref{b19}, \eqref{b43}, \eqref{b152}, \eqref{bd172} and \eqref{bd13}. This completes the proof of Lemma 4.6.
\end{proof}

\section{Lower bound of $\theta\left(t,x\right)$}
This main purpose of this section is to derive a local in time lower positive bound on the absolute temperature $\theta\left(t, x\right)$, which is the main content of the  following lemma
\begin{lemma} Under the assumptions stated in Lemma 2.1, for each $0\leq s\leq t\leq T$ and $x\in[0,\infty)$, the following estimate
  \begin{eqnarray}\label{b180}
  \theta\left(t,x\right)\geq \frac{\min\limits_{x\in[0,\infty)}\{\theta(s,x)\}}{C+C\left(t-s+1\right)\min\limits_{x\in[0,\infty)}\{\theta(s,x)\}}
  \end{eqnarray}
holds for some positive constant $C$ which depends only on the fixed constants $\mu$, $\lambda_{1}$, $\lambda$, $K$, $A$, $d$, $R$, $c_{v}$, $a$, $\kappa_{1}$, $\kappa_{2}$, $n$ and the initial data $(v_0(x), u_0(x), \theta_0(x), z_0(x))$.
\end{lemma}
\begin{proof}
To start with,  we can deduce from \eqref{a5} and the fact $\alpha=2\mu+\lambda_{1}>0$ that there exists a positive constant $\delta$, such that
\begin{eqnarray}\label{f1}
0\leq\frac{2(n-2)\mu}{(n-1)\alpha}<\frac{2(n-1)\mu}{n\alpha}<\delta<1.
\end{eqnarray}

If we set $h\left(t, x\right)=\frac{1}{\theta(t, x)}$, then multiplying \eqref{b117} by $-\frac{1}{\theta}$, we find that
\begin{eqnarray}\label{f2}
e_{\theta}h_{t}&=&\left(\frac{r^{2n-2}\kappa h_{x}}{v}\right)_{x}-\bigg\{\frac{2r^{2n-2}\kappa h^{2}_{x}}{vh}+\frac{h^{2}}{v}\left((n-1)\delta\alpha-2(n-2)\mu\right)\nonumber\\
&&\times\left(\frac{uv}{r}
+\frac{\left(\alpha\delta-2\mu\right)r^{n-1}u_{x}}{(n-1)\delta\alpha-2(n-2)\mu}\right)^{2}
+\frac{2\mu\left(n\delta\alpha-2(n-1)\mu\right)r^{2n-2}u^{2}_{x}h^{2}}{v\left((n-1)\delta\alpha-2(n-2)\mu\right)}\\
&&+\frac{\alpha\left(1-\delta\right)h^{2}}{v}\left(\left(r^{n-1}u\right)_{x}
-\frac{vP_{\theta}}{2\alpha\left(1-\delta\right)h}\right)^{2}+\lambda\phi zh^{2}\bigg\}+\frac{vP^{2}_{\theta}}{4\alpha\left(1-\delta\right)}.\nonumber
\end{eqnarray}

Due to \eqref{b151} and \eqref{f1}, one can infer from \eqref{f2} that
\begin{eqnarray}\label{f4}
h_{t}&&\leq \frac{1}{e_{\theta}}\left(\frac{r^{2n-2}\kappa\left(v,\theta\right)h_{x}}{v}\right)_{x}
+\frac{vP^{2}_{\theta}}{4\alpha\left(1-\delta\right)e_\theta}\nonumber\\
&&\leq\frac{1}{e_{\theta}}\left(\frac{r^{2n-2}\kappa\left(v,\theta\right)h_{x}}{v}\right)_{x} +C\left(1+\theta^{3}\right)\\
&&\leq\frac{1}{e_{\theta}}\left(\frac{r^{2n-2}\kappa\left(v,\theta\right)h_{x}}{v}\right)_{x} +C.\nonumber
\end{eqnarray}

Setting
\begin{eqnarray*}
g(t, x)&=&\frac{1}{\theta(t,x)}-1,
  \end{eqnarray*}
then \eqref{f4} becomes
  \begin{eqnarray}\label{f5}
 g_{t}&&\leq\frac{1}{e_{\theta}}\left(\frac{r^{2n-2}\kappa\left(v,\theta\right)g_{x}}{v}\right)_{x} +C,
  \end{eqnarray}
and if we define $g_{+}=\max\left\{\frac{1}{\theta(t,x)}-1, 0\right\}$, then one can get by multiplying (\ref{f5}) by $\left(g_{+}\right)^{2p-1}$ and by integrating the result with respect to $x$ over $[0,\infty)$ that
\begin{eqnarray}\label{f6}
&& \left\|g_{+}\right\|_{L^{2p}\left([0,\infty)\right)}^{2p-1}\left(\|g_{+}\|_{L^{2p}\left([0,\infty)\right)}\right)_{t}\nonumber\\
&\leq& C\int_{0}^{\infty}\left(g_{+}\right)^{2p-1}dx
 +C\int_{0}^{\infty}\frac{\left(g_{+}\right)^{2p-1}}{e_{\theta}}\left(\frac{r^{2n-2}\kappa\left(v,\theta\right)g_{x}}{v}\right)_{x}dx.
  \end{eqnarray}

Having obtained \eqref{f6} and noticing
\begin{eqnarray*}
\left(g_{+}\right)^{2p-1}g_{x}=\left(g_{+}\right)^{2p-1}\frac{1+\text{sgn}\left(g\right)}{2}g_{x}=\left(g_{+}\right)^{2p-1}\left(g_{+}\right)_{x},
\end{eqnarray*}
where
\begin{eqnarray*}
\text{sgn}\left(g\right)=
\begin{cases}
 1, &  \text{when}\quad g\left(t,x\right)>0,\\
 0, &  \text{when}\quad g\left(t,x\right)=0,\\
-1, & \text{when}\quad g\left(t,x\right)<0,
 \end{cases}
  \end{eqnarray*}
we can get that
\begin{eqnarray}\label{f7}
&&\frac{\left(g_{+}\right)^{2p-1}}{e_{\theta}}\left(\frac{r^{2n-2}\kappa\left(v,\theta\right)g_{x}}{v}\right)_{x}\nonumber\\
&=&\left(\frac{r^{2n-2}\kappa\left(v,\theta\right)\left(g_{+}\right)^{2p-1}g_{x}}{ve_{\theta}}\right)_{x}
+\frac{r^{2n-2}\kappa\left(v,\theta\right)\left(g_{+}\right)^{2p-1}g_{x}\left(e_{\theta}\right)_{x}}{v e_{\theta}^{2}}\nonumber\\
&&-\frac{\left(2p-1\right)r^{2n-2}\kappa\left(v,\theta\right)\left(g_{+}\right)^{2p-2}g_{x}\left(g_{+}\right)_{x}}{ve_{\theta}}\\
&=&\left(\frac{r^{2n-2}\kappa\left(v,\theta\right)\left(g_{+}\right)^{2p-1}g_{x}}{ve_{\theta}}\right)_{x}
+\frac{r^{2n-2}\kappa\left(v,\theta\right)\left(g_{+}\right)^{2p-1}\left(g_{+}\right)_{x}\left(e_{\theta}\right)_{x}}{v e_{\theta}^{2}}\nonumber\\
&&-\frac{\left(2p-1\right)r^{2n-2}\kappa\left(v,\theta\right)\left(g_{+}\right)^{2p-2}\left[\left(g_{+}\right)_{x}\right]^{2}}{ve_{\theta}}.\nonumber
\end{eqnarray}

On the other hand, it is easy to see that
\begin{eqnarray}\label{f8}
&&\frac{r^{2n-2}\kappa\left(v,\theta\right)\left(g_{+}\right)^{2p-1}\left(g_{+}\right)_{x}\left(e_{\theta}\right)_{x}}{v e_{\theta}^{2}}\nonumber\\
&\leq&\frac{\left(2p-1\right)r^{2n-2}\kappa\left(v,\theta\right)\left(g_{+}\right)^{2p-2}\left[\left(g_{+}\right)_{x}\right]^{2}}{ve_{\theta}}\\
&&+\frac{C}{2p-1}\frac{r^{2n-2}\kappa\left(v,\theta\right)\left(g_{+}\right)^{2p}\left(\left(e_{\theta}\right)_{x}\right)^{2}}{ve_{\theta}^{3}}.\nonumber
  \end{eqnarray}

Combining \eqref{f6}-\eqref{f8} and by making use of H\"{o}lder's inequality yield
\begin{eqnarray}
&&\left\|g_{+}(t)\right\|_{L^{2p}\left([0,\infty)\right)}^{2p-1} \left(\|g_{+}(t)\|_{L^{2p}\left([0,\infty)\right)}\right)_{t}\nonumber\\
&\leq&
 C\left\|g_{+}(t)\right\|_{L^{\frac{2p}{3}}\left([0,\infty)\right)}^{\frac{1}{2}} \left\|g_{+}(t)\right\|^{\frac{4p-3}{2}}_{L^{2p}\left([0,\infty)\right)}\nonumber\\
&&+\frac{C}{2p-1}\left\|g_{+}(t)\right\|_{L^{2p}\left([0,\infty)\right)}^{2p-1}
\left\|\frac{r^{2n-2}\kappa\left(v,\theta\right)g_{+}\left(\left(e_{\theta}\right)_{x}\right)^{2}}
{ve_{\theta}^{3}}\right\|_{L^{2p}\left([0,\infty)\right)}.\nonumber
  \end{eqnarray}
That is,
\begin{eqnarray}\label{f9}
\left(\|g_{+}(t)\|_{L^{2p}\left([0,\infty)\right)}\right)_{t}
\leq C\frac{\|g_{+}(t)\|_{L^{\frac{2p}{3}}\left([0,\infty)\right)}^{\frac{1}{2}}} {\|g_{+}(t)\|_{L^{2p}\left([0,\infty)\right)}^{\frac{1}{2}}}
+\frac{C}{2p-1}\left\|\frac{r^{2n-2}\kappa\left(v,\theta\right)g_{+}\left(\left(e_{\theta}\right)_{x}\right)^{2}}
{ve_{\theta}^{3}}\right\|_{L^{2p}\left([0,\infty)\right)}.
  \end{eqnarray}

Integrating \eqref{f9} with respect to $t$ over $\left(s,t\right)$, we arrive at
\begin{eqnarray}\label{f12}
\|g_{+}\left(t\right)\|_{L^{2p}\left([0,\infty)\right)}&\leq&\|g_{+}\left(s\right)\|_{L^{2p} \left([0,\infty)\right)}+
 C\int_{s}^{t}\frac{\|g_{+}(\tau)\|_{L^{\frac{2p}{3}}\left([0,\infty)\right)}^{\frac{1}{2}}} {\|g_{+}(\tau)\|_{L^{2p}\left([0,\infty)\right)}^{\frac{1}{2}}}d\tau\nonumber\\
&&+\frac{C}{2p-1}\int_{s}^{t}\left\|\frac{r^{2n-2}\kappa\left(v,\theta\right)g_{+}\left(\left(e_{\theta}\right)_{x}\right)^{2}}
{ve_{\theta}^{3}}\right\|_{L^{2p}\left([0,\infty)\right)}d\tau.
  \end{eqnarray}

 Moreover, one can conclude from \eqref{b43} and \eqref{f3} that
\begin{eqnarray}\label{f15}
\left(\left(e_{\theta}\right)_{x}\right)^{2}&=&\left(4a\theta^{3}v_{x}+12av\theta^{2}\theta_{x}\right)^{2}\nonumber\\
&\leq& C\left(\theta^{6}v_{x}^{2}+\theta^{4}\theta_{x}^{2}\right).
  \end{eqnarray}
If we define $\Omega_{1}(t)=\{x\in[0,\infty)|0<\theta\left(t,x\right)<1\}$, then it follows from \eqref{bb13}, \eqref{b43}, \eqref{b151}, \eqref{bb160}, \eqref{bd173} and \eqref{f15} that
\begin{eqnarray}\label{f16}
&&\int_{s}^{t}\left\|\frac{r^{2n-2}\kappa\left(v,\theta\right)g_{+} \left(\left(e_{\theta}\right)_{x}\right)^{2}}
{ve_{\theta}^{3}}\right\|_{L^{2p}\left([0,\infty)\right)}d\tau\nonumber\\
&=&
\int_{s}^{t}\left\|\frac{r^{2n-2}\kappa\left(v,\theta\right)g\left(\left(e_{\theta}\right)_{x}\right)^{2}}
{ve_{\theta}^{3}}\right\|_{L^{2p}\left(\Omega_{1}(\tau)\right)}d\tau\nonumber\\
&\leq&
C\int_{s}^{t}\left\|\frac{r^{2n-2}\left(1+\theta^{b}\right)\left(1-\theta\right) \left(\theta^{6}v_{x}^{2}+\theta^{4}\theta_{x}^{2}\right)}
{\theta}\right\|_{L^{2p}\left(\Omega_{1}(\tau)\right)}d\tau\\
&\leq&C\int_{s}^{t}\left\|r^{2n-2}\left(1-\theta\right)\left(v_{x}^{2} +\theta_{x}^{2}\right)\right\|_{L^{2p}\left(\Omega_{1}(\tau)\right)}d\tau\nonumber\\
&\leq&C\int_{0}^{t}\left(\left\|r^{n-1}v_{x}(\tau)\right\|^2_{L^{\infty}\left([0,\infty)\right)}
+\left\|r^{n-1}\theta_{x}(\tau)\right\|^2_{L^{\infty}\left([0,\infty)\right)}\right) \left(\int_{0}^{\infty}\left(1-\theta(\tau,x)\right)^{2p}dx\right)^{\frac{1}{2p}}d\tau\nonumber\\
&\leq& C.\nonumber
  \end{eqnarray}

Hence \eqref{f12} and \eqref{f16} give birth to
\begin{eqnarray}\label{f17}
\|g_{+}\left(t\right)\|_{L^{2p}\left([0,\infty)\right)} \leq\|g_{+}\left(s\right)\|_{L^{2p}\left([0,\infty)\right)}+
 C\int_{s}^{t}\frac{\|g_{+}(\tau)\|_{L^{\frac{2p}{3}}\left([0,\infty)\right)}^{\frac{1}{2}}} {\|g_{+}(\tau)\|_{L^{2p}\left([0,\infty)\right)}^{\frac{1}{2}}}d\tau+\frac{C}{2p-1}.
  \end{eqnarray}

Furthermore, we can deduce from the fact $g_{+}\left(t,x\right)\in L^{\infty}\left([0,\infty)\right)\cap L^{2}\left([0,\infty)\right)$ and is sufficiently smooth that
 \begin{eqnarray}\label{f18}
\lim_{p\rightarrow+\infty}\|g_{+}\left(t\right)\|_{L^{p}\left([0,\infty)\right)} =\|g_{+}\left(t\right)\|_{L^{\infty}\left([0,\infty)\right)}.
  \end{eqnarray}

Letting $p\rightarrow+\infty$ in \eqref{f17} and taking advantage of \eqref{f18}, we have
 \begin{eqnarray}\label{f19}
\|g_{+}\left(t, \cdot\right)\|_{L^{\infty}\left([0,\infty)\right)}\leq \|g_{+}\left(s, \cdot\right)\|_{L^{\infty}\left([0,\infty)\right)}+C\left(t-s\right).
  \end{eqnarray}
Then \eqref{b180} follows from \eqref{f19} and the definition of $g_{+}\left(t, x\right)$ immediately. This completes the proof of Lemma 5.1.

\end{proof}

With Lemma 2.1-Lemma 5.1 in hand,  we can deduce Theorem 1.1 by the continuation argument designed in Liao and Zhao \cite{Liao-Zhao-arXive-2017} and we thus omit the details for brevity.

\begin{center}
{\bf Acknowledgement}
\end{center}
Research of the authors was supported by the Fundamental Research Funds for the Central Universities, the Project funded by China Postdoctoral Science Foundation, and three grants from National Natural Science Foundation of China under contracts 11601398,  11671309, and 11731008, respectively.

\end{document}